\documentclass[a4paper,12pt]{article}

\usepackage[utf8]{inputenc}
\usepackage{scrpage2}
\usepackage[backend=biber,style= numeric-comp,bibstyle=numeric,giveninits=true]{biblatex}

\renewbibmacro{in:}{%
  \ifentrytype{article}{}{\printtext{\bibstring{in}\intitlepunct}}}

\usepackage{amsmath}
\usepackage{amsfonts,amssymb,amsthm}
\usepackage{mathtools}
\usepackage{enumerate}
\usepackage{nicefrac}
\usepackage{graphicx}
\usepackage{bbm}
\usepackage{esint}
\usepackage{caption}
\usepackage{stmaryrd}
\usepackage[labelformat = brace, position=top]{subcaption}
\usepackage{mathrsfs}%für mathscr
\usepackage{textcomp}
\usepackage{gensymb} %degree symbol
 \usepackage{hyperref}

\usepackage{pgf,tikz}
\usepackage{pgfplots}
\pgfplotsset{compat=1.11}
\usetikzlibrary{calc,intersections,through,backgrounds,patterns,arrows.meta} %necessary for calculations of intersections
\usepackage{tikz-3dplot}
\usepackage[usestackEOL]{stackengine}
\usepackage{emptypage}

\newcommand\numberthis{\addtocounter{equation}{1}\tag{\theequation}} % Nummerierung
\newcounter{statement} \numberwithin{statement}{section}
\newtheorem{Def}[statement]{Definition}

\newtheorem{thm}[statement]{Theorem}
\newtheorem{lemma}[statement]{Lemma}
\newtheorem{cor}[statement]{Corollary}
\newtheorem{prop}[statement]{Proposition}

\theoremstyle{remark}
\newtheorem{remark}[statement]{Remark}

\newcommand{\R}{\mathbb R} 				%real numbers
\newcommand{\N}{\mathbb N}
\newcommand{\ball}[2]{{B_{#2}\left(#1\right)}}				%ball of radius #2 around #1
\newcommand{\intd}{\, \mathrm{d}} 				%schönes d für Integrale
\newcommand{\warr}{\rightharpoonup}				%Pfeil für schwache Konvergenz
\newcommand{\Leb}{\mathcal L}				%Lebesgue measure
\newcommand{\Sph}{{\mathbb S}}				%Sphere in 2D
\newcommand{\Hd}{\mathcal H}				%Hausdorff measures

\newcommand{\supp}{\operatorname{supp}}
\newcommand{\dist}{\operatorname{dist}}
\newcommand{\osc}{\operatorname{osc}}
\newcommand{\tr}{\operatorname{tr}}
\newcommand{\overbar}[1]{\mkern 1.5mu\overline{\mkern-1.5mu#1\mkern-1.5mu}\mkern 1.5mu} %Bar for closures
\newcommand{\stcomp}[1]{{#1}^{\mathsf{c}}}		%complements

\newcommand{\K}{\mathcal{K}}				%Differential Inclusion Set
\newcommand{\eps}{\varepsilon}
\newcommand{\interior}[1]{{\kern0pt#1}^{\mathrm{o}}	%Definition of interior set
}
\newcommand{\translatepoint}[1]%			For translation of coordinate system in 3D tikzpictures
{   \coordinate (middlepoint) at (#1);
}

    \newcommand{\epsi}{\varepsilon}
    
    \DeclareMathOperator*{\esssup}{ess\,sup}
    \DeclareMathOperator*{\essinf}{ess\,inf}

\def\Xint#1{\mathchoice
{\XXint\displaystyle\textstyle{#1}}%
{\XXint\textstyle\scriptstyle{#1}}%
{\XXint\scriptstyle\scriptscriptstyle{#1}}%
{\XXint\scriptscriptstyle\scriptscriptstyle{#1}}%
\!\int}
\def\XXint#1#2#3{{\setbox0=\hbox{$#1{#2#3}{\int}$ }
\vcenter{\hbox{$#2#3$ }}\kern-.58\wd0}}

\def\dashint{\Xint-}

\bibliography{SMA.bib}

\title{Rigidity of branching microstructures in shape memory alloys}
\author{Thilo M. Simon\footnote{Research carried out at: Max-Planck-Institut f\"ur Mathematik in den Naturwissenschaften, Inselstra{\ss}e 22, 04103 Leipzig, Germany. Now at:  New Jersey Institute of Technology, University Heights Newark, New Jersey 07102, USA.  Please use {thilo.m.simon@njit.edu} for correspondence.}}

\bibliography{SMA.bib}

\begin{document}
\maketitle
\begin{abstract}
 We analyze generic sequences for which the geometrically linear energy
 \[E_\eta(u,\chi):= \eta^{-\frac{2}{3}}\int_\ball{0}{1} \left| e(u)- \sum_{i=1}^3 \chi_ie_i\right|^2 \intd x+\eta^\frac{1}{3} \sum_{i=1}^3 |D\chi_i|(\ball{0}{1})\]
remains bounded in the limit $\eta \to 0$.
Here $ e(u) :=1/2(Du + Du^T)$ is the (linearized) strain of the displacement $u$, the strains $e_i$ correspond to the martensite strains of a shape memory alloy undergoing cubic-to-tetragonal transformations and $\chi_i:\ball{0}{1} \to \{0,1\}$ is the partition into phases.
In this regime it is known that in addition to simple laminates also branched structures are possible, which if austenite was present would enable the alloy to form habit planes.

In an ansatz-free manner we prove that the alignment of macroscopic interfaces between martensite twins is as predicted by well-known rank-one conditions.
Our proof proceeds via the non-convex, non-discrete-valued differential inclusion
\[e(u)  \in \bigcup_{1\leq i\neq j\leq 3} \operatorname{conv} \{e_i,e_j\}\]
satisfied by the weak limits of bounded energy sequences and of which we classify all solutions.
In particular, there exist no convex integration solutions of the inclusion with complicated geometric structures.
% The key ingredient is that ``discontinuity'' of $e(u)$ and the differential inclusion complement each other.

    \medskip


    \medskip

\end{abstract}

\section{Introduction}
Due to the many possible applications of the eponymous shape memory effect, shape memory alloys have attracted a lot of attention of the engineering, materials science and mathematical communities.
Their remarkable properties are due to certain diffusionless solid-solid phase transitions in the crystal lattice of the alloy, enabling the material to form microstructures.
More specifically, the lattice transitions between the cubic \emph{austenite} phase and multiple lower-symmetry \emph{martensite} phases, triggered by crossing a critical temperature or applying stresses.
% 
% \begin{figure}
%  \centering
%  \includegraphics[height= 3cm]{tetra2.png}
%  \caption{A sketch of the cubic-to-tetragonal transformation. The left hand side represents austenite, while the right hand side represents the martensite variants.} \label{fig:cubic-to-tetragonal}
% %  \vspace{-10pt}
% \end{figure}

\begin{figure}
 \centering
 \begin{tikzpicture}
 	\draw (0,0,0) -- (1,0,0) -- (1,1,0) -- (0,1,0) -- cycle;
 	\draw (0,1,0) -- (0,1,-1) -- (1,1,-1) -- (1,1,0);
 	\draw (1,1,-1) -- (1,0,-1) -- (1,0,0);
 	
	\draw[->]  (1.75, 1,0) -- (3.5, 2.2); 	
 	
 	\begin{scope}[shift={(4,2,0)}]
 		\draw (0,0,0) -- (1.25,0,0) -- (1.25,.75,0) -- (0,.75,0) -- cycle;
	 	\draw (0,.75,0) -- (0,.75,-.75) -- (1.25,.75,-.75) -- (1.25,.75,0);
 		\draw (1.25,.75,-.75) -- (1.25,0,-.75) -- (1.25,0,0);
 	\end{scope}
 	
 	\draw[->]  (1.75, .75,0) -- (3.5, .75,0);
 	
 	\begin{scope}[shift={(4,0,0)}]
 	 		\draw (0,0,0) -- (.75,0,0) -- (.75,1.25,0) -- (0,1.25,0) -- cycle;
	 	\draw (0,1.25,0) -- (0,1.25,-.75) -- (.75,1.25,-.75) -- (.75,1.25,0);
 		\draw (.75,1.25,-.75) -- (.75,0,-.75) -- (.75,0,0);
 	\end{scope}
 	
 	\draw[->]  (1.75, .5,0) -- (3.5, -.75,0);
 	
 	\begin{scope}[shift={(4,-1.5,0)}]
 		\draw (0,0,0) -- (.75,0,0) -- (.75,.75,0) -- (0,.75,0) -- cycle;
	 	\draw (0,.75,0) -- (0,.75,-1.25) -- (.75,.75,-1.25) -- (.75,.75,0);
 		\draw (.75,.75,-1.25) -- (.75,0,-1.25) -- (.75,0,0);
 	\end{scope}
 \end{tikzpicture}
 \caption{A sketch of the cubic-to-tetragonal transformation. The left-hand side represents the cubic austenite phase, while the right-hand side represents the martensite variants that are elongated in the direction of one of the axes of the cube and shortened in the other two. Adapted from \cite[Figure 4.5]{Bhattacharya_microstructure}.} \label{fig:cubic-to-tetragonal}
%  \vspace{-10pt}
\end{figure}
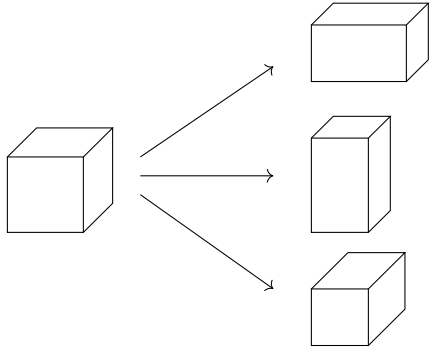

In shape memory alloys undergoing cubic-to-tetragonal transformations, see \ref{fig:cubic-to-tetragonal}, one frequently observes the following types of microstructures:
  \begin{enumerate}
   \item \emph{Twins:} Fine-scale laminates of martensite variants, see Figure \ref{fig:micro_a}
   and both sides of the interface at the center of Figure \ref{fig:micro_b}. 
   \item \emph{Habit planes:} Almost sharp interfaces between austenite, and a twin of martensites, where the twin refines as it approaches the interface, see Figure \ref{fig:micro_a}.
   \item \emph{Second-order laminates, or twins within a twin:} Essentially sharp interfaces between two different refining twins, see Figure \ref{fig:micro_b}.
   \item \emph{Crossing second-order laminates:} Two crossing interfaces between twins and pure phases, see for example \cite[Figure 17]{basinski1954experiments}.
%    Figure \ref{fig:crossing}.
   \item \emph{Wedges:} Materials whose lattice parameters satisfy a certain relation can form a wedge of two martensite twins in austenite, see \cite[Chapter 7.3.1]{Bhattacharya_microstructure} and Figure \ref{fig:wedge}.
  \end{enumerate}
Furthermore, at least in Microstructures 1, 2 and 5, all observed interfaces form parallel to finitely many different hyperplanes relative to the crystal orientation.

  \begin{figure}[htb]
    \begin{subfigure}[t]{0.475\linewidth}
     \centering
     \includegraphics[height=4cm]{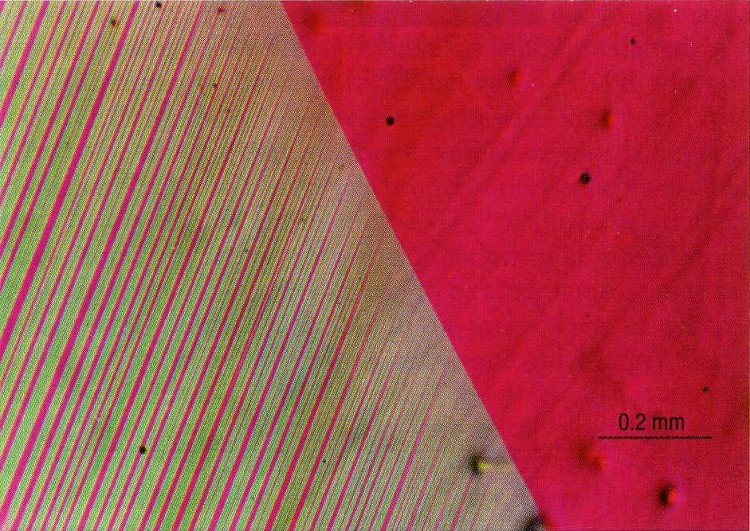}
     \caption{}
     \label{fig:micro_a}
    \end{subfigure}    
    \begin{subfigure}[t]{0.475\linewidth}
     \centering
     \includegraphics[height=4cm]{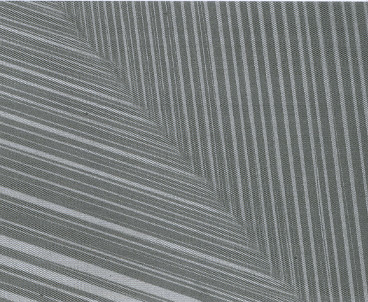}
      \caption{}
      \label{fig:micro_b}
    \end{subfigure}
    
%    \begin{subfigure}[t]{0.475\linewidth}
%     \centering
%     \includegraphics[height=4cm]{crossing_comp.jpeg}
%     \caption{}
%     \label{fig:crossing}
%    \end{subfigure}
	\centering 
    \begin{subfigure}[t]{0.475\linewidth}
     \centering
     \includegraphics[height=4cm]{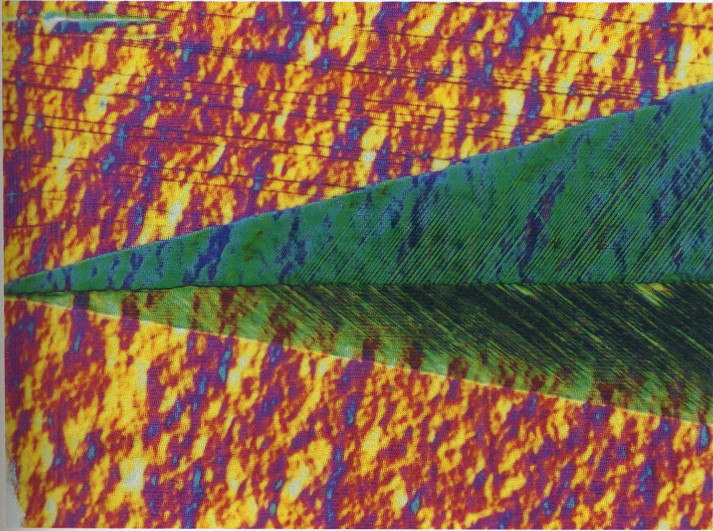}
     \caption{}
     \label{fig:wedge}
    \end{subfigure}
    \caption{a) Optical micrograph of a habit plane with austenite on the right-hand side and twinned martensite on the left-hand side in a Cu-Al-Ni alloy undergoing cubic-to-orthorhombic transformations,
%      \cite[Fig.\ 9.1]{planes2005magnetism}
     by courtesy of C.\ Chu and R.D.\ James.
     b) Optical micrograph of a second-order laminate in a Cu-Al-Ni alloy,
%       \cite[Fig.\ 7.9]{Bhattacharya_microstructure}
     by courtesy of C.\ Chu and R.D.\ James.
%     c) Optical micrograph of two crossing second-order laminates in an Indium-Thallium crystal. The bottom region is in the austenite phase. All other regions show twinned martensite variants with the twinning in the left-hand side one being almost parallel to the surface of the sample. Reprinted from \cite{basinski1954experiments}, with permission from Elsevier.
     d) Optical micrograph of a wedge in a Cu-Al-Ni alloy,
%      \cite[Plate 1]{Bhattacharya_microstructure}.
      by courtesy of C.\ Chu and R.D.\ James.}
    \label{fig:twins_and_interfaces}
  \end{figure}

%   \begin{figure}[htb]
%     \begin{subfigure}[t]{0.475\linewidth}
%      \centering
%      \includegraphics[height=4cm]{branching.jpg}
%      \caption{Optical micrograph of a habit plane with austenite on the right-hand side and twinned martensite on the left-hand side in a Cu-Al-Ni alloy undergoing cubic-to-orthorombic transformations, \cite[Fig.\ 9.1]{planes2005magnetism}.}\label{fig:micro_a}
%     \end{subfigure}    
%     \begin{subfigure}[t]{0.475\linewidth}
%      \centering
%      \includegraphics[height=4cm]{twins_in_twins_crop.jpg}
%      \caption{Optical micrograph of a second-order laminate in a Cu-Al-Ni alloy, \cite[Fig.\ 7.9]{Bhattacharya_microstructure}.} \label{fig:micro_b}
%     \end{subfigure}
%     \begin{subfigure}{0.475\linewidth}
%      \centering
%      \includegraphics[height=4cm]{crossing.jpeg}
%      \caption{Optical micrograph of two crossing second-order laminates in an Indium-Thallium crystal, \cite[Fig.\ 17]{basinski1954experiments}. The bottom region is in the austenite phase. All other regions show twinned martensite variants with the twinning in the left-hand side one being almost parallel to the surface of the sample.}\label{fig:crossing}
%     \end{subfigure}\quad\quad
%     \begin{subfigure}[t]{0.475\linewidth}
%      \centering
%      \includegraphics[height=4cm]{wedge.pdf}
%      \caption{Optical micrograph of a wedge in a Cu-Al-Ni alloy, \cite[Plate 1]{Bhattacharya_microstructure}.} \label{fig:wedge}
%     \end{subfigure}
%     \caption{\label{fig:twins_and_interfaces}}
%   \end{figure}
 
\subsection{Contributions of the mathematical community}\label{subsec:SMA_literature}
\subsubsection{Modeling}
 
The first use of energy minimization in the modeling of martensitic phase transformations has been made by Khatchaturyan, Roitburd and Shatalov \cite{khachaturyan1967some, khachaturyan1983theory, khachaturyan1969theory, roitburd1969domain, roitburd1978martensitic} on the basis of linearized elasticity.
This allowed to predict certain large scale features of the microstructure such as the orientation of interfaces between phases.

Variational models based on nonlinear elasticity go back to Ball and James \cite{BallJames87minimizers, BallJames92Youngmeasure}.
They formulated a model in which the microstructures correspond to minimizing sequences of energy functionals vanishing on \[K = \bigcup_{i} SO(3)U_i\]
for finitely many suitable symmetric matrices $U_i$.
In their theory, the orientation of interfaces arise from a kinematic compatibility condition known as rank-one connectedness, see \cite[Chapter 2.5]{Bhattacharya_microstructure}.
For cubic-to-tetragonal transformations Ball and James prove in an ansatz-free way that the fineness of the martensite twins in a habit plane is due only certain mixtures of martensite variants being compatible with austenite.
Their approach is closely related to the phenomenological (or crystallographic) theory of martensite independently introduced by Wechsler, Lieberman and Read \cite{WechslerLiebermanRead} and Bowles and MacKenzie \cite{bowles1954crystallography, mackenzie1954crystallography}.
In fact, the variational model can be used to deduce the phenomenological theory. 

A comparison of the nonlinear and the geometrically linear theories can be found in an article by Bhattacharya \cite{Bhattacharya93comparison}.
Formal derivations of the geometrically linear theory from the nonlinear one have been given by Kohn \cite{Kohn1991relaxation} and Ball and James \cite{BallJames92Youngmeasure}.
A rigorous derivation via $\Gamma$-convergence has been given by Schmidt \cite{schmidt2008linear} with the limiting energy in general taking a more complicated form than the usually used piecewise quadratic energy densities.

\subsubsection{Rigidity of differential inclusions}
The interpretation of microstructure as minimizing sequences naturally leads to analyzing the differential inclusions
\[Du \in K = \bigcup_{i=1}^m SO(3)U_i,\]
sometimes called the $m$-well problem, or variants thereof such as looking for sequences $u_k$ such that $\dist(Du_k,K) \to 0$ in measure.
In fact, the statements of Ball and James are phrased in this way \cite{BallJames87minimizers, BallJames92Youngmeasure}.
A detailed discussion of these problems which includes the theory of Young measures has been provided by M\"uller \cite{muller_microstructure}.

%As in the case of cubic-to-tetragonal transformations the martensite variants are pairwise compatible, the material can form the fine twins discussed above.
%% shown in Figure \ref{fig:twins_and_interfaces}. 
%In addition, Ball and James prove, already in an ansatz-free way, that the fineness of the martensite twins in a habit plane is due to martensite variants only being compatible with austenite on average.
%Their ideas turned out to be very fruitful for analyzing related microstructures also in other types of transformations, for an overview see \cite[Chapters 4 and 7]{Bhattacharya_microstructure}.

However, differential inclusions in themselves are not accurate models:
M\"uller and \v{S}ver{\'a}k \cite{MS99convexintegration} constructed solutions with a complex arrangement of phases of the differential inclusion $Du \in SO(2) A \cup SO(2) B$ with $\operatorname{det}(A) = \operatorname{det} B = 1$, for which one would naively only expect laminar solutions, in two space dimensions using convex integration.
Later, Conti, Dolzmann and Kirchheim \cite{CDK07existence} extended their result to three dimensions and the case of cubic-to-tetragonal transformations.

But Dolzmann and M\"uller \cite{DM95} also noted that if the inclusion $Du \in SO(2) A \cup SO(2) B$ is augmented with the information that the set $\{Du \in SO(2)A\}$ has finite perimeter, then $Du$ is in fact laminar.
Also this result holds in the case of cubic-to-tetragonal transformations as shown by Kirchheim \cite{Kirchheim98rigidity}.
There has been a series of generalizations including stresses \cite{lorent2005two, conti2006rigidity, LorentLp}, culminating in the papers by Conti and Chermisi \cite{chermisi2010multiwell} and Jerrard and Lorent \cite{jerrard2013multiwell}.
However, these are more in the spirit of the geometric rigidity theorem due to Friesecke, James and M\"uller \cite{friesecke2002theorem} since they rely on the perimeter being too small for lamination and as such do not give insight into the rigidity of twins.

In contrast, the differential inclusion arising from the geometrically linear setting
\[\frac{1}{2}(Du + Du^T) \in \{e_1,e_2,e_3\},\]
where $e_i$ for $i=1,2,3$ are the linearized strains corresponding to the cubic-to-tetragonal transformation, see \eqref{strains}, is rigid in the sense that all solutions are laminates even without further regularizations as proven by Dolzmann and M\"uller \cite{DM95}.
Quantifying this result Capella and Otto \cite{CO12,CO09} proved that laminates are stable in the sense that if the energy \eqref{ENERGIE!} (including an interfacial penalization) is small then the geometric structure of the configuration is close to a laminate.
Additionally, there is either only austenite or only mixtures of martensite present.
Capella and Otto also noted that for sequences with bounded energy such a result cannot hold due to a well-known branching construction of habit planes (Figure \ref{fig:micro_a}) given by Kohn and M\"uller \cite{KM92, KM94branching}.

Therein, Kohn and M\"uller used a simplified scalar version of the geometrically linear model with surface energy to demonstrate that compatibility of austenite with a mixture of martensites only requires a fine mixture close to the interface so that the interfacial energy coarsens the twins away from the interface.
Kohn and Müller also conjectured that the minimizers exhibit this so-called branching, which Conti \cite{Conti00branched} affirmatively answered by proving minimizers of the Kohn-M\"uller functional to be asymptotically self-similar.

In view of the results by Kohn and M\"uller, and Capella and Otto it is natural to consider sequences with bounded energy in order to analyze the rigidity of branching microstructures.

\subsubsection{Some related problems}
So far, we mostly discussed the literature describing the microstructure of single crystals undergoing cubic-to-tetragonal transformations.
However, the variational framework can be used to address related problems, for which we highlight a few contributions as an exhaustive overview is outside the scope of this introduction:

An overview of microstructures arising in other transformations can be found in the book by Bhattacharya \cite{Bhattacharya_microstructure}.
Rigorous results for cubic-to-orthorhombic transformations in the geometrically linear theory can be found in a number of works by R\"uland \cite{ruland2016cubic, ruland2016rigidity}.
For the much more complicated cubic-to-monoclinic-I transformations with its twelve martensite variants, Chenchiah and Schl\"omerkemper \cite{chenchiah2013non} proved the existence of certain non-laminate microstructures in the geometrically linear case without surface energy.

For an overview over the available literature on polycrystalline shape memory alloys we refer the reader once again to Bhattacharya's book \cite[Chapter 13]{Bhattacharya_microstructure} and an article by Bhattacharya and Kohn \cite{bhattacharya1997elastic}.

Another problem is determining the shape of energy-minimizing inclusions of martensite with given volume in a matrix of austenite, for which scaling laws have been obtained by Kohn, Kn\"upfer and Otto \cite{KKO13} for cubic-to-tetragonal transformations in the geometrically linear setting.
% Here, we point out some examples of works which 
% 
% although there are a number of works in the area, see for example \cite{ball1991dynamics, roubivcek2004models, ball2015quasistatic, blesgen2014allen, dondl2010sharp, knupfer2016sharp}.

\subsection{Definition of the energy}\label{subsec:SMA_model}
In order to analyze the rigidity properties of branched microstructures we choose the geometrically linear setting, since the quantitative rigidity of twins is well understood due to the results by Capella and Otto \cite{CO09, CO12}.
In fact, we continue to work with the same  already non-dimensionalized functional, namely
\begin{align}
  E_\eta(u,\chi)&:= E_{elast}(u,\chi)+ E_{inter,\eta}(u,\chi),\numberthis \label{ENERGIE!}\\
  \intertext{where}
  E_{elast,\eta}(u,\chi)&:=\eta^{-\frac{2}{3}}\int_\Omega \left| e(u) -\sum_{i=1}^3\chi_i e_i\right|^2 \intd \Leb^3,\\
  E_{inter,\eta}(u,\chi)&:=\eta^{\frac{1}{3}} \sum_{i=1}^3|D \chi_i|(\Omega).
\end{align}
Here $\Omega\subset \R^3$ is a bounded Lipschitz domain, $u:\Omega \to \mathbb{R}^3$ is the displacement and $e(u) =\frac{1}{2}\left(Du + Du^T\right)$ denotes the strain.
Furthermore, the partition into the phases is given by $\chi_i:\Omega \to \{1,1\}$ for $i=1,\ldots,3$ with $ \sum_{i=1}^3\chi_i = 1$ and the strains associated to the phases are given by
\[
    e_0 := 0,
    e_1:=\begin{pmatrix}
             -2 & 0 & 0 \\
             0 & 1 & 0\\
             0 & 0 & 1\\
	  \end{pmatrix},
    e_2:=\begin{pmatrix}
             1 & 0 & 0 \\
             0 & -2 & 0\\
             0 & 0 & 1\\
	\end{pmatrix},         
    e_3:=\begin{pmatrix}
             1 & 0 & 0 \\
             0 & 1 & 0\\
             0 & 0 & -2\\
	\end{pmatrix}.\numberthis \label{strains}
\]
In particular, we assume the reference configuration to be in the austenite state, but that the transformation has occured throughout the sample, i.e., there is no austenite present.
This simplifying assumption does rule out habit planes, see Figure \ref{fig:micro_a}, but a look at Figure \ref{fig:micro_b} suggests that we can still hope for an interesting result.
Furthermore, the responsible mechanism for macroscopic rigidity is the rank-one connectedness of the average strains $e(u_\eta) \warr e(u)$ in $L^2$ (encoded in the decomposition provided by Lemma \ref{lemma: decomposition}), which cannot distinguish between pure phases and mixtures.

The condition of the material being a shape memory alloy is encoded in the fact that $\operatorname{tr}(e_i) = 0$ for $i=1,2,3$ as this corresponds to the transformation being volume-preserving.

%However, we will not stress this point outside of Chapter \ref{chapter:macroscopic}, where we will have to assume $\chi_0 \equiv 0$.
Further simplifying choices are using equal isotropic elastic moduli with vanishing second Lam\'e constant and penalizing interfaces by the total variation of $D\chi_i$.
Of course, as such it is unlikely that the model can give quantitatively correct predictions.
Bhattacharya for example argues that assuming equal elastic moduli is not reasonable \cite[Page 238]{Bhattacharya93comparison}.

We still expect our analysis to give relevant insight as we will for the most part prove compactness properties of \emph{generic} sequences  $u_\eta \in W^{1,2}(\Omega;\R^3)$ and partitions $\chi_\eta$ such that
\[\limsup_{\eta \to 0} E_{\eta}(u_\eta,\chi_\eta) < \infty.\]
This regime is the appropriate one to analyze branching microstructures:
On the one hand, (generalizations of) the Kohn-M\"uller branching construction of habit planes have bounded energy.
On the other hand, the stability result of Capella and Otto \cite{CO12} rules out branching by ensuring that in a strong topology there is either almost exclusively austenite or the configuration is close to a laminate.
In other words, the branching construction implies that the stability result is sharp with respect to the energy regime as pointed out by Capella and Otto in their paper.

\subsubsection{Compatibility properties of the stress-free strains}
It is well known, see \cite[Chapter 11.1]{Bhattacharya_microstructure}, that for $A$, $B \in \R^{3\times 3}$ and $n \in \Sph^2$ the following two statements are equivalent:
\begin{itemize}
 \item There exists a continuous function $u:\R^3 \to \R^3$ with
  \[e(u)(x) = \begin{cases}
	    A & \text{ if } x \cdot n > 0,\\
	    B & \text{ if } x \cdot n < 0,
         \end{cases}\numberthis \label{rank-one-connection}
  \]
see Figure \ref{fig:single_interface}.
 \item The two strains are (symmetrically) rank-one connected in the sense that there exists $a \in \R^3$ such that
\[A - B = \frac{1}{2}( a \otimes n + n \otimes a) := a \odot n.\]
\end{itemize} 
Note that the condition is symmetric in $a$ and $n$ thus every rank-one connection generically gives rise to two possible normals.
Additionally, as rank-one connectedness is also symmetric in $A$ and $B$ this allows for the construction of laminates.

 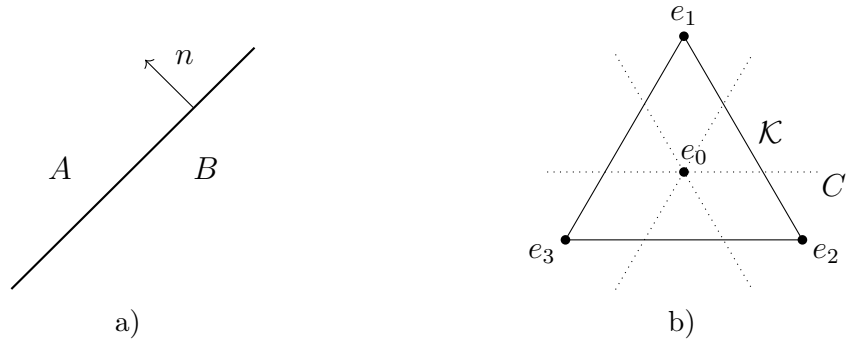
\begin{figure}[htbp]
   \centering
   \subcaptionbox{\label{fig:single_interface}}[.45\linewidth]{
     \centering      
      \begin{tikzpicture}[scale=1.6]
% 	\node at (-2,1) {a)};
	%interfaces
% 	\draw[thick] (-2,-1) -- (0,1);
	\draw[thick] (-1,-1) -- (1,1);
%       \draw[thick] (0,-1) -- (2,1);
	%normal
	\draw[->] (0.5,0.5) -- (0.1,0.9);
	%labels
% 	\node at (-2,0) {$A$};
	\node at (-0.6,0) {$A$};
	\node at (0.6,0) {$B$};
% 	      \node at (1.25,0) {$B$};
	\node at (.425,0.925) {$n$};
      \end{tikzpicture}
   } 
   \subcaptionbox{\label{fig:K_cone}}[.45\linewidth]{
    \centering
      \begin{tikzpicture}[scale=1.8]
% 	  \node at (-1,1) {b)};
	  \fill (0,0) circle (1pt);
	  \node at (60:0.15) {$e_0$};
	  \fill (90:1) circle (1pt);
	  \node at (90:1.15) {$e_1$};
	  \fill (210:1) circle (1pt);
	  \node at (210:1.2) {$e_3$};
	  \fill (330:1) circle (1pt);
	  \node at (330:1.2) {$e_2$};
	  \draw (90:1) -- (210:1);
	  \draw (210:1)-- (330:1);
	  \draw (330:1) -- (90:1);
	  \node at (25:.7) {$\K$};
	  \node at (355:1.1) {$C$};
% 	  \clip (90:1) -- (210:1) -- (330:1) -- (90:1);
	  \draw[dotted] (240:1) -- (60:1);
	  \draw[dotted] (180:1) -- (0:1);
	  \draw[dotted] (120:1) -- (300:1);
      \end{tikzpicture}
  }
  \caption{a) Geometry of an interface parallel to the plane $\{x\cdot n =0\}$ in a laminate joining the strains $A$ and $B$.
  b) Sketch relating the martensite strains with the cone $C$ (dotted) of symmetrized rank-one matrices in the two-dimensional strain space $S$. Note that $C$ is a union of three lines parallel to the edges of the triangle $\K$.}
  \label{fig:rank_one_connections_in_triangle}
 \end{figure}

In order to present the result of applying the rank-one connectedness condition to the case of cubic-to-tetragonal transformations notice that
\[e_0, \ldots e_3 \in S:=\left\{e\in \R^{3\times 3}: e \text{ diagonal, }\tr e = 0\right\}.\]
Here, we call the two-dimensional space $S$ \emph{strain space}.
It can be shown, either by direct computation or an application of \cite[Lemma 3.1]{chenchiah2013non}, that all rank-one directions in $S$ are multiples of $e_2 - e_1$, $e_3 - e_2$ and $e_1 - e_3$.
This means that they are parallel to one of the sides of the equilateral triangle
\[\K := \bigcup_{i=1}^3 \{\lambda e_{i+1} + 1-\lambda e_{i-1}: \lambda \in [0,1]\}\numberthis \label{intro:K}\]
spanned by $e_1, e_2$ and $e_3$ shown in Figure \ref{fig:K_cone}.
In particular, the martensite strains are mutually compatible but austenite is only compatible to certain convex combinations of martensites which turn out to be $\frac{1}{3} e_i + \frac{2}{3} e_j$ for $i,j =1,2,3$ with $i\neq j$.

\subsection{The contributions of the paper}\label{subsec:SMA_contributions}
We study the rigidity of branching microstructures due to ``macroscopic'' effects in the sense that we only look at the limiting volume fractions $\chi_{i,\eta} \overset{\ast}{\warr} \theta_i$ in $L^\infty$ after passage to a subsequence, which completely determines the limiting strain $e(u_\eta) \warr e(u)$ in $L^2$.

Similarly to the result of Capella and Otto \cite{CO12}, our main result, Theorem \ref{main}, is local in the sense that for $\Omega = \ball{0}{1}$ we can classify the function $\theta$ on a smaller ball $\ball{0}{r}$ of universal radius $0<r<1$.
As the characterization of each of the four possible cases is a bit lengthy, we postpone a detailed discussion to Subsection \ref{subsec:limiting_configurations}.
An important point is that we deduce all interfaces between different mixtures of martensites to be hypersurfaces whose normals are as predicted by the rank-one connectedness of the average strains on either side.
In this respect our theorem improves on previously available ones, as they either explicitly assume the correct alignment of a habit plane, see e.g.\ Kohn and M\"uller \cite{KM94branching}, or require other ad-hoc assumptions:
For example, Ball and James \cite[Theorem 3]{BallJames87minimizers} show habit planes to be flat under the condition that the set formed by the austenite phase is taken is topologically well-behaved.

The broad strategy of our proof is to first ensure that in the limit the displacement satisfies the non-convex differential inclusion
\[e(u)  \in \K\]
encoding that locally at most two variants are involved, see Definition \ref{intro:K} and Figure \ref{fig:rank_one_connections_in_triangle}, and then to classify all solutions.
We strongly stress the point that we do not need to assume any additional regularity in order to do so.
In particular, the differential inclusion is rigid in the sense that it does not allow for convex integration solutions with extremely intricate geometric structure.
To our knowledge this is the first instance of a rigidity result for a non-discrete differential inclusion in the framework of linearized elasticity.

The main idea is that ``discontinuity'' of $e(u)$ and the differential inclusion $e(u) \in \K$ balance each other:
If $e(u) \notin VMO$, see Definition \ref{VMO}, a blow-up argument making use of measures describing the distribution of values $e(u)\in \K$, similar in spirit to Young measures, proves that the strain is independent of one direction.
If $e(u) \in VMO$ the differential inclusion gives us less information, but we can still prove that only two martensite variants are involved by using an approximation argument.
Finally, we classify all solutions which are independent of one direction.

The structure of the paper is as follows:
In Section \ref{sec:main_result} we state and discuss our main theorem in detail.
We then proceed to break down its proof into several main steps in Section \ref{sec:outline}, and give an in-depth explanation of all necessary auxiliary results.
Finally, we give the proofs of these results in Section \ref{sec:proof}.

% Also Roitburd-Khatchaturyan-Shatalov!!!
% NO CONVEX INTEGRATION!!!
% We refer to planes by $\plane{\nu}{c}:=\{x\in \R^3: x\cdot \nu = c\}$ for $\nu\in N$, $c\in \R$.
% 
% DIE BEDEUTUNG VON THETA HAT SICH GEAENDERT, von den Eintraegen von e(u) zu den Volumenanteilen!!!
% 
% VORSICHT, RESKALIERUNG LEBESGUE PUNKTE FUNKTIONIEREN NICHT SO WIE DU ES GERADE NOCH IN DEN ANDEREN DOKUMENTEN HAST!!!
% 
% IST PI IRGENDWO ANDERS DEFINIERT?
% 
% g IST SCHLECHTE NOTATION IN DEM POLYEDER LEMMA!

% VORSICHT, b IST SCHLECHTE NOTATION FUER KONSTANTE TEILE IN DEN AFFINEN FUNKTIONEN BEI DREI FUNKTIONEN!!!

\section{The main rigidity theorem}\label{sec:main_result}
Note that any sequence with asymptotically bounded energy has subsequences such that $u_\eta \warr u$ in $W^{1,2}$ and $\chi_\eta \overset{*}{\warr} \theta$ in $L^\infty$.

\begin{thm}\label{main}
  There exists a universal radius $r>0$ such that the following holds:
  Let $(u_\eta$, $\chi_\eta)$ be a sequence of displacements and partitions such that $E_\eta(u_\eta,\chi_\eta) < C$ for some $0 < C < \infty$.
  Then, for any subsequence along which they exist, the weak limits
  \[u_\eta \rightharpoonup u \text{ in } W^{1,2}\text{, } \chi_\eta \stackrel{*}{\rightharpoonup} \theta \text{ in } L^\infty\]
  satisfy
  \begin{align*}
   e(u)  \equiv \sum_{i=1}^3 \theta_i e_i & \text{ and } e(u) \in \K = \bigcup_{i=1}^3 \{\lambda e_{i+1} + (1-\lambda) e_{i-1}: \lambda \in [0,1]\},
  \end{align*}
  see Figure \ref{K}, for almost all $x\in \ball{0}{1}$.
  
  Furthermore, on the smaller ball $\ball{0}{r}$ all solutions to this differential inclusion are two-variant configurations, planar second-order laminates, planar checkerboards or planar triple intersections, according to Definitions \ref{def:degenerate}-\ref{def:triple} below.
\end{thm}

The first part of the conclusion states that the volume fractions $\theta_i$ for $i=1,2,3$ act as barycentric coordinates for the triangle in strain space with vertices $e_1$, $e_2$ and $e_3$.
In terms of these, the differential inclusion $e(u) \in \K$ boils down to locally only two martensite variants being present.

\begin{figure}[htbp]
	\centering
 	\begin{tikzpicture}[scale=1.8]
% 	  \node at (-1,1) {b)};
	  \draw [name path=I--J] (90:1) -- (210:1) -- (330:1) -- cycle;
	  \node [fill=black,inner sep=1pt,label=90:$e_1$] at (90:1) {};
	  \node [fill=black,inner sep=1pt,label=180:$e_3$] at (210:1) {};
	  \node [fill=black,inner sep=1pt,label=0:$e_2$] at (330:1) {};
% 	  \node at (30:1) {$V_2$};
% 	  \node at (150:1) {$V_3$};
% 	  \node at (270:1) {$V_1$};
 	  \node at (0:0.3) {$\K$};
	\end{tikzpicture} 
	\caption{\label{K} Sketch of $\K$.}
 \end{figure}

In plain words, the classification of solutions states that
\begin{enumerate}
 \item only two martensite variants are involved, see Definition \ref{def:degenerate},
 \item or the volume fractions $\theta$ only depend on one direction and look like a second order laminate, see Definition \ref{def:second-order_laminate},
 \item or they are independent of one direction and look like a checkerboard of up to two second-order laminates crossing, see Definition \ref{def:checkerboard},
 \item or they are independent of one direction and macroscopically look like three second-order laminates crossing in an axis, see \ref{def:triple}.
\end{enumerate}

Comparing this list to the list of observed microstructures in the introduction, we see that three crossing second-order laminates are missing.
Indeed, we are unaware of them being mentioned in the currently available literature.
One possible explanation is that planar triple intersections are an artifact of the linear theory.
Another one is that its very rigid geometry, see Definition \ref{def:triple}, leads to it being unlikely to develop during the inherently dynamic process of microstructure formation.

Furthermore, we see that the theorem of course captures neither wedges (which are known to be missing in the geometrically linearized theory anyway \cite{Bhattacharya93comparison}) nor habit planes due to austenite being absent.
Unfortunately, an extension of the theorem including austenite does not seem tractable with the methods used here:
The central step allowing to classify all solutions of the differential inclusion is to that most configurations are independent of some direction.
And even those that do depend on all three variables have a direction in which they vary only very mildly.
However, with austenite being present this property is lost, as the following example shows:

\begin{lemma}\label{lemma:austenite_present}
	There exist solutions $u: \R^3 \to \R^3$ of the differential inclusion $e(u) \in \K\cup \{0\}$ such that $e(u)$ has a fully three dimensional structure.
\end{lemma}
We will give the construction in Subsection \ref{subsec:macro_construction}.

Note that Theorem \ref{main} strongly restricts the geometric structure of the strain, even if the four cases exhibit varying degrees of rigidity.
Therefore, we can interpret it as a rigidity statement for the differential inclusion $e(u) \in \K$.
For example, it can be used to prove that $u(x) \equiv e \in \K$ is the only solution of the boundary value problem
\[\begin{cases}
 e(u)  \in \K & \text{ in }\ball{0}{1},\\
 u(x)  \equiv e x  & \text{ on } \partial \ball{0}{1}
\end{cases}\]
with affine boundary data, for which convex integration constructions would give a staggering amount of solutions with complicated geometric structures.
This can be seen by transporting the decomposition into one-dimensional functions of Definitions \ref{def:degenerate}-\ref{def:triple} to the boundary using the fact that they are unique up to affine functions, see \cite[Lemma 5]{CO12}.

\subsection{Inferring the microscopic behavior}\label{subsection:Young}
In order to properly interpret the various cases Theorem \ref{main} provides, we first need a clear idea of precisely what information the local volume fractions contain.
In principle, they have the same downside of using Young measures to describe microstructures:
They do not retain information about the microscopic geometric properties of the microstructures.
In fact, the Young measures generated by finite energy sequences are determined by the volume fractions and are given by the expression $\sum_{i=1}^3\theta_i \delta_{e_i}$, since the Young measures concentrate on the matrices $e_1$, $e_2$ and $e_3$, which span a non-degenerate triangle.

As every rank-one connection has two possible normals, see equations \eqref{rank_one_connections}, giving rise to two different twins, we cannot infer from the volume fractions which twin is used.
Consequently, what looks like a homogeneous limit could in principle be generated by a patchwork of different twins.
In fact, Figure \ref{fig:two_variant} shows an experimental picture of such a situation.

\begin{figure}
 \centering
 \includegraphics{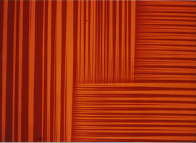}
 \caption{Experimental picture of a two-variant microstructure in a Cu-Al-Ni alloy, by courtesy of R.D.\ James and C.\ Chu.}
 \label{fig:two_variant}
\end{figure}

Additionally, without knowing which twin is present the interpretation of changes in volume fractions is further complicated by the fact there are at least three mechanisms which could be responsible:
\begin{enumerate}
	\item If there is only one twin throughout $\ball{0}{1}$ then the volume fractions can vary freely in the direction of lamination because there are no restrictions on the thickness of martensite layers in twins apart from the very mild control coming from the interface energy.
	\item If there is only one twin, the volume fractions may, perhaps somewhat surprisingly, vary perpendicularly to the direction of lamination in a sufficiently regular manner.
	Constructions exhibiting this behavior have been given by Conti \cite[Lemma 3.1]{Conti00branched} and Kohn, Mesiats and M\"uller \cite{Mesiats16zigzag} for the scalar Kohn-M\"uller model.
	\item There is a jump in volume fractions across a habit plane or a second-order twin.
	As such a behavior costs energy, one would expect that it cannot happen too often.
	However, without assuming the sequence to be minimizing is some sense we can only prove, roughly speaking, that the corresponding set of interfaces has at most Hausdorff-dimension $3-\frac{2}{3}$.
	This will be the content of a follow-up paper.
\end{enumerate}

\subsection{Some notation}\label{subsec:notation}
The rank-one connections between the martensite strains are
\begin{align}\label{rank_one_connections}
	\begin{split}
		e_2 - e_1 & = 6 \, \nu_3^+ \odot \nu_3^-\\
% 		:= 6 \cdot \frac{1}{2}\left( \nu_3^+ \otimes \nu_3^- + \nu_3^- \otimes \nu_3^+ \right),	\\
		e_3 - e_2 & = 6\, \nu_1^+ \odot \nu_1^-,\\
		e_1 - e_3 & = 6\, \nu_2^+ \odot \nu_2^-,
	\end{split}
\end{align}
where the possible normals are given by
\begin{align*}
%  \label{definition_normals}
 \begin{split}
	\nu_1^+:= \frac{1}{\sqrt 2}(011), \nu_1^-:=\frac{1}{\sqrt 2}(01\overline1),\\
	\nu_2^+:= \frac{1}{\sqrt 2}(101), \nu_2^-:=\frac{1}{\sqrt 2}(\overline101),\\
	\nu_3^+:= \frac{1}{\sqrt 2}(110), \nu_3^-:=\frac{1}{\sqrt 2}(1\overline10).
 \end{split}
\end{align*}
Here, we use crystallographic notation, meaning we define $\overline1 := -1$.
In addition, we use round brackets ``( )'' for dual vectors, i.e., normals of planes, while square brackets ``[ ]'' are used for primal vectors, i.e., directions in real space.

These normals can be visualized as the surface diagonals of a cube with side lengths $\frac{1}{\sqrt 2}$, see Figure \ref{fig:geometry_normals_a}.
We group them into three pairs according to which surface of the cube they lie in, i.e., according to the relation $\nu_i\cdot E_i=0$, where $E_i$ is the standard $i$-th basis vector of $\R^3$:
Let
\begin{align*}
	N_1:= \{\nu_1^+,\nu_1^-\},\\
	N_2:=\{\nu_2^+,\nu_2^-\},\\
	N_3:= \{\nu_3^+,\nu_3^-\}.
\end{align*}
Note that this grouping is also appears in equations \eqref{rank_one_connections}.
We will also frequently want to talk about the set of all possible twin and habit plane normals, which we will refer to by $N:=N_1\cup N_2 \cup N_3$.

Throughout the paper we make use of cyclical indices $1,2$ and $3$ corresponding to martensite variants whenever it is convenient.

\begin{remark}\label{rem:combinatorics}
An essential combinatorial property is that for any $\nu_i \in N_i$, $\nu_{i+1} \in N_{i+1}$ with $i\in \{1,2,3\}$ there exists exactly one $\nu_{i-1} \in N_{i-1}$ such that $\{\nu_{i},\nu_{i+1},\nu_{i-1}\}$ is linearly dependent:
Indeed, the linear relation is given by $\nu_j \cdot d = 0$ for a space diagonal
\[d\in \mathcal{D}:= \{[111],[\overline111],[1\overline11],[11\overline1]\}\numberthis\label{intro:space_d}\]
of the unit cube, see Figure \ref{fig:geometry_normals_b}.
We will prove in Step 1 of the Proof of Proposition \ref{prop: six_corner} that they form $120\degree$ angles.
Additionally, for every $\nu \in N$ there exist precisely two $d\in \mathcal{D}$ such that $\nu\cdot d = 0$ and for $\nu \in N_i$ and $\tilde \nu \in N_{i+1}$ there exists a single $d \in \mathcal{D}$ such that $\nu \cdot d = \tilde \nu \cdot d =0$.
In contrast, for each $d\in \mathcal{D}$ we have $\nu_i^+\cdot d=0$ and $\nu_i^- \cdot d \neq 0$ or vice versa.
\end{remark}

\begin{figure}
 \centering
 \subcaptionbox{ \label{fig:geometry_normals_a}}{
  \centering
  \begin{tikzpicture}[scale=3]
    \fill[color=gray,opacity=.6] (-.2,0,-1.2) -- (-.2,1,-1.2) -- (1.2,1,.2) -- (1.2,0,.2);
    \fill[color=gray,opacity=.6] (-.2,0,.2) -- (-.2,1,.2) -- (1.2,1,-1.2) -- (1.2,0,-1.2);
    \draw (0,0,0) -- (0,1,0) -- (1,1,0) -- (1,0,0) -- cycle;
    \draw (0,1,0) -- (0,1,-1) -- (1,1,-1) -- (1,1,0);
    \draw (1,1,-1) -- (1,0,-1) -- (1,0,0);
    \draw[dotted] (0,0,0) -- (0,0,-1) -- (0,1,-1);
    \draw[dotted] (0,0,-1) -- (1,0,-1);
    \draw[-{Latex[length=2mm]},dashed] (.5,0,-.5) -- (.5,1,-.5);
    \node at (.5,1.1,-.5) {$E_3$};
    \draw[-{Latex[length=2mm]}] (.25,0,-.25) -- (.5,0,0);
    \node at (.55,0,.2) {$\nu_3^-$};
    \draw[-{Latex[length=2mm]}] (.75,0,-.25) -- (1,0,-.5);
    \node at (1.15,0,-.3) {$\nu_3^+$};
  \end{tikzpicture}
}
 \subcaptionbox*{}{
  \begin{tikzpicture}
   \draw[-{Latex[length=2mm]}] (0,0,0) --  (1,0,0);
   \node at (1.3,0,.2) {$E_1$};
   \draw[-{Latex[length=2mm]}] (0,0,0) --  (0,0,-1);
   \node at (0,.2,-1) {$E_2$};
   \draw[-{Latex[length=2mm]}] (0,0,0) --  (0,1,0);
   \node at (-.3,1,0) {$E_3$};
  \end{tikzpicture}
 }
 \subcaptionbox{\label{fig:geometry_normals_b}}{
  \centering
  \begin{tikzpicture}[scale=3]
     \draw (0,0,0) -- (0,1,0) -- (1,1,0) -- (1,0,0) -- cycle;
      \draw (0,1,0) -- (0,1,-1) -- (1,1,-1) -- (1,1,0);
      \draw (1,1,-1) -- (1,0,-1) -- (1,0,0);
      \draw[dotted] (0,0,0) -- (0,0,-1) -- (0,1,-1);
      \draw[dotted] (0,0,-1) -- (1,0,-1);
      \draw[-{Latex[length=2mm]},dashed] (0,1,0)--(1,0,-1);
      \node at (.8,.2,-1) {$d$};
      \draw[-{Latex[length=2mm]}] (0,0,0) -- node[right]{$\nu_2$} (1,1,0);
      \draw[{Latex[length=2mm]}-] (1,1,0) -- node[above]{$\nu_3$} (0,1,-1);
      \draw[{Latex[length=2mm]}-] (0,1,-1) --node[left]{$\nu_1$} (0,0,0);
      \fill[color=gray,opacity=.6] (0,0,0) -- (1,1,0) -- (0,1,-1);
      \node [fill=black,inner sep=1.5pt,circle,label=270:$p$] at (.3333,.6666,-.3333) {};
      \node at (-.2,0,.2) {\phantom{h}};
  \end{tikzpicture}
 }
 \caption{ a) Sketch relating the normals $\nu_{3}^+, \nu_3^- \in N_3$ of the gray planes and $E_3$. Primal vectors are shown as dashed, dual vectors as continuous lines. The picture does not attempt to accurately capture the lengths.
 b) Sketch showing the linearly dependent normals $\nu_1^+$, $\nu_2^+$ and $\nu_3^-3$ spanning the gray plane. The point $p$ indicates the intersection of the affine span of the space diagonal $[11\overline1] \in \mathcal{D}$, see definition \eqref{intro:space_d}, with the span of the normals.}
 \label{fig:geometry_normals}
\end{figure}
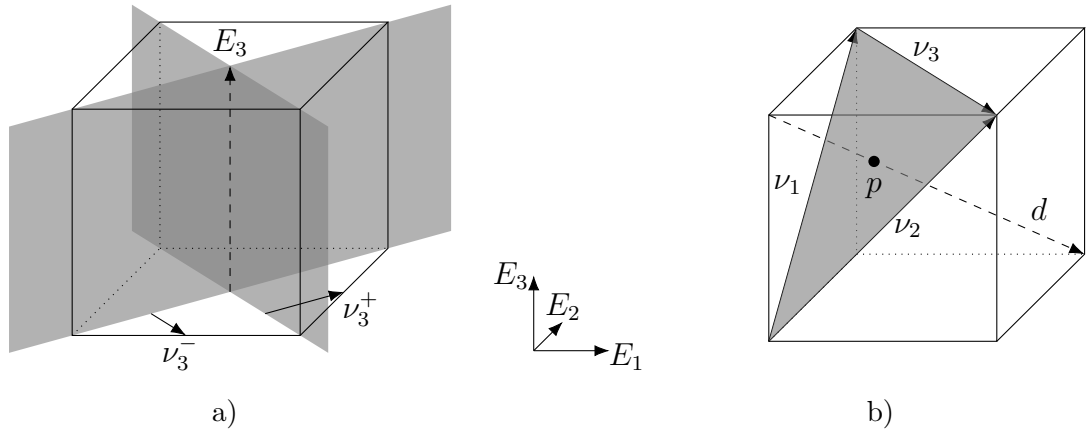

Additionally, we will also set
\begin{align*}
	\pi_\nu (x) := \nu \cdot x \text{ and } H(\alpha,\nu) := \left\{x \in \R^3: x \cdot \nu = \alpha\right\}
\end{align*}
for $\nu \in N$ to be the projection onto $\operatorname{span} (\nu)$, respectively the plane normal to $\nu$ containing $\alpha \nu$ for $\alpha \in \R$.

Furthermore, we use the notation $A \lesssim B$ if there exists a universal constant $C>0$ such that $A \leq C B$.
In proofs, such constants may grow from line to line in proofs.
In a similar vein, radii $r>\tilde r$ may shrink, where $\tilde r>0$ is a universal lower radius that stays fixed throughout a proof and whose numerical value we will typically choose at the end of the argument.

\subsection{Description of the limiting configurations} \label{subsec:limiting_configurations}

In the following we describe all types of configurations we can obtain as weak limits.
We start with those in which globally only two martensite variants are involved.
\begin{Def}\label{def:degenerate}
	We say that the configuration $e(u) \in \K$ is a \emph{two-variant configuration} on $\ball{0}{r}$ with $r>0$ if there exists $i \in \{1,2,3\}$ such that
	\begin{align*}
     \theta_i(x) & \equiv \phantom{-} 0,\\
     \theta_{i+1}(x) & \equiv \phantom{-} f_{\nu_i^+}\left(\nu_i^+\cdot x \right) + f_{\nu_i^-}\left(\nu_i^-\cdot x \right) +  \lambda x_i + 1,\\
     \theta_{i-1}(x) & \equiv - f_{\nu_i^+}\left(\nu_i^+\cdot x \right) - f_{\nu_i^-}\left(\nu_i^-\cdot x \right) -  \lambda x_i,
    \end{align*}
	for all $x \in \ball{0}{r}$, for some $\lambda \in \R$ and measurable functions $f_\nu$ for $\nu \in N_i$.
	For a definition of the normals $\nu$ see Subsection \ref{subsec:notation}.
\end{Def}

\begin{figure}[h]
     \centering
     \subcaptionbox{\label{fig:degenerate_a}}{
      \centering
      \includegraphics[height=3.9cm]{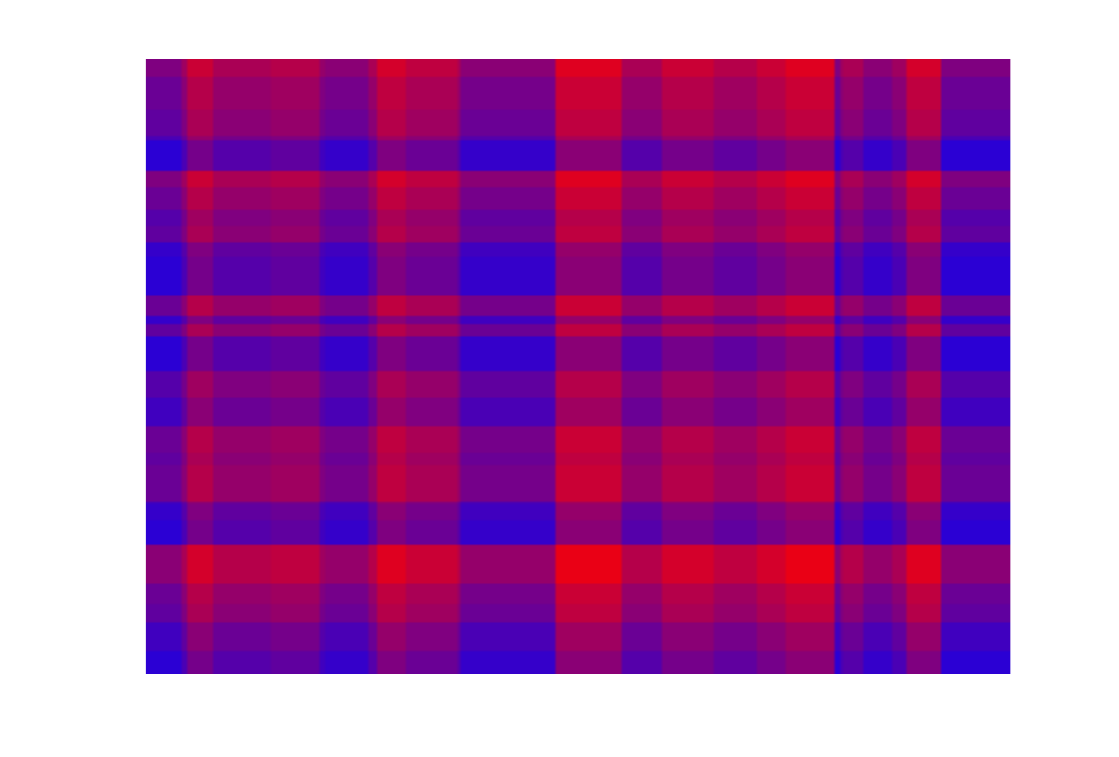}
   }
   \subcaptionbox{\label{fig:degenerate_b}}{
      \centering
      \begin{tikzpicture}
	\node at (0,0) {\includegraphics[height=4cm]{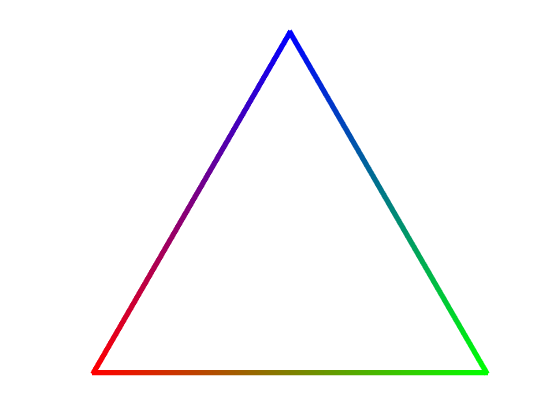} };
	\node at (.2,2) {$e_{i-1}$};
	\node at (-2,-1.75) {$e_{i+1}$};
	\node at (2.2,-1.7) {$e_i$};
      \end{tikzpicture}
  }
%   \subcaptionbox{\label{fig:two_variant}}{
%     \centering
%     \includegraphics[height=4.5cm]{degenerate_experimental.png}
%     c) Experimental of a two-variant microstructure in a Cu-Al-Ni alloy, by courtesy of R.D.\ James and C.\ Chu.
%   }
     \caption{a) Cross-section through a two-variant configuration. The configuration may be affine in the direction perpendicular to the cross-section.
     b) Color code indicating the volume fractions of martensite variants with pure blue, green and red corresponding to pure phases.
     } \label{fig:degenerate}
\end{figure}
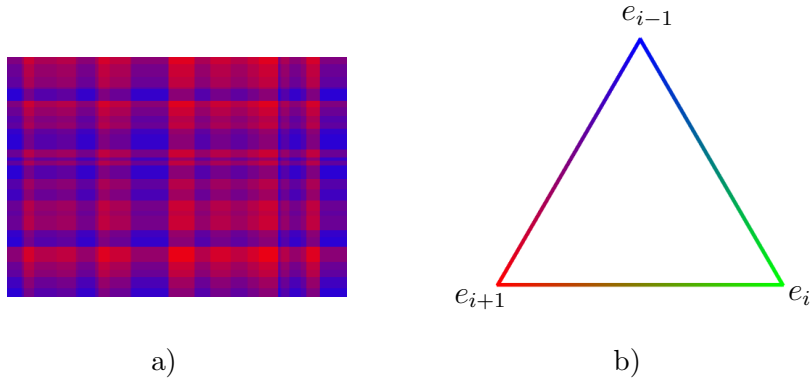 

% Apparently only two martensite variants are involved, which explains our terminology.

An experimental picture of a two-variant configuration can be found in Figure \ref{fig:two_variant}, but be warned that comparing it with Figure \ref{fig:degenerate_a} is not entirely straightforward:
The former fully resolves a microstructure with mostly constant overall volume fraction.
In contrast, the latter only keeps track of the local volume fractions indicated by mix of pure red and blue, and indicates how they can vary in space.
Their deceptively similar overall geometric structure is due to the rank-one connections  for the microscopic and macroscopic interfaces coinciding.
This is also the reason why we cannot infer the microscopic structure from the limiting volume fractions.
We can only say that the affine change in $x_i$ should be due to Mechanism 2 from Subsection \ref{subsection:Young}.

In the context of the other structures appearing in Theorem \ref{main}, two-variant configurations are best interpreted as their building blocks, since said structures typically consist of patches where only two martensite variants are involved.
In the following, we will see that on these patches the microstructures are usually much more rigid than those in Figure \ref{fig:degenerate_a}.
This is a result of the non-local nature of kinematic compatibility when gluing two different two-variant configurations together to obtain a more complicated one.
% 
% One could also be tempted to interpret degenerate configurations as representing special cases of the various microstructures mentioned in Subsection \ref{subsec:shape_memory_effect}.
% However, we advise against doing so, because at this level there is no telling how the three mechanisms mentioned in Section \ref{subsection:Young} interact to produce the configuration:
% As the entire configuration takes values in a single edge of the triangle the rank-one connection for microscopic twinning and macroscopic interfaces is the same.
% Therefore the possible normals for both agree, so that we get next to no information about the underlying geometry of the microstructures from the changes in volume fraction.

Apart from two-variant configurations, all others will only depend on two variables.
We will call such configurations planar.

\begin{Def}\label{def:planar}
	A configuration is \emph{planar with respect to $d \in \{[111], [\overline111],[1\overline11],[11\overline1]\}$} on a ball $\ball{0}{r}$ with $r>0$ if the following holds:
	There exist measurable functions $f_{\nu_i}$ only depending on $x\cdot \nu_i$ and affine functions $g_j$ with $\partial_d g_j=0$ such that
	\begin{align}\label{reduced_decomposition}
 		\begin{split}
 		  \theta_1 & = \phantom{ {}-{} f_{\nu_1} {}+{} } f_{\nu_2} - f_{\nu_3} + g_1,\\
		  \theta_2 & = {}-{} f_{\nu_1} \phantom{ {}-{}  f_{\nu_2}} + f_{\nu_3} + g_2,\\
		  \theta_3 & = \phantom{ {}-{} } f_{\nu_1} - f_{\nu_2} \phantom{ {} + {} f_{\nu_3}} +g_3
	 	 \end{split}
	\end{align}
	on $\ball{0}{r}$.
	Here $\nu_i$ is the unique normal $\nu_i \in N_i$ with $\nu_i \cdot d = 0$, see Figure \ref{fig:geometry_normals_b}.
\end{Def}

There will be three cases of planar configurations, which at least in terms of their volume fractions look like one of the following:
single second-order laminates, ``checkerboard'' structures of two second order laminates crossing, and three single interfaces of second order laminates crossing.

The first two cases are closely related to each other, the first one being almost contained in the second.
However, the first case has slightly more flexibility away from macroscopic interfaces.
Despite the caveat discussed in Subsection \ref{subsection:Young}, we will name them planar second-order laminates.

\begin{Def}\label{def:second-order_laminate}
 A configuration is a \emph{planar second-order laminate} on a ball $\ball{0}{r}$ for $r>0$ if there exists an index $i\in \{1,2,3\}$ and $\nu \in N_i$ such that
 \begin{align*}
   \theta_{i-1}(x) & =  (1 - a x\cdot \nu - b) \chi_{\stcomp{A}}(x\cdot \nu),\\
   \theta_i (x) & =   \phantom{(1 {}+ {}  } a x\cdot \nu + b,\\
   \theta_{i+1}(x) & =  (1 - a x\cdot \nu + b) \chi_{A}(x\cdot \nu)
 \end{align*}
 with $A\subset \R$ measurable and $a$, $b \in \R$ such that $0\leq \theta_i \leq 1
 $ for almost all $x \in \ball{0}{r}$.
\end{Def}

A sketch of a planar second-order laminate can be found in Figure \ref{fig:second-order_laminate}, along with a matching experimental picture of a Cu-Al-Ni alloy, which, admittedly, undergoes a cubic-to-orthorhombic transformation.

\begin{figure}[h]
     \subcaptionbox{\label{fig:second-order_laminate_a}}{
	  \centering
% 	  \vspace{-\baselineskip}
	  \includegraphics[height=3.5cm]{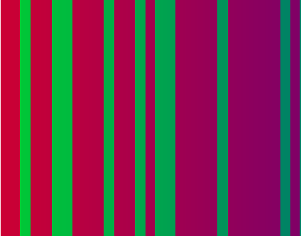}
     }
     \subcaptionbox{\label{fig:second-order_laminate_b}}{
      \centering
      \begin{tikzpicture}
	\node at (0,0) {\includegraphics[height=4cm]{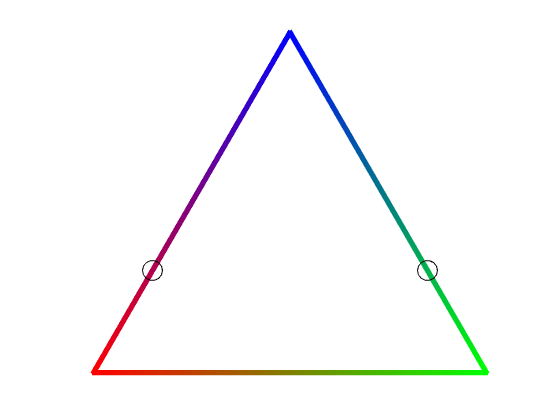} };
	\node at (.2,1.9) {$e_{i}$};
	\node at (-2,-1.75) {$e_{i-1}$};
	\node at (2.4,-1.75) {$e_{i+1}$};
      \end{tikzpicture}
     }
     \subcaptionbox{\label{fig:second-order_laminate_c}}{
		\centering
     		\includegraphics[height=3.5cm]{twin_in_twin.jpg}
     }
     \caption{
     	a) Cross-section of a planar second-order laminate arranged in such a way that it is constant in the direction perpendicular to plane of the paper.
     	b) Color code for the mixtures involved at one of the interfaces in the center of Subfigure \ref{fig:second-order_laminate_a}. The set $\{x\cdot \nu \in A\}$ is shown as mostly green.
     	c) Second-order laminate in a Cu-Al-Ni alloy, by courtesy of C.\ Chu and R.D.\ James.
     	The fine twins correspond to mixtures of pure blue and green and, respectively, blue and red in Subfigure \ref{fig:second-order_laminate_a}.
     	} \label{fig:second-order_laminate}
\end{figure} 

Indeed, such configurations can be interpreted and constructed as limits of finite-energy sequences as follows, using Figure \ref{fig:second-order_laminate} as a guide:
For simplicity let us assume that $A$ is a finite union of intervals, and that $i=1$.
Then on the interior of $\{x\cdot \nu \in A\}$ the configuration will be generated by twins of variants 1 and 2, while on the interior of $\{x\cdot \nu \in \stcomp{A}\}$, it will be generated by twins of variants 1 and 3.
At interfaces, a branching construction on both sides will be necessary to join these twins in a second-order laminate.
In order to realize the affine change in the direction of $\nu$ we will need to combine Mechanisms 1 and 2 of Subsection \ref{subsection:Young} because $\nu$ is neither a possible direction of lamination between variants 1 and 2 or variants 1 and 3, nor is it normal to one of them.
%The corresponding detailed construction is a bit more involved and can be found in Chapter \ref{chapter:constructions}.

The second case consists of configurations in which two second-order laminates cross.
In contrast to the first case, the strains are required to be constant away from macroscopic interfaces leading to only four different involved macroscopic strains.
\begin{Def}\label{def:checkerboard}
	We will say that a configuration is a \emph{planar checkerboard} on $\ball{0}{r}$ for $r>0$ if it is planar and there exists $i\in \{1,2,3\}$ such that
	\begin{align*}
 				 \theta_i(x) = &  - a \chi_A (x\cdot \nu_{i+1})  -  b \chi_B(x\cdot \nu_{i-1})  + 1,\\
 				 \theta_{i+1}(x) = &  \phantom{{}-{}a \chi_A(x\cdot \nu_{i+1})  {}+{} } b \chi_B (x\cdot \nu_{i-1}) ,\\
				 \theta_{i-1}(x) = & \phantom{ {}-{}} a \chi_A(x\cdot \nu_{i+1}) 
	\end{align*}
	with $A,B \subset \R$ measurable, $a,b \geq 0 $ such that $a + b = 1$ and $\nu_j \in N_j$ for $j \in \{1,2,3\}\setminus \{i\}$ on $\ball{0}{r}$.
\end{Def}

For a sketch of such configurations, see Figure \ref{fig:checkerboard}.
An experimental picture can be found in \cite[Figure 17]{basinski1954experiments}

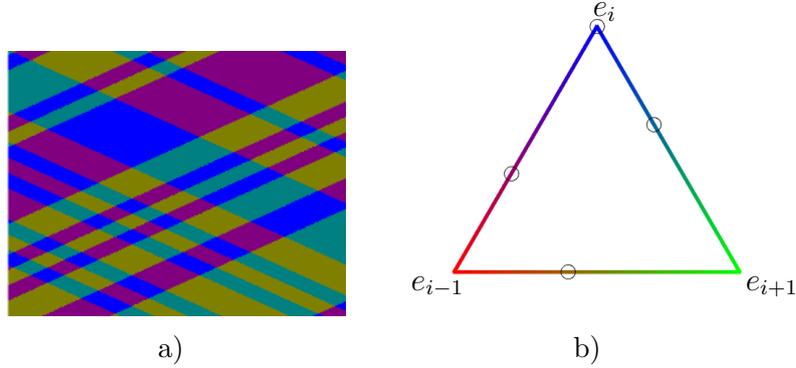
\begin{figure}
	\centering
     \subcaptionbox{}{
		\centering
      		\includegraphics[height=3.5cm]{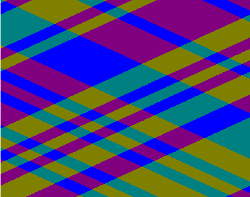}
     }
     \subcaptionbox{}{
      \centering
      \begin{tikzpicture}
	\node at (0,0) {\includegraphics[height=4cm]{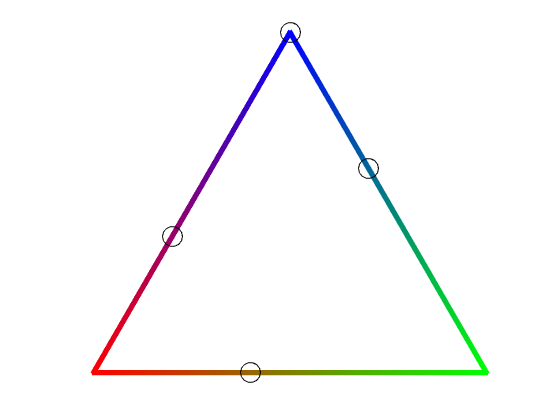} };
	\node at (.2,1.9) {$e_{i}$};
	\node at (-2,-1.75) {$e_{i-1}$};
	\node at (2.4,-1.75) {$e_{i+1}$};
      \end{tikzpicture}     
     }
%     \subcaptionbox{}{
%      \centering
%      \includegraphics[height=3.5cm]{crossing_comp.jpeg}
%     }
     \caption{a) Sketch of a planar checkerboard that is independent of the direction perpendicular to the cross-section.
     b) Color code showing the involved mixtures. The set $\{x\cdot \nu_{i+1} \in \stcomp{A}\}\cap \{x\cdot \nu_{i-1} \in \stcomp{B}\}$, colored in blue, corresponds to pure martensite. The set $\{x\cdot \nu_{i+1} \in A\}\cap \{x\cdot \nu_{i-1} \in \stcomp{B}\}$ is shown in turquoise and $\{x\cdot \nu_{i+1} \in \stcomp{A}\}\cap \{x\cdot \nu_{i-1} \in B\}$ is drawn as purple.
 %    c) Checkerboard structure in an Indium-Thallium crystal, although the bottom region is in the austenite phase. An interpretation of this structure in terms of the differential inclusion $e(u) \in \K \cup \{e_0\}$ can be found in Figure \ref{fig:austenite_2}. Reprinted from \cite{basinski1954experiments}, with permission from Elsevier.
 	} \label{fig:checkerboard}
\end{figure}

Again, we briefly discuss the construction of such limiting strains.
On $\{x\cdot \nu_{i+1} \in \stcomp{A}\}\cap\{ x\cdot \nu_3 \in \stcomp{B}\}$ there is of course only the martensite variant $i$ present.
On all other patches there will be twinning and the macroscopic interfaces require branching constructions unless the interface and the twinning normal coincide, which can only happen if both strains lie on the same edge of $\K$.
In particular, on $\{x\cdot \nu_{i+1} \in A,\, x\cdot \nu_{i-1} \in B\}$ there has to be branching towards all interfaces, i.e., the structure has to branch in two linearly independent directions.
%Also this construction is given in detail in Chapter \ref{chapter:constructions}.
% On $\{x\cdot \nu_2 \in A,\, x\cdot \nu_3 \in \stcomp{B}\}$, there will be twinning between variants 1 and 3.
% If the direction of lamination is $\nu_2$, the interface with $\{x\cdot \nu_2 \in \stcomp{A},\, x\cdot \nu_3 \in \stcomp{B}\}$ can be realized using Mechanism 1 from Subsection \ref{subsec:Young} by just leaving out variant 3.
% Otherwise it requires a branching contruction.
% On the other hand, interfaces with the set $\{x\cdot \nu_2 \in A,\, x\cdot \nu_3 \in B\}$ will require branching because these interfaces will be normal to $\nu_3$ and $\nu_3 \neq \nu_2$.
% Similar considerations hold for the sets $\{x\cdot \nu_2 \in \stcomp{A},\, x\cdot \nu_3 \in \stcomp{B}\}$ and $\{x\cdot \nu_2 \in A,\, x\cdot \nu_3 \in B\}$, except that on the latter set all interfaces require branching.
% As the choices $A, B = \emptyset $ or $A, B= [-r,r]$ are possible as well, the microstructures shown in Figure \ref{fig:twins_and_interfaces} which we set out to find are of this type.
% EIGENTLICH BRAUCHE ICH EINEN PHYSIKALISCHEN GRUND WARUM DEGENERIERTE TEILE NICHT SO OFT AUFTAUCHEN???

Lastly, we remark on the case of three crossing second-order laminates.

\begin{Def}\label{def:triple}
	A configuration is called a \emph{planar triple intersection} on $\ball{0}{r}$ for $r>0$ if it is planar and we have
	\begin{align*}
	  \theta_1(x) & = \phantom{(ax\cdot \tilde \nu_1 + b_1)\chi_{\stcomp{K_1}}{}+{}}  (a x\cdot \tilde\nu_2 +b_2)\chi_{\stcomp{K_2}}(x\cdot\tilde\nu_2) + (ax\cdot \tilde\nu_3 + b_3)\chi_{K_3}(x\cdot\tilde\nu_3),\\
 	  \theta_2(x) & = (ax\cdot \tilde\nu_1+b_1)\chi_{K_1}(x\cdot\tilde\nu_1)  \phantom{{}+{} (ax\cdot \tilde\nu_2 + b_2)\chi_{\stcomp{\stcomp{K_2}}}} + (ax\cdot \tilde\nu_3 + b_3) \chi_{\stcomp{K_3}}(x\cdot\tilde\nu_3),\\
	  \theta_3 (x) &  = (ax\cdot \tilde\nu_1 + b_1)\chi_{\stcomp{K_1}}(x\cdot\tilde\nu_1) + (ax\cdot \tilde\nu_2 + b_2)\chi_{K_2}(x\cdot\tilde\nu_2) \phantom{{}+{} (ax\cdot \tilde\nu_3+ b_3)\chi_{\stcomp{\stcomp{K_3}}}}
	\end{align*}
	for almost all $x\in \ball{0}{r}$.
	Here $\tilde \nu_i = \pm \nu_i$ for $i=1,2,3$ are oriented such that they are linearly dependent by virtue of $\tilde \nu_1 + \tilde \nu_2 + 
	\tilde \nu_3 = 0$, see Remark \ref{rem:combinatorics}.
	Furthermore, we have either
	\[K_i = (-\infty,x_0\cdot \tilde \nu_i]\text{ for } i=1,2,3\]
	or
	\[K_i = [x_0 \cdot \tilde \nu_i,\infty) \text{ for } i=1,2,3\]
	for some $x_0 \in \ball{0}{r}$ and $a, b_i \in \R$ for $i=1,2,3$ such that $\sum_{i=1}^3 b_i = 1$.
\end{Def}
A sketch of a planar triple intersection can be found in Figure \ref{fig:triple_intersection}.

\begin{figure}
     \centering
     \subcaptionbox{}{
      \centering
      \begin{tikzpicture}
	\node at (0,0) {\includegraphics[height=3.3cm]{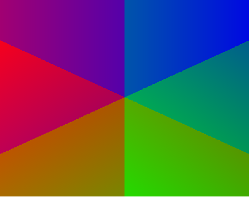}};
	\node at (-1,1.9) {$1$};
	\node at (1,1.9) {$2$};
	\node at (2.3,0) {$3$};
	\node at (1,-1.9) {$4$};
	\node at (-1,-1.9) {$5$};
	\node at (-2.3,0) {$6$};
      \end{tikzpicture}
     }
     \subcaptionbox{}{
      \centering
      \begin{tikzpicture}
	\node at (0,0) {\includegraphics[height=4cm]{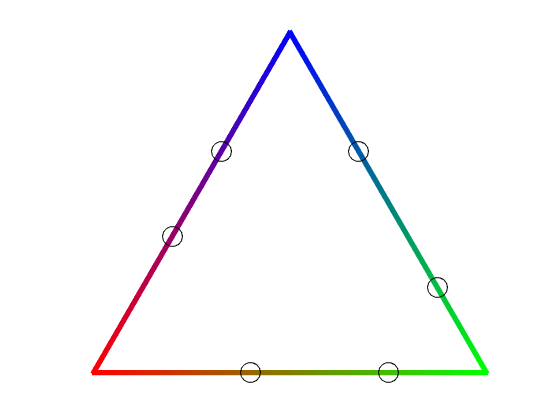} };
	\node at (.2,1.9) {$e_{i}$};
	\node at (-2,-1.75) {$e_{i-1}$};
	\node at (2.4,-1.75) {$e_{i+1}$};
	
	\node at (-.2,.6) {$1$};
	\node at (.4,.6) {$2$};
	\node at (1.15,-.7) {$3$};
	\node at (1.05,-1.2) {$4$};
	\node at (-.275,-1.2) {$5$};
	\node at (-.7,-.25) {$6$};
      \end{tikzpicture}  
     }
     
     \subcaptionbox{\label{fig:triple_c}}{
	\centering
	\begin{tikzpicture}[scale=1.5]
	  \node at (-.75,1.2) {1};
	  \node at (.75,1.2) {2};
	  \node at (1.7,0) {3};
	  \node at (.75,-1.2) {4};
	  \node at (-.75,-1.2) {5};
	  \node at (-1.7,0) {6};
 	  \clip (-1.5,1) -- (1.5,1) -- (1.5,-1) -- (-1.5,-1);
	  \draw[very thick] (30:2) -- (210:2);
	  \draw[very thick] (150:2) -- (330:2);
	  \begin{scope}
  	   \clip (0,1) -- (0,0) -- (210:2) -- (-1.5,1);
	   \foreach \s in {-30,...,30}
	    {
	      \draw[color=gray] ($(150:2)+\s/20*(60:1)$) -- ($(330:2) + \s/20*(60:1)$);
	    };
	  \end{scope}
	  \begin{scope}
  	   \clip (0,1) -- (0,0) -- (330:2) -- (1.5,1);
	   \foreach \s in {-30,...,30}
	    {
	      \draw[color=gray] ($(30:2)+\s/20*(120:1)$) -- ($(210:2) + \s/20*(120:1)$);
	    };
	  \end{scope}
	  \begin{scope}
  	   \clip (0,0) -- (330:2) -- (210:2);
	   \foreach \s in {-30,...,30}
	    {
	      \draw[color=gray] ($(0,-1)+\s/20*(1,0)$) -- ($(0,1) + \s/20*(1,0)$);
	    };
	  \end{scope}
	  \draw[very thick] (0,-1) -- (0,1);
	\end{tikzpicture}
     }
     \caption{a) Sketch of a planar triple intersection that is independent of the direction perpendicular to the cross-section. The numbers relate the different subfigures to each other.
     b) Color code indicating the mixtures involved at the center of the structure.
     c) Sketch indicating a possible choice for the microscopic twins where parallel lines represent equal normals, but neither their volume fractions nor the necessary branching.} \label{fig:triple_intersection}
\end{figure}
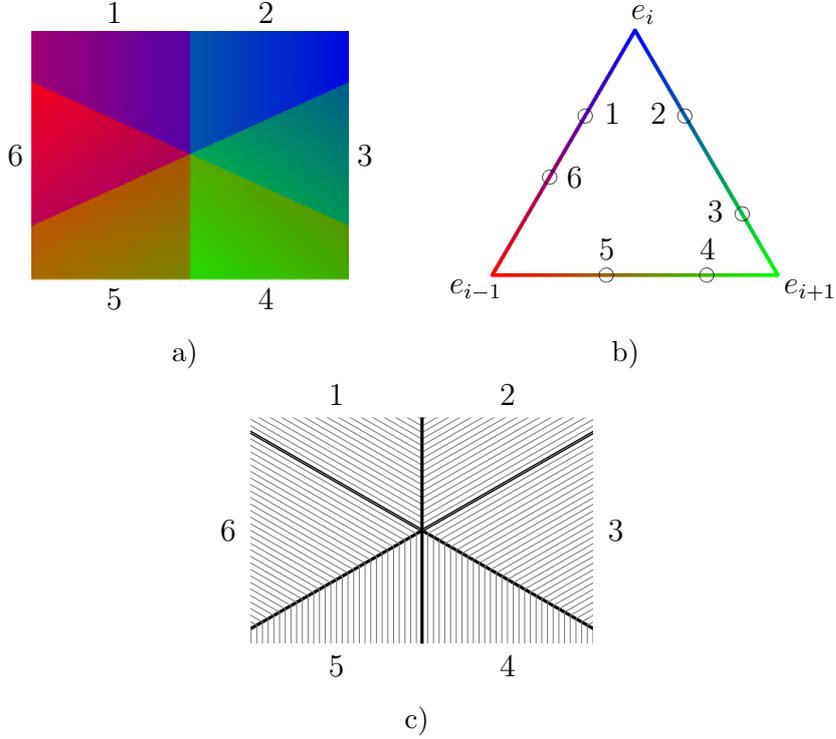

There are a number of possible choices of microscopic twins for constructing triple sections.
We will only describe the simplest one here, which is depicted in Figure \ref{fig:triple_c}.
Going around the central axis the macroscopic interfaces alternate between being a result of Mechanism 1 from Subsection \ref{subsection:Young}, namely varying the relative thickness of layers in a twin, and Mechanism 3, i.e., branching, otherwise.
Similarly to the case of second-order laminates, the affine changes require a combination of Mechanisms 1 and 2 on the individual patches in Figure \ref{fig:triple_c}.

\subsection{Construction of a fully three-dimensional structure in the presence of austenite}\label{subsec:macro_construction}
Here we flesh out the previously announced example in Lemma \ref{lemma:austenite_present}.
The idea is to construct planar checkerboards on hyperplanes $H(c,\nu)$ for some normal $\nu \in N$ and $c \in \R$ that include austenite and between which we can switch as $c$ varies, see Figure \ref{fig:austenite}.

\begin{proof}[Proof of Lemma \ref{lemma:austenite_present}]
Recall $\nu_1^+= \frac{1}{\sqrt 2} (011),\nu_1^- = \frac{1}{\sqrt 2} (01\overline1) $ from Subsection \ref{subsec:notation} and let $\nu_3  :=\nu_3^+ =\frac{1}{\sqrt 2} (110)$.
It is clear that $\{\nu_1^+,\nu_1^-,\nu_3\}$ is a basis of $\R^3$, see also Figure \ref{fig:basis}.
Let $\chi_1^+,  \chi_1^-, \chi_3: \R \to \{0,1\}$ be measurable characteristic functions.
We define the volume fractions to be
\begin{align*}
	\theta_1 &  := \phantom{ 1 - \frac{1}{3}\chi_1^+(x\cdot \nu_1^+) -\frac{1}{3}\chi_1^-(x\cdot \nu_1^-) -{} } \frac{1}{3}\chi_3(x\cdot \nu_3),\\
	\theta_2 & := 1 - \frac{1}{3}\chi_1^+(x\cdot \nu_1^+) -\frac{1}{3}\chi_1^-(x\cdot \nu_1^-) -\frac{1}{3}\chi_3(x\cdot \nu_3),\\
	\theta_3 &  := \phantom{1 - {}} \frac{1}{3}\chi_1^+(x\cdot \nu_1^+)  + \frac{1}{3}\chi_1^-(x\cdot \nu_1^-),
\end{align*}
which clearly satisfy $0\leq \theta_i \leq 1$ for $i=1,2,3$ and $\theta_1 + \theta_2 + \theta_3 \equiv 1$.
As $\{\nu_1^+,\nu_1^-,\nu_3\}$ constitutes a basis of $\R^3$, the structure is indeed fully three-dimensional.

Straightforward case distinctions ensure that $\theta_i = 0$ for some $i=1,2,3$ or $\theta_i =\frac{1}{3}$ for all $i=1,2,3$ almost everywhere.
Setting $G:= \sum_{i=1}^3 \theta_i e_i$ we see that this implies $G \in \K \cup \{0\}$ almost everywhere.
A sketch of cross-sections through $G$ on $H(c,\nu_1^-)$ both with $\chi_1^-(c) =0$ and $\chi_1^-(c) =1$ is given in Figure \ref{fig:austenite}.

\begin{figure}
 \centering
 \begin{subfigure}{.15\linewidth}
  \begin{tikzpicture}
   \draw[-{Latex[length=2mm]}] (0,0,0) --  (1,0,0);
   \node at (1.3,0,.2) {$E_1$};
   \draw[-{Latex[length=2mm]}] (0,0,0) --  (0,0,-1);
   \node at (0,.2,-1) {$E_2$};
   \draw[-{Latex[length=2mm]}] (0,0,0) --  (0,1,0);
   \node at (-.3,1,0) {$E_3$};
  \end{tikzpicture}
 \end{subfigure}
 \begin{subfigure}{.45\linewidth}
   \begin{tikzpicture}[scale=3]
    \fill[color=gray,opacity=.6] (0,-.1,.1) -- (1,-.1,.1) -- (1,1.1,-1.1) -- (0,1.1,-1.1);
    \draw (0,0,0) -- (0,1,0) -- (1,1,0) -- (1,0,0) -- cycle;
    \draw (0,1,0) -- (0,1,-1) -- (1,1,-1) -- (1,1,0);
    \draw (1,1,-1) -- (1,0,-1) -- (1,0,0);
    \draw[dotted] (0,0,0) -- (0,0,-1) -- (0,1,-1);
    \draw[dotted] (0,0,-1) -- (1,0,-1);
    \draw[-{Latex[length=2mm]}] (0,1,0) -- (1,1,-1);
    \node at (.5,1.1,-.5) {$\nu_3$};
    \draw[-{Latex[length=2mm]}] (1,0,0) -- (1,1,-1);
    \node at (1.125,.95,-.25) {$\nu_1^+$};
    \draw[{Latex[length=2mm]}-] (1,0,-1) -- (1,1,0);
    \node at (1.15,.5,-.45) {$\nu_1^-$};
  \end{tikzpicture}
 \end{subfigure}
 \caption{Sketch showing the basis $\{\nu_1^+, \nu_1^-,\nu_3\}$ and a plane with normal $\nu_1^-$, parallel to which the cross-sections of Figure \ref{fig:austenite} are chosen.}
 \label{fig:basis}
\end{figure}
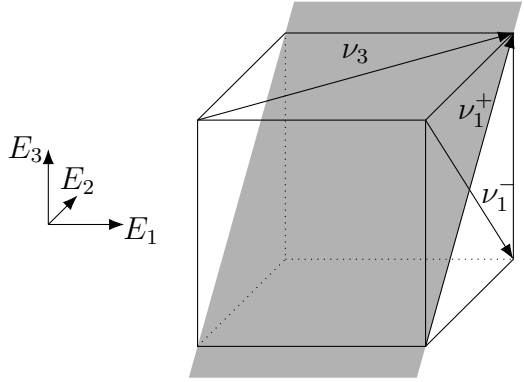

\begin{figure}
 \begin{subfigure}{\linewidth}
  \begin{minipage}{.45\linewidth}
    \centering
    \includegraphics[height=4cm]{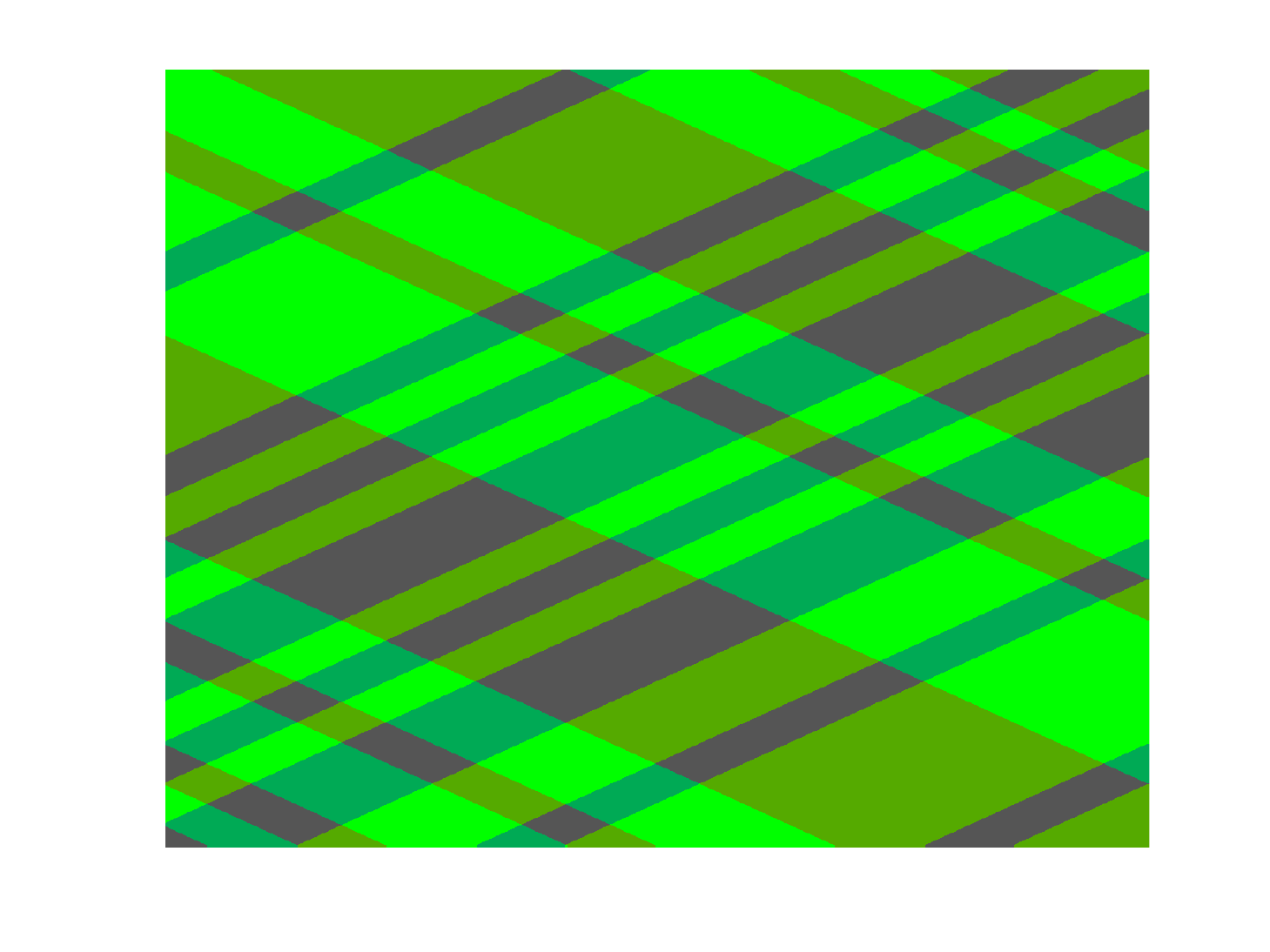}
  \end{minipage}
  \begin{minipage}{.45\linewidth}
    \centering
    \includegraphics[height=4cm]{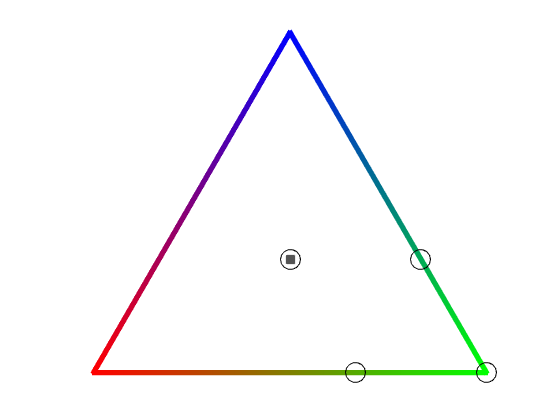}
  \end{minipage}
  \caption{
    The left-hand side shows a cross-section with $x\cdot \nu_1^- = c$ such that $\chi_1^-(c) = 0$.
    (Be warned that the angles between interfaces are not accurate in the picture because we have $\nu_1^-\cdot \nu_1^+, \nu_1^-\cdot\nu_3 \neq 0$.)
    The involved strains are marked on the right-hand side.
  }
  \label{fig:austenite_1}
 \end{subfigure}
 
 \begin{subfigure}{\linewidth}
  \begin{minipage}{.45\linewidth}
    \centering
    \includegraphics[height=4cm]{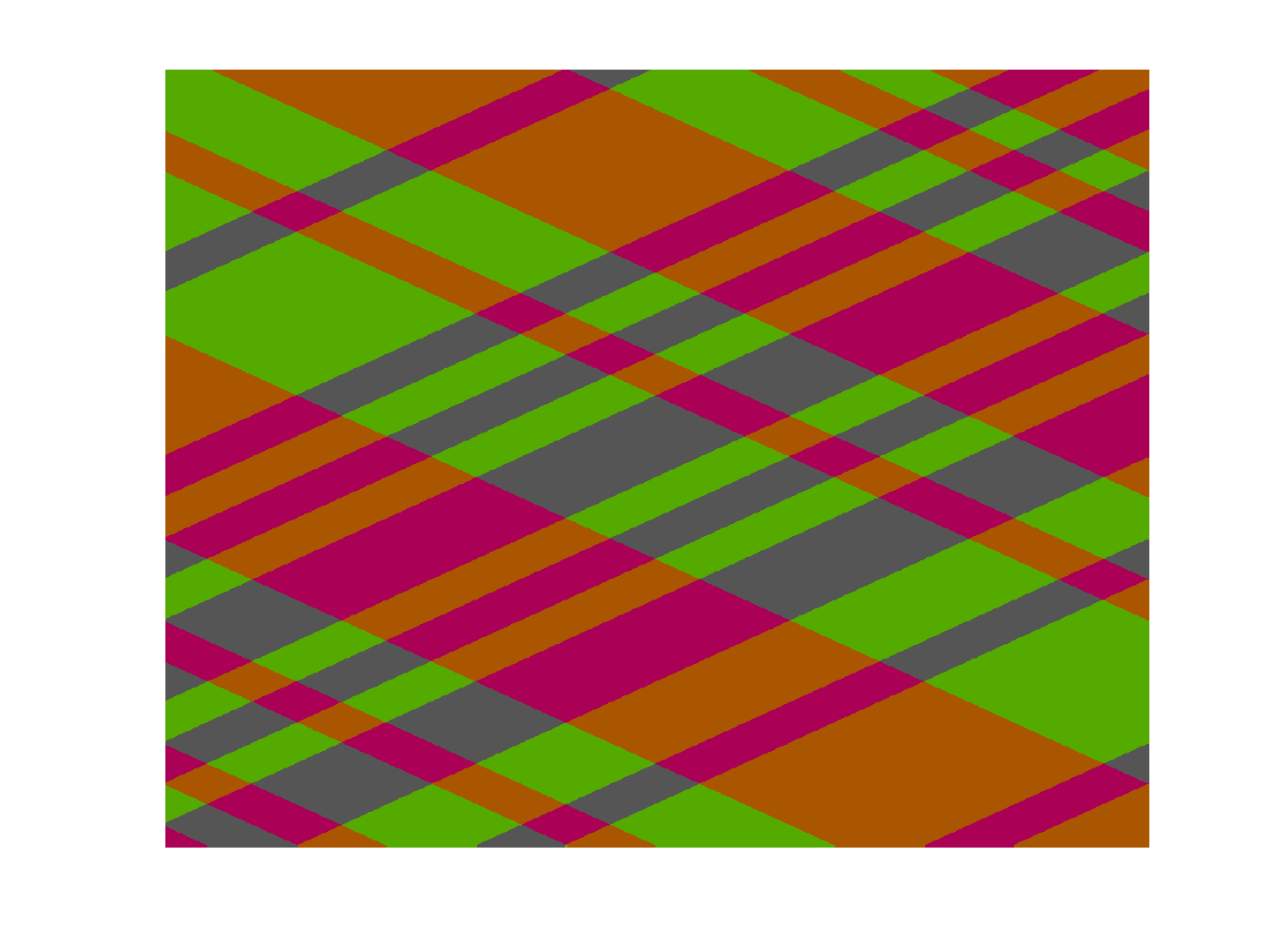}
  \end{minipage}
  \begin{minipage}{.45\linewidth}
    \centering
    \includegraphics[height=4cm]{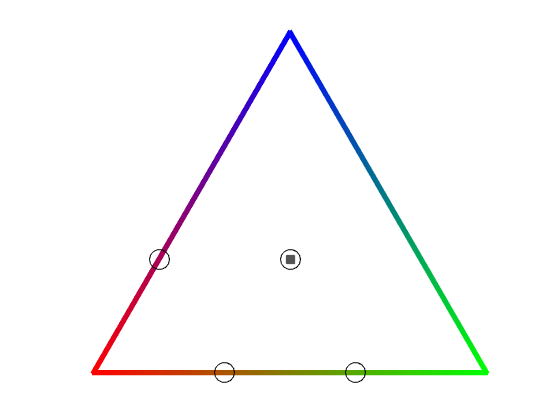}
  \end{minipage}
  \caption{
    On the left-hand side there is a cross-section with $x\cdot \nu_1^- = c$ depicted such that $\chi_1^-(c) = 1$.
    Again, the right-hand side indicates the involved strains.
  }
  \label{fig:austenite_2}
 \end{subfigure}
 \caption{}
 \label{fig:austenite}
\end{figure}

Finally, in order to identify $G$ as the symmetric gradient of a displacement we set
\begin{align*}
	u_1 & := \phantom{ x_2 - F_1^+(x\cdot \nu_1^+) - F_1^-(x\cdot \nu_1^-) - {} } F_3(x\cdot \nu_3),\\
	u_2 & := x_2 - F_1^+(x\cdot \nu_1^+) - F_1^-(x\cdot \nu_1^-) -F_3(x\cdot \nu_3),\\
	u_3 & :=\phantom{x_2 - {} } F_1^+(x\cdot \nu_1^+) - F_1^-(x\cdot \nu_1^-),
\end{align*}
for functions $F_1^+,F_1^-, F_3 : \R \to \R$ such that
\begin{align*}
	(F_1^+)'  = \frac{\sqrt 2}{3} \chi_1^+, \,  (F_1^-)'  = \frac{\sqrt 2}{3} \chi_1^- \text{ and }	(F_3)'  = \frac{\sqrt 2}{3} \chi_3.
\end{align*}
The identity $e(u) \equiv G$ is straightforward to check.
\end{proof}

\section{Outline of the proof}\label{sec:outline}
We will give the ideas behind each individual part of the proof of our main theorem in its own subsection.
The contents of each are organized by increasing detail, so that the reader may skip to the next subsection once they are satisfied with the explanations given.
However, we will first prove Theorem \ref{main} itself here to provide a road map to the following subsections.

Throughout the paper the number $r$ denotes a generic, universal radius that in proofs may decrease from line to line.

\begin{proof}[Proof of Theorem \ref{main}] 
 We first use Lemma \ref{diff_incl} to see that the limiting differential inclusion $e(u)\in \K$ in fact holds.
 Next, we apply Lemma \ref{lemma: decomposition} to deduce the existence of six one-dimensional functions $f_\nu \in L^\infty$ only depending on $x\cdot \nu$ for $\nu \in N$ and three affine functions $g_i$ for $i=1,2,3$ such that
 \begin{align*}
  \begin{split}
 		e(u)_{11} & = \phantom {- f_{(011)} - f_{(01\overline1)} +} f_{(101)} + f_{(\overline101)} - f_{(110)} - f_{(1\overline10)} + g_1,\\
 		e(u)_{22} & = - f_{(011)} - f_{(01\overline1)} \phantom{+ f_{(101)} + f_{(\overline101)} } + f_{(110)} + f_{(1\overline10)} + g_2,\\
 		 e(u)_{33} & =\phantom{-} f_{(011)} + f_{(01\overline1)} - f_{(101)} - f_{(\overline101)}\phantom{ - f_{(110)} - f_{(1\overline10)}} + g_3\\
  \end{split}
 \end{align*}
 on some smaller ball $\ball{0}{r}$.
 
 If $f_\nu \in VMO(-r,r)$ for all $\nu \in N$, then Proposition \ref{Prop: Rigidity VMO} implies that the solution of the differential inclusion is a two-variant configuration.
 If $f_\nu \notin VMO(-r,r)$ for some $\nu \in N_i$ and $i \in \{1,2,3\}$ we can use Proposition \ref{Prop: dimension reduction} to deduce that the configuration is planar or involves only two variants.
 Furthermore, if it is not a two-variant configuration, then there exists a plane $H(\alpha,\nu)$ for some $\alpha \in (-r,r)$ with the following property:
 It holds that
 \[\theta_i|_{H(\alpha,\nu)} = b\chi_B \numberthis \label{additional_information}\]
 for some $0<b<1$ and a Borel-measurable subset $B \subset \{x\cdot \nu = \alpha\} \cap \ball{0}{r}$ of non-zero $\mathcal{H}^2$-measure.
 This is measure-theoretically meaningful since $H(\alpha,\nu)$ is not normal to directions involved in the decomposition of $\theta_j$, see Lemma \ref{lemma: traces}.
 
 We are thus left with classifying planar configurations.
 If additionally one of the one-dimensional functions $f_{\nu_j}$ for $j\in \{1,2,3\}\setminus \{i\}$ is affine, we can apply Lemma \ref{lemma: two_functions} using the additional information \eqref{additional_information} to see that the configuration is a planar second-order laminate or a planar checkerboard.
 Otherwise an application of Proposition \ref{prop: six_corner} yields that the configuration is a planar triple intersection.
\end{proof}

\subsection{The differential inclusion}
We first mention that the inclusion $e(u) \in \K$ holds.

\begin{lemma}\label{diff_incl}
  Let $(u_\eta, \chi_\eta)$ be a sequence of displacements and partitions such that \[\limsup_{\eta \to 0} E_\eta(u_\eta,\chi_\eta) < \infty\]
  for some $0 < C < \infty$. Then for any subsequence for which the weak limits
  \[u_\eta \rightharpoonup u \text{ in } W^{1,2}\text{, } \chi_\eta \stackrel{*}{\rightharpoonup} \theta \text{ in } L^\infty\]
  exist, they satisfy
  \begin{align*}
   e(u) & \equiv \sum_{i=1}^3 \theta_i e_i\text{, } \theta \in \tilde K \text{ and } e(u) \in \K
  \end{align*}
  for almost all $x\in \ball{0}{1}$.
\end{lemma}
%We will need a much more precise version of this argument later in Chapter \ref{chapter:h-measures}.
%Therefore, rather than providing the same argument twice, we refer the reader to Corollary \ref{cor:diff_incl} for the full proof and only provide an outline here:

The statement $e(u) \equiv \sum_{i=1}^3 \theta_i e_i$ is an immediate consequence of the elastic energy vanishing in the limit and the proof of the non-convex inclusion relies on the rescaling properties of the energy.
We will set
\begin{align*}
	r \hat x = x\text{, }\hat u ( \hat x) =  r u (x)\text{, } \hat \chi (\hat x)   = \chi(x)\text{, } r \hat \eta  = \eta,  
\end{align*}
where $\eta$ needs to be re-scaled as well due to it playing the role of a length scale, to obtain
\[ E_{\hat \eta}(\hat u, \hat \chi) = r^{-3 + \frac{2}{3}}  E_\eta(u,\chi).\]
The right-hand side consequently behaves better than just taking averages, which allows us to locally apply the result by Capella and Otto \cite{CO12} to get the statement.

\subsection{Decomposing the strain}

Next, we link the convex differential inclusion
\[e(u) \in S = \{e \in \R^{3\times3}: e \text{ diagonal, } \tr e = 0\}\]
to a decomposition of the strain into simpler objects, namely functions of only one variable and affine functions.
Already Dolzmann and M\"uller \cite{DM95} used the interplay of this decomposition with the non-convex inclusion $e(u) \in \{e_1,e_2,e_3\}$ to get their rigidity result.
% KANN MAN DA AUCH DIE L hoch sinfty SCHRANKE HERHOLEN???
% AUSSERDEM: SPUREN FUER 1D FUNKTIONEN"?

\begin{lemma}\label{lemma: decomposition}
 There exists a universal $r>0$ with the following property:
 Let a displacement $u \in W^{1,2}(\ball{0}{1})$ be such that $e(u) \in K$ a.e., where $K\subset S$ is a compact set.
 Then there exist
 \begin{enumerate}
  \item a function $f_\nu \in L^\infty([-r,r])$ for each $\nu\in N$ which will take $\nu\cdot x$ as its argument and
  \item affine functions $g_1$, $g_2$, $g_3$
 \end{enumerate}
 such that we have
 \begin{align}\label{decomposition}
 	\begin{split}
 		e(u)_{11} & = \phantom {- f_{(011)} - f_{(01\overline1)} +} f_{(101)} + f_{(\overline101)} - f_{(110)} - f_{(1\overline10)} + g_1,\\
 		e(u)_{22} & = - f_{(011)} - f_{(01\overline1)} \phantom{+ f_{(101)} + f_{(\overline101)} } + f_{(110)} + f_{(1\overline10)} + g_2,\\
 		 e(u)_{33} & =\phantom{-} f_{(011)} + f_{(01\overline1)} - f_{(101)} - f_{(\overline101)}\phantom{ - f_{(110)} - f_{(1\overline10)}} + g_3\\
 	\end{split}
 \end{align}
  on $\ball{0}{r}$.
\end{lemma}

Here we abuse notation by dropping $\frac{1}{\sqrt 2}$ when referring to the one-dimensional functions, e.g., we write $f_{(011)}$ instead of $f_{\frac{1}{\sqrt 2} (011)}$.
  Furthermore, we will at times not distinguish between $f_{\nu}$ and $f_{\nu}(\nu\cdot x)$ as long as the context clearly determines which we mean. 

Throughout the paper, we only use the fact that the inclusion $e(u)(x) \in \K$ a.e.\ involves a differential through decomposition \eqref{decomposition}.
Therefore, we can easily transfer all the relevant information to the volume fractions $\theta$ via the relation
\[e(u)_{ii} =  \sum_{j=1}^3 \theta_j e_j = -2\theta_i + \theta_{i+1} + \theta_{i-1} = 1-3\theta_i\]
for all $i=1,2,3$.
In fact, most of the arguments in the following subsections become much more transparent if we re-formulate the differential inclusion in terms of the volume fractions as $\theta(x) \in \tilde \K$ a.e.\ with
\[\tilde \K := \left\{\hat \theta \in \R^3: 0\leq \hat \theta_i \leq 1 \text{ for } i=1,2,3,\ \sum_{i=1}^3 \hat \theta_i =1,\ \hat \theta_i = 0 \text{ for some } i=1,2,3\right\}.\numberthis \label{def_tilde_K}
\]

The only (marginally) new aspect of Lemma \ref{lemma: decomposition} compared to the previously known versions \cite[Lemma 3.2]{DM95} and \cite[Proposition 1]{CO12} is the statement $f_\nu \in L^\infty$ for all $\nu \in N$.
We will thus only highlight the required changes to the proof of Capella and Otto \cite[Proposition 1]{CO12}.
Essentially, the strategy here is to integrate the Saint-Venant compatibility conditions for linearized strains, which in our situation take the form of six two-dimensional wave equations, see Lemma \ref{lemma: wave equations}.
Thus it is not surprising that the decomposition is in fact equivalent to
\[\begin{pmatrix}
	e(u)_{11} & 0 & 0 \\
	0 & e(u)_{22} & 0 \\
	0 & 0 & e(u)_{33}
\end{pmatrix}\]
being a symmetric gradient, which reassures us in our approach of only appealing to the differential information through equations \eqref{decomposition}.

A central part of the proof of Lemma \ref{lemma: decomposition} is uniqueness up to affine functions of the decomposition \cite[Lemma 3.8]{CO12}.
We can apply this result to characterize two-variant configurations as the only ones with $\theta_i \equiv 0$ for some $i =1,2,3$, i.e., as the only ones that indeed only combine two variants.

\begin{cor}\label{cor: theta constant}
 There exists a universal radius $r>0$ with the following property:
 If for $i \in \{1,2,3\}$ we have $\theta_i \equiv 0$ in the setting of Theorem \ref{main}, then the solution of the differential inclusion is a two-variant configuration on $\ball{0}{r}$ according to Definition \ref{def:degenerate}.
\end{cor}

Another very useful consequence of the decomposition \eqref{decomposition} is that such functions have traces on hyperplanes as long as none of the individual one-dimensional functions are necessarily constant on them. See Figure \ref{fig:traces} for the geometry in a typical application.

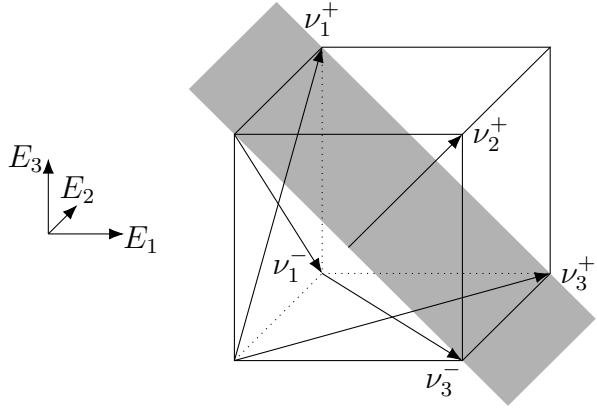
\begin{figure}
 \centering
 \begin{subfigure}{.15\linewidth}
  \begin{tikzpicture}
   \draw[-{Latex[length=2mm]}] (0,0,0) --  (1,0,0);
   \node at (1.3,0,.2) {$E_1$};
   \draw[-{Latex[length=2mm]}] (0,0,0) --  (0,0,-1);
   \node at (0,.2,-1) {$E_2$};
   \draw[-{Latex[length=2mm]}] (0,0,0) --  (0,1,0);
   \node at (-.3,1,0) {$E_3$};
  \end{tikzpicture}
 \end{subfigure}
 \begin{subfigure}{.45\linewidth}
    \begin{tikzpicture}[scale=3]
    \fill[color=gray,opacity=.6] (-.2,1.2,0) -- (-.2,1.2,-1) -- (1.2,-.2,-1) -- (1.2,-.2,0);

    \draw (0,0,0) -- (0,1,0) -- (1,1,0) -- (1,0,0) -- cycle;
    \draw (0,1,0) -- (0,1,-1) -- (1,1,-1) -- (1,1,0);
    \draw (1,1,-1) -- (1,0,-1) -- (1,0,0);
    \draw[dotted] (0,0,0) -- (0,0,-1) -- (0,1,-1);
    \draw[dotted] (0,0,-1) -- (1,0,-1);

    \draw[-{Latex[length=2mm]}] (0,0,0) -- (1,0,-1);
    \draw[-{Latex[length=2mm]}] (0,0,-1) -- (1,0,0);
    \draw[-{Latex[length=2mm]}] (0,0,0) -- (0,1,-1);
    \draw[-{Latex[length=2mm]}] (0,1,0) -- (0,0,-1);
    \draw[-{Latex[length=2mm]}] (.5,.5,0) -- (1,1,0);
    
    \node at (1,0,.215) {$\nu_3^-$};
    \node at (1.125,0,-1) {$\nu_3^+$};
    \node at (-.1,0.075,-.9) {$\nu_1^-$};
    \node at (0,1.125,-1) {$\nu_1^+$};
    \node at (1.125,1,0) {$\nu_2^+$};
  \end{tikzpicture}
 \end{subfigure}
 \caption{Sketch indicating that $\theta_2$ has traces on hyperplanes with normal $\nu_2^+$ since its decomposition only involves continuous functions and the normals $\nu_i^\pm$ for $i=1,3$. As usual we do not keep track of the lengths of the drawn vectors.}
 \label{fig:traces}
\end{figure}

\begin{lemma}\label{lemma: traces}
 Let $F: \R^n \to \mathcal C$ for a closed convex set $\mathcal C \subset \R^m$ satisfy the decomposition
 \[F(x) \equiv \sum_{i=1}^P f_i(x\cdot \nu_i) \numberthis \label{decomp_abstract}\]
 with locally integrable functions $f_i:\R \to \R^m$ and directions $\nu_i\in \Sph^{n-1}$ for $i=1,\ldots,P$.
 Let furthermore $V\subset \R^n$ be a $k$-dimensional subspace such that $\nu_i \notin V^\perp$ for all indices $i=1,\ldots, P$.
 
 Then the decomposition \eqref{decomp_abstract} defines a locally integrable restriction
 $F|_V: V \to \mathcal{C}$ and
 \[F_\delta(x) := \int_\ball{x}{\delta}F(y) \intd \Leb^n(y) \to F(x)\]
 for $\mathcal{H}^k$-almost all $x \in V$.
\end{lemma}

Finally, we give the wave equations constituting the Saint-Venant compatibility conditions.
%Also these we will require in Chapter \ref{chapter:h-measures} in a stronger version given in Lemma \ref{lemma:H-measure_wave_equations}, and thus we here only show how to obtain the present statement from the stronger one.

\begin{lemma}\label{lemma: wave equations}
 If $e(u) \in S$, the diagonal elements of the strain satisfy the following wave equations:
 \begin{align}
  \begin{split}
   \partial_{[111]}\partial_{[\overline 111]} \theta_1& = 0,\\
   \partial_{[1\overline11]}\partial_{[ 11\overline1]} \theta_1 & = 0,\\
   \partial_{[1\overline11]}\partial_{[111]} \theta_2 & = 0,\\
   \partial_{[\overline111]}\partial_{[ 11\overline1]} \theta_2 & = 0,\\
   \partial_{[111]}\partial_{[ 11\overline1]} \theta_3 & = 0,\\
   \partial_{[1\overline11]}\partial_{[\overline 111]} \theta_3 & = 0.
  \end{split}
 \end{align}
\end{lemma}

\subsection{Planarity in the case of non-trivial blow-ups}

While the statements in the previous subsections either rely on rather soft arguments or were previously known, we now come to the main ideas of the paper.
As $\tilde \K$, see definition \eqref{def_tilde_K}, is a connected set, there are no restrictions on varying single points continuously in $\tilde \K$.
However, the crucial insight is that two different points $\tilde \theta, \bar \theta \in \tilde \K$ with $\tilde \theta_1 = \bar \theta_1 >0$ are much more constrained.

To exploit this rigidity, we first for simplicity assume the decomposition
\begin{align*}
  \theta_1(x) & = \phantom{+ f_1(x_1) {}+{} }  f_2(x_2) - f_3(x_3) + 1,\\
  \theta_2(x) & = - f_1(x_1)  \phantom{{}+{}   f_2(x_2)} + f_3(x_3),\\
  \theta_3(x) & = \phantom{+} f_1(x_1) - f_2(x_2). \phantom{- f_3(x_3)}
\end{align*}
Furthermore, suppose that $f_1$ is a $BV$-function with a jump discontinuity of size $\delta f_1$ at $x_1=0$ and that the other functions are continuous.
Thus the blow-up of $\theta$ at some point $(0,x')\in \ball{0}{1}$ takes two values $\tilde \theta, \bar \theta$, both of which satisfy $\tilde \theta_1 = \bar \theta_1 =\theta_1(0,x')$.
A look at Figure \ref{fig:sketch_theta_two_valued} hopefully convinces the reader that $\theta_1(0,x')$ can take at most two values, which furthermore are independent of $x'$.
As it is a sum of two one-dimensional functions some straightforward combinatorics imply that one of the two functions must be constant.
Consequently $\theta$ only depends on two directions.

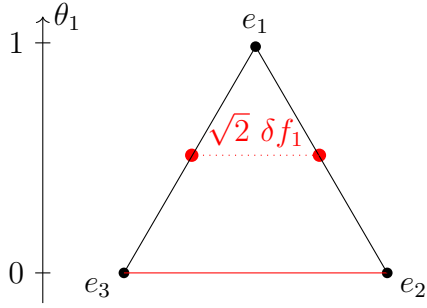
\begin{figure}
    \centering
      \begin{tikzpicture}[scale=2]
% 	  \node at (-1,1) {b)};
	  \fill[red] (0.42,0.28) circle (1.25pt);
	  \fill[red] (-0.42,0.28) circle (1.255pt);
	  \fill (90:1) circle (1pt);
	  \node at (90:1.15) {$e_1$};
	  \fill (210:1) circle (1pt);
	  \node at (210:1.2) {$e_3$};
	  \fill (330:1) circle (1pt);
	  \node at (330:1.2) {$e_2$};
	  \draw (90:1) -- (210:1);
	  \draw[red] (210:1)-- (330:1);
	  \draw (330:1) -- (90:1);
	  \draw[red,dotted] (0.42,0.28) -- (-0.42,0.28);
	  \node at (0,0.45) {\textcolor{red}{$\sqrt 2\,\, \delta f_1$}};
	  \draw[->] (-1.4, -0.7 ) -- (-1.4, 1.2) node[right] {$\theta_1$};
	  \draw (-1.35, -0.5) -- (-1.45, -0.5) node[left] {0};
	  \draw (-1.35, 1.025) -- (-1.45, 1.025) node[left] {1};
      \end{tikzpicture}
     \caption{Illustration of the argument for two-valuedness of $\theta_1$ near $x_1= 0$. The length of the dotted line has to be $\sqrt2\,\, \delta f_1$, where $\delta f_1 >0$ is the size of the jump of $f_1$ at zero. Consequently, the function $\theta_1$ can only take the two values $0$ or $1-\sqrt{\frac{3}{2}  }\,\, \delta f_1$. \label{fig:sketch_theta_two_valued} }
\end{figure}

This can be adapted to our more complex decomposition \eqref{decomposition}, even without any a priori regularity of the one-dimensional functions.
To do so we need to come up with a topology for the blow-ups which respects the non-convex inclusion $e(u) \in \K$, and a quantification of discontinuity for $f_\nu$ which ensures that its blow-up is non-constant.

In order to keep the non-convexity we consider the push-forwards
\[f \mapsto \int_{\ball{0}{1}} f(\theta(x+\epsi y)) \intd y \text{ for } f\in C_0(\R^3)\]
for $x\in \R^3$ and $\eps \to 0$.
This approach is very similar in spirit to using Young-measures, but without a further localization in the variable $y$.
Positing that $f_\nu$ does not have a constant blow-up along some sequence then means that $f_\nu$ does not converge strongly to a constant on average, i.e., it does not converge to its average on average.
If one allows the midpoints $x$ of the blow-ups to depend on $\epsi$, we see that this is equivalent to $f_\nu  \notin VMO$ according to Definition \ref{VMO} given below. 

The resulting statement is:

\begin{prop}\label{Prop: dimension reduction}
 There exists a universal radius $r>0$ with the following property:
 Let $e(u) \in \K$ on $\ball{0}{1}$.
 Furthermore, let the decomposition in Lemma \ref{lemma: decomposition} hold in $\ball{0}{1}$ and let $f_{\nu} \notin VMO([-r,r])$ for some $\nu \in N_i$ with $i\in\{1,2,3\}$.
 Then on $\ball{0}{r}$ the configuration is planar with respect to some $d \in \{[111],[\overline111],[1\overline11],[11\overline1]\}$ with $d\cdot \nu =0$ or we have $\theta_i \equiv 0$, i.e., a two-variant configuration.
 
 Furthermore, if $\theta_i \not \equiv 0$ there exists $\alpha \in (-r,r)$ such that $\theta_i|_{ \{x \cdot \nu = \alpha\}} = b \chi_B $ for some $0<b<1$ and a Borel-measurable set $B \subset H(\alpha,\nu)\cap \ball{0}{r}$ of non-zero $\mathcal{H}^2$-measure.
\end{prop}

Note that the second part is measure-theoretically meaningful by Lemma \ref{lemma: traces}, see in particular Figure \ref{fig:traces}.

For the convenience of the reader, we provide a definition of the space $VMO(U)$ for an open domain $U \subset \R^n$ for $n\in \N$, which is modeled after the one given by Sarason \cite{sarason1975functions} in the whole space case.

\begin{Def}\label{VMO}
 Let $U \subset \R^n$ with $n \in N$ be an open domain and let $f\in L_1(U)$.
 We say that the function $f$ is of bounded mean oscillation, or $f\in BMO(U)$, if we have
 \[\sup_{x\in U, 0< r < 1} \dashint_{\ball{x}{r}\cap U} \left|f(y) - \dashint_{\ball{x}{r}\cap U}  f(z) \intd z \right| \intd y < \infty.\]
 If we additionally have
 \[\lim_{r\to 0} \sup_{x\in U} \dashint_{\ball{x}{r}\cap U} \left|f(y) - \dashint_{\ball{x}{r}\cap U}  f(z) \intd z \right| \intd y =0,\]
 then $f$ is of vanishing mean oscillation, in which case we write $f\in VMO(U)$.
\end{Def}

It can be shown that at least for sufficiently nice sets $U$ the space $VMO$ is the $BMO$-closure of the continuous functions on $U$ and as such it serves as a substitute for $C(U)$ in our setting.
Functions of vanishing mean oscillation need not be continuous, although they do share some properties with continuous functions, such as the ``mean value theorem'', see Lemma \ref{lemma: Mean value theorem for VMO}.
We stress that the uniformity of the convergence in $x$ is crucial and cannot be omitted without changing the space, as can be proven by considering a function consisting of very thin spikes of height one clustering at some point.

There is another slightly more subtle issue in the proof of Proposition \ref{Prop: dimension reduction}:
As already explained, our argument works by looking at a single plane at which we blow-up.
Consequently, we can only distinguish the two cases $\theta_i \equiv 0$ and $\theta_i \not \equiv 0$ on said hyperplane.
Therefore we need a way of transporting the information $\theta_i \equiv 0$ from the hyperplane to an open ball.
Given our combinatorics this turns out to be the 3D analog of the question:
``If $F(x,y) = f(x)+g(y)$ is constant on the diagonal, is it constant on an non-empty open set?''
Looking at the function $F(x,y) = x - y$ one might think that the argument is doomed since $F$ vanishes on the diagonal but clearly does not do us the favor of vanishing on a non-empty open set.

However, the fact that $0$ is an extremal value for $\theta_1$ saves us:
If $F$ is constant on the diagonal of a square and achieves its minimum there, then it has to be constant on the entire square, see also Figure \ref{fig:polyhedron_a}.
For later use we already state this fact in its perturbed form.

\begin{lemma}\label{lemma: almost maxima on transversal lines}
 Let $f, g \in L^\infty(0,1)$ such that $f(x_1) + g(x_2) \geq c$ for almost all $x \in (0,1)^2$ and some constant $c\in \R$.
 Let $\eps \geq 0$ and let one of the following two statements be true:
 \begin{enumerate}
  \item The sum satisfies $f(x_1) + g(x_2) \leq c +\eps$ almost everywhere in $(0,1)^2$.
  \item The sum satisfies $f(t) + g(t) \leq c + \eps$ for almost all $t\in (0,1)$.
 \end{enumerate}
 Then for $\operatorname{ess\ inf} h:= - \esssup - h$ for functions $h \in L^\infty$ it holds that 
 \begin{enumerate}
  \item[3.] We have $f \leq \operatorname{ess\ inf} f + \eps$, $g \leq \operatorname{ess\ inf} g + \eps$ and $c \leq \operatorname{ess\ inf} f + \operatorname{ess\ inf} g \leq c + \eps$ for almost every $x_1,\, x_2 \in (0,1)$.
 \end{enumerate}
 If $\eps = 0$, then all three statements are equivalent.
\end{lemma}

This statement can be lifted to three-dimensional domains.
It states that in order to deduce that $\theta_i$ is constant and extremal, it is enough to know that the extremal value is attained on a suitable line, which we will parametrize by $l(t) := x_0 + \sqrt{2}t E_i$.
Here, $E_i$ is the $i$-th standard basis vector of $\R^3$ and the restriction of $\theta_i$ to the image of $l$ is defined by Lemma \ref{lemma: traces}.
It will later be important that we have a precise description of the maximal set to which the information $\theta_i = 0$ can be transported, which turns out to be the polyhedron
\[P := \bigcap_{\nu \in N_{i+1} \cup N_{i-1}} \{x \in \R^3: \nu \cdot x = \nu \cdot l(I)\},\]
see Figure \ref{fig:polyhedron_b}.
The general strategy of the proof is described in Figure \ref{fig:polyhedron_proof}.

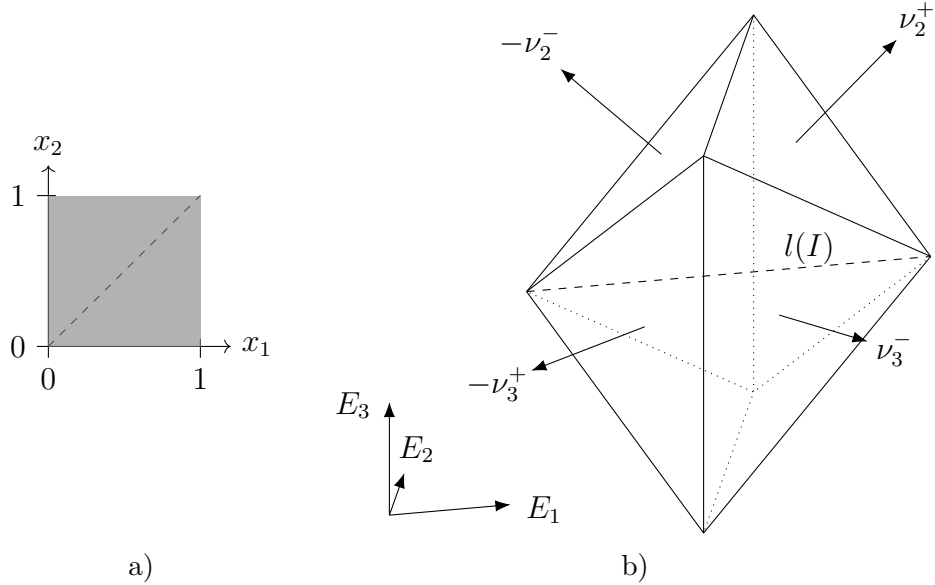
\begin{figure}
 \setbox9=\hbox{
  \begin{minipage}[t]{.1\linewidth}
    \begin{tikzpicture}
      \begin{scope}[rotate around y=14,scale=1.5]
	\draw[-{Latex[length=2mm]}] (0,0,0) --  (1,0,0);
	\node at (1.3,0,.2) {$E_1$};
	\draw[-{Latex[length=2mm]}] (0,0,0) --  (0,0,-1);
	\node at (0.1,.2,-1) {$E_2$};
	\draw[-{Latex[length=2mm]}] (0,0,0) --  (0,1,0);
	\node at (-.3,1,0) {$E_3$};
      \end{scope}
    \end{tikzpicture}
   \end{minipage}
   \begin{tikzpicture}
    \begin{scope}[rotate around y=14,scale=2.5]
%     \draw (1,0,0) -- (0,1,1) -- (-1,0,0) -- (0,-1,-1) -- cycle;
%     \draw (1,0,0) -- (0,1,-1) -- (-1,0,0) -- (0,-1,1) -- cycle;
      \draw (0,-1,1) -- (0,1,1) -- (0,1,-1);
      \draw[dotted] (0,1,-1) -- (0,-1,-1) -- (0,-1,1);
      \draw (1,0,0) -- (0,1,1) -- (-1,0,0);
      \draw (1,0,0) -- (0,-1,1) -- (-1,0,0);
      \draw (1,0,0) -- (0,1,-1) -- (-1,0,0);
      \draw[dotted]  (1,0,0) -- (0,-1,-1) -- (-1,0,0);
    
      \draw[dashed] (-1,0,0) --  (1,0,0);
      \node at (.4,.1,0) {$l(I)$};
%     \node at (1.2,0,0) {$E_1$};
    
      \draw[-{Latex[length=2mm]}] (.333,0,.666) -- ++ (.5,0,.5);
      \draw[-{Latex[length=2mm]}] (-.333,0,.666) -- ++ (-.5,0,.5);
      \draw[-{Latex[length=2mm]}] (.333,.666,0) -- ++ (.5,.5,0);
      \draw[-{Latex[length=2mm]}] (-.333,.666,0) -- ++ (-.5,.5,0);
    
      \node at ($(.333,0,.666)  + (.65,0,.65)$) {$\nu_3^-$};
      \node at ($(-.333,0,.666)  + (-.65,0,.65)$) {$-\nu_3^+$};
      \node at ($(.333,.666,0)  + (.6,.6,0)$) {$\nu_2^+$};
      \node at ($(-.333,.666,0)  + (-.65,.65,0)$) {$-\nu_2^-$};
    \end{scope}
    \end{tikzpicture}
 }
 \centering
 \subcaptionbox{\label{fig:polyhedron_a}}{\raisebox{\dimexpr.5\ht9-.5\height}{
   \centering 
    \begin{tikzpicture}[scale=2]
      \draw[->] (0,0)-- (1.2,0) node[right]{$x_1$};
      \draw[->] (0,0) -- (0,1.2) node[above]{$x_2$};
      \draw (-.075,0)node[left] {$0$}--(0,0);
      \draw (-.075,1)node[left] {$1$}--(0.05,1);
      \draw (0,-.075)node[below] {$0$}--(0,0);
      \draw (1,-.075)node[below] {$1$}--(1,.05);     
      \draw[dashed] (0,0)--(1,1);
      \fill[color=gray, opacity=.6] (0,0) -- (1,0) -- (1,1) -- (0,1)  -- cycle;  
    \end{tikzpicture}
  }
 }
 \subcaptionbox{\label{fig:polyhedron_b}}{\raisebox{\dimexpr\ht9-\height}{
  \centering
   \begin{minipage}[t]{.1\linewidth}
    \begin{tikzpicture}
      \begin{scope}[rotate around y=14,scale=1.5]
	\draw[-{Latex[length=2mm]}] (0,0,0) --  (1,0,0);
	\node at (1.3,0,.2) {$E_1$};
	\draw[-{Latex[length=2mm]}] (0,0,0) --  (0,0,-1);
	\node at (0.1,.2,-1) {$E_2$};
	\draw[-{Latex[length=2mm]}] (0,0,0) --  (0,1,0);
	\node at (-.3,1,0) {$E_3$};
      \end{scope}
    \end{tikzpicture}
  \end{minipage}
  \begin{tikzpicture}
   \begin{scope}[rotate around y=14,scale=2.5]
%     \draw (1,0,0) -- (0,1,1) -- (-1,0,0) -- (0,-1,-1) -- cycle;
%     \draw (1,0,0) -- (0,1,-1) -- (-1,0,0) -- (0,-1,1) -- cycle;
      \draw (0,-1,1) -- (0,1,1) -- (0,1,-1);
      \draw[dotted] (0,1,-1) -- (0,-1,-1) -- (0,-1,1);
      \draw (1,0,0) -- (0,1,1) -- (-1,0,0);
      \draw (1,0,0) -- (0,-1,1) -- (-1,0,0);
      \draw (1,0,0) -- (0,1,-1) -- (-1,0,0);
      \draw[dotted]  (1,0,0) -- (0,-1,-1) -- (-1,0,0);
    
      \draw[dashed] (-1,0,0) --  (1,0,0);
      \node at (.4,.1,0) {$l(I)$};
%     \node at (1.2,0,0) {$E_1$};
    
      \draw[-{Latex[length=2mm]}] (.333,0,.666) -- ++ (.5,0,.5);
      \draw[-{Latex[length=2mm]}] (-.333,0,.666) -- ++ (-.5,0,.5);
      \draw[-{Latex[length=2mm]}] (.333,.666,0) -- ++ (.5,.5,0);
      \draw[-{Latex[length=2mm]}] (-.333,.666,0) -- ++ (-.5,.5,0);
    
      \node at ($(.333,0,.666)  + (.65,0,.65)$) {$\nu_3^-$};
      \node at ($(-.333,0,.666)  + (-.65,0,.65)$) {$-\nu_3^+$};
      \node at ($(.333,.666,0)  + (.6,.6,0)$) {$\nu_2^+$};
      \node at ($(-.333,.666,0)  + (-.65,.65,0)$) {$-\nu_2^-$};
    \end{scope}
   \end{tikzpicture}
  }
 }
 \caption{a) The information $f(x_1) +g(x_2) = c$ along the dashed diagonal can be transported to the whole gray square provided $f(x_1) +g(x_2) \geq c$.\\
 b) Sketch of the polyhedron $P$ with normals $\nu_i^\pm$ for $i=2,3$, which is the maximal set to which we can propagate the information $\theta_1 \equiv 0$ or $\theta_1\equiv1$ on the dashed line $l(I)$.}
 \label{fig:polyhedron}
\end{figure}

There is also a generalization of the one-dimensional functions being almost constant in two dimensions:
In three dimensions, the one-dimensional functions are close to being affine on $P$ in the sense that the inequality \eqref{almost_affinity} holds.
(Lemma \ref{lemma:almost_affine} ensures that then there exist affine functions which are close.)
% Also here the point is identifying $P$ to be the maximal set on which we get the estimate.
As we only need this part of the statement in approximation arguments we may additionally assume that the one-dimensional functions are continuous to avoid technicalities.

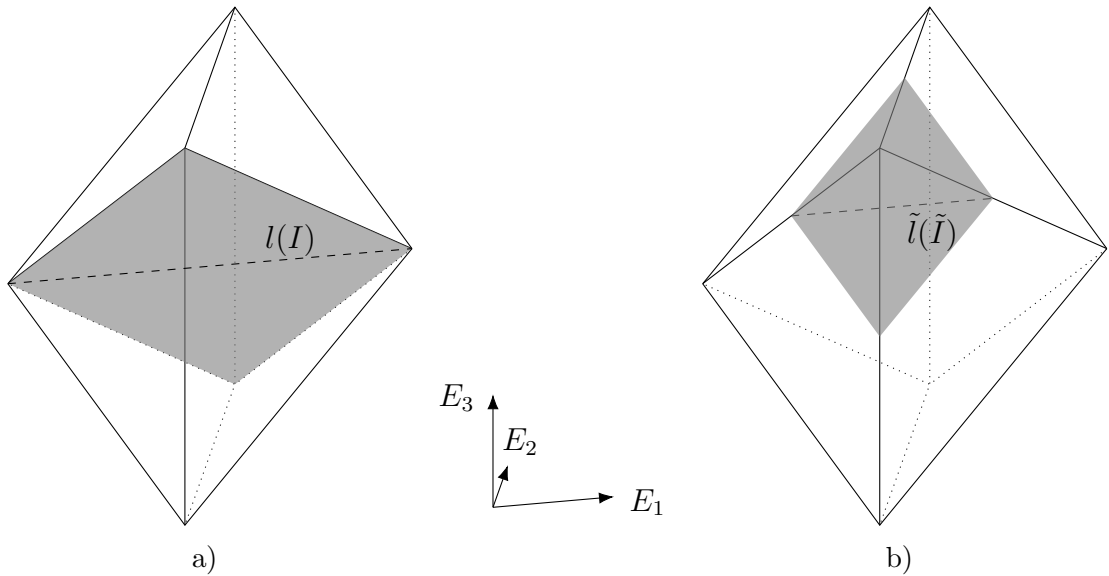
\begin{figure}
 \subcaptionbox{\label{fig:polyhedron_proof_a}}{
  \begin{tikzpicture}
   \begin{scope}[rotate around y=14,scale=2.5]
%     \draw (1,0,0) -- (0,1,1) -- (-1,0,0) -- (0,-1,-1) -- cycle;
%     \draw (1,0,0) -- (0,1,-1) -- (-1,0,0) -- (0,-1,1) -- cycle;
      \draw (0,-1,1) -- (0,1,1) -- (0,1,-1);
      \draw[dotted] (0,1,-1) -- (0,-1,-1) -- (0,-1,1);
      \draw (1,0,0) -- (0,1,1) -- (-1,0,0);
      \draw (1,0,0) -- (0,-1,1) -- (-1,0,0);
      \draw (1,0,0) -- (0,1,-1) -- (-1,0,0);
      \draw[dotted]  (1,0,0) -- (0,-1,-1) -- (-1,0,0);
      
      \begin{scope}
	\fill[color=gray, opacity=.6] (1,0,0) -- (0,-1,-1) -- (-1,0,0) -- (0,1,1);
      \end{scope}
      
      \draw[dashed] (-1,0,0) --  (1,0,0);
      \node at (.4,.1,0) {$l(I)$};
    \end{scope}
   \end{tikzpicture}
 }
    \begin{tikzpicture}
      \begin{scope}[rotate around y=14,scale=1.5]
	\draw[-{Latex[length=2mm]}] (0,0,0) --  (1,0,0);
	\node at (1.3,0,.2) {$E_1$};
	\draw[-{Latex[length=2mm]}] (0,0,0) --  (0,0,-1);
	\node at (0.1,.2,-1) {$E_2$};
	\draw[-{Latex[length=2mm]}] (0,0,0) --  (0,1,0);
	\node at (-.3,1,0) {$E_3$};
      \end{scope}
    \end{tikzpicture}
 \subcaptionbox{\label{fig:polyhedron_proof_b}}{
  \begin{tikzpicture}
   \begin{scope}[rotate around y=14,scale=2.5]
%     \draw (1,0,0) -- (0,1,1) -- (-1,0,0) -- (0,-1,-1) -- cycle;
%     \draw (1,0,0) -- (0,1,-1) -- (-1,0,0) -- (0,-1,1) -- cycle;
      \draw (0,-1,1) -- (0,1,1) -- (0,1,-1);
      \draw[dotted] (0,1,-1) -- (0,-1,-1) -- (0,-1,1);
      \draw (1,0,0) -- (0,1,1) -- (-1,0,0);
      \draw (1,0,0) -- (0,-1,1) -- (-1,0,0);
      \draw (1,0,0) -- (0,1,-1) -- (-1,0,0);
      \draw[dotted]  (1,0,0) -- (0,-1,-1) -- (-1,0,0);
      
%       \begin{scope}
% 	\fill[color=gray, opacity=.5] (1,0,0) -- (0,-1,-1) -- (-1,0,0) -- (0,1,1);
%       \end{scope}
%       
      \draw[dashed] (-.5,.5,.5) --  (.5,.5,.5);
      \begin{scope}[shift={(0,.5,.5)},scale=.5]
	\fill[color=gray, opacity=.6] (1,0,0) -- (0,1,-1) -- (-1,0,0) -- (0,-1,1);
      \end{scope}
      \node at (.2,.35,.5) {$\tilde l(\tilde I)$};
%       \node at (.4,.1,0) {$l(I)$};
% %     \node at (1.2,0,0) {$E_1$};
%     
%       \draw[-{Latex[length=2mm]}] (.333,0,.666) -- ++ (.5,0,.5);
%       \draw[-{Latex[length=2mm]}] (-.333,0,.666) -- ++ (-.5,0,.5);
%       \draw[-{Latex[length=2mm]}] (.333,.666,0) -- ++ (.5,.5,0);
%       \draw[-{Latex[length=2mm]}] (-.333,.666,0) -- ++ (-.5,.5,0);
%     
%       \node at ($(.333,0,.666)  + (.65,0,.65)$) {$\nu_3^-$};
%       \node at ($(-.333,0,.666)  + (-.65,0,.65)$) {$-\nu_3^+$};
%       \node at ($(.333,.666,0)  + (.6,.6,0)$) {$\nu_2^+$};
%       \node at ($(-.333,.666,0)  + (-.65,.65,0)$) {$-\nu_2^-$};
    \end{scope}
   \end{tikzpicture}
  }
  \caption{a) First, we transport the information $\{\theta_1 \approx 0 \}$ from the dashed line $l(I)$ to the gray plane $H(0,(011))\cap P$ using the two-dimensional result.\\
  b) In a second step, we use $\{\theta_1 \approx 0 \}$ along another dashed line $\tilde l (\tilde I)$ parallel to $E_1$ to propagate the information to $H(\alpha,(011))\cap P$ for all $\alpha\in \R$.}
  \label{fig:polyhedron_proof}
\end{figure}

The resulting statement is the following:

\begin{lemma}\label{lemma: almost maxima on transversal lines, 3D}
 There exists a radius $0<r<1$ with the following property:
 Let $\theta$ satisfy decomposition \eqref{decomposition} on $\ball{0}{1}$ and let $0\leq \theta_i \leq 1$ for all $i=1,2,3$.
 Let $I \subset \R$ be a closed interval, let $x_0 \in \R^3$ and let $l(t) := x_0 + \sqrt{2} tE_i \in \ball{0}{r}$ for $t \in I$ and some $i\in\{1,2,3\}$.
 Additionally, let $\nu \in N_i$.
% Let $\pi_\nu(x) := \nu \cdot x$ for $\nu \in N$ and $x \in \R^3$.
 We define the polyhedron $P$ to be
 \[P := \bigcap_{\nu \in N_{i+1} \cup N_{i-1}} \{x \in \R^3: \nu \cdot x \in \nu \cdot l(I)\},\]
 see also Figure \ref{fig:polyhedron}.
 
 For $\eps>0$ assume that either 
 \[ \theta_i\circ l(t) \leq \eps \text{ for almost all } t \in I \text{ or } 1 -\theta_i\circ l(t) \leq \eps \text{ for almost all } t \in I. \numberthis\label{max trans lines, main ass}\]
 
 Then for almost all $x \in P\subset \ball{0}{1}$ we have
 \[0 \leq \theta_i(x) \leq 6\eps \text{ or, respectively, } 1-6\eps \leq \theta_i(x) \leq 1.\numberthis \label{close to maximum on 3D set}\]
 
 Furthermore, if additionally the one-dimensional functions $f_\nu$ are continuous for every $\nu \in N_{i+1} \cup N_{i-1}$, then they are almost affine in the sense that
 \[\left|f_\nu(s + h + \tilde h) + f_\nu(s) - f_\nu(s + h) - f_\nu (s + \tilde h)\right| \leq 24 \eps\numberthis \label{almost_affinity}\]
 for all $(s,h,\tilde h) \in \R \times (0,\infty)^2$ with $s, s+ h, s+ \tilde h, s+h + \tilde h \in \nu\cdot l(I)$.
\end{lemma}

There is yet another minor subtlety of measure theoretic nature.
We already mentioned that we require the midpoints of the blow-ups to be dependent on its radius.
It is thus entirely possible that the radii vanish much faster than the midpoints converge.
This means we cannot use Lebesgue point theory in an entirely straightforward manner to prove that the blow-ups of $f_{\tilde \nu}$ converge to their point values almost everywhere.
We deal with this issue by exploiting density of continuous functions in $L^p$ in a straightforward manner.

\begin{lemma}\label{lemma: shifts don't matter}
 Let $f \in L^p(\R^n)$ for some dimension $n\in \N$ and $1\leq p < \infty$.
 For $\tau>0$ and $y, z \in \R^n$ we have
 \[\lim_{\tau,|z| \to 0} \int_{\R^n} \dashint_{\ball{0}{1}} |f(x + z + \tau y) - f(x)|^p \intd y \intd x =0.\]
\end{lemma}

\subsection{The case \texorpdfstring{$f_\nu \in VMO$}{f in VMO}  for all \texorpdfstring{$\nu \in N$}{n in N}}
Having simplified the case where one of the one-dimensional functions is not of vanishing mean oscillation, we now turn to the case where all of them lie in $VMO$.
The statement we will need to prove here is the following:

\begin{prop}\label{Prop: Rigidity VMO}
 There exists a universal radius $r>0$ with the following property:
 Let $e(u) \in \K$ almost everywhere, and let the decomposition $\eqref{decomposition}$ of Lemma \ref{lemma: decomposition} hold throughout $\ball{0}{1}$.
 Furthermore, let $f_\nu \in VMO([-r,r])$ for all $\nu \in N$.
 Then on $\ball{0}{r}$ the $e(u)$ is a two-variant configuration in the sense of Definition \ref{def:degenerate}.
\end{prop}

\begin{figure}
 \centering
 \begin{tikzpicture}[scale=2]
% 	  \node at (-1,1) {b)};
	  \fill (90:1) circle (1pt);
	  \node at (90:1.15) {$e_1$};
	  \node at (210:1.2) {$e_3$};
	  \node at (330:1.2) {$e_2$};
	  \draw (90:1) -- (210:1);
	  \draw (210:1)-- (330:1);
	  \draw (330:1) -- (90:1);
	  \fill (330:1) circle (1pt);
	  \fill (210:1) circle (1pt);
	  \fill ($(330:1)!.4!(210:1)$) circle (1pt);
	  \node at ($(330:1)!.4!(210:1) - (0,.3)$) {$e(u)(x)$};
	  \draw[->] (-1.4, -0.7 ) -- (-1.4, 1.2) node[right] {$\theta_1$};
	  \draw (-1.35, -0.5) -- (-1.45, -0.5) node[left] {0};
	  \draw (-1.35, 1.025) -- (-1.45, 1.025) node[left] {1};
 \end{tikzpicture}
 \caption{Sketch of how $e(u)(x)$ lies in $\K$. At the boundary of $\theta_1^{-1}(0)$ the strain needs to take the two values $e_2$ and $e_3$.}
 \label{fig:stays_on_side}
\end{figure}
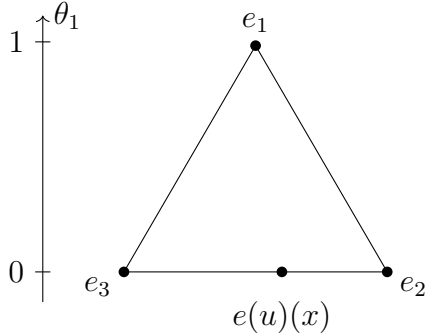

To fix ideas, let us first illustrate the argument in the case of continuous functions in the whole space:

By the mean value theorem the case $e(u) \in \{e_1,e_2,e_3\}$ is trivial, so let us suppose that there is a point $x$ such that $e(u)(x)$ lies strictly between two pure martensite strains.
We may as well suppose $\theta_1(0) = 0$ and $0 < \theta_2(0), \theta_3(0) <1$, see Figure \ref{fig:stays_on_side}.
By continuity, the set $\{\theta_1 = 0\}$ has non-empty interior, and, by the decomposition \eqref{decomposition}, any connected component of it should be a polyhedron $P$ whose faces have normals lying in $N_2\cup N_3$, see Figure \ref{fig:preimage_a}.
Additionally, continuity implies that
\begin{align*}
	e(u) \equiv e_2  \text{ or } e(u) \equiv e_2 \text{ on each face}.
\end{align*}
Unfortunately, on a face with normal in $N_i$ for $i=2,3$ only $\theta_i$ will later be a well-defined function due to Lemmas \ref{lemma: decomposition} and \ref{lemma: traces} after dropping continuity.
Therefore on such a face we can only use the above information in the form
\begin{align*}
	\theta_i \equiv 0  \text{ or } \theta_i \equiv 1.
\end{align*}

\begin{figure}
 \centering
 \subcaptionbox{\label{fig:preimage_a}}{
  \centering
  \begin{tikzpicture}
   \begin{scope}[rotate around y=14,scale=1.8]
      \fill[color=red, opacity=.6] (0,0,0) -- (-1,1,-1)-- (.5,-.5,-2.5) -- (1.5,-1.5,-1.5);
      \draw (0,0,0) -- (-1,1,-1)-- (.5,-.5,-2.5) -- (1.5,-1.5,-1.5) -- cycle;
      \draw (0,0,0) -- (0,-1.5,0);
      \draw (-1,1,-1) -- (-1,-1,-1);
      \draw[dotted] (.5,-.5,-2.5) -- (.5,-2,-2.5);
      \draw (1.5,-1.5,-1.5) -- (1.5,-2.5,-1.5);
      
      \draw[dashed] (.3,-.3,-.3) -- node[right]{$l$} (.3,-.3,-2.3);
      
      \draw[-{Latex[length=2mm]}] (.3,-.6,-.3) -- ++ (.5,0,.5);
      \draw[-{Latex[length=2mm]}] (-.3,-.6,-.3) -- ++ (-.5,0,.5);
      \draw[-{Latex[length=2mm]}] ($(-.3,.3,-.3) + (.3,-.3,-.3)$) -- ++ (.5,.5,0);
%       \draw[-{Latex[length=3mm]}] (-.333,.666,0) -- ++ (-.5,.5,0);    
      \node at ($(.3,-.6,-.3)  + (.65,0,.65)$) {$\nu_3^-$};
      \node at ($(-.3,-.6,-.3)  + (-.7,0,.7)$) {$-\nu_3^-$};
      \node at ($(-.3,.3,-.3) + (.3,-.3,-.3)  + (.6,.6,0)$) {$\nu_2^+$};
%       \node at ($(-.333,.666,0)  + (-.65,.65,0)$) {$-\nu_2^-$};
   \end{scope}
  \end{tikzpicture}
}
 \subcaptionbox*{}{
  \centering
  \begin{tikzpicture}
   \begin{scope}[rotate around y=14,scale=1.3]
    \draw[-{Latex[length=2mm]}] (0,0,0) --  (1,0,0);
    \node at (1.3,0,.2) {$E_1$};
    \draw[-{Latex[length=2mm]}] (0,0,0) --  (0,0,-1);
    \node at (0.1,.2,-1) {$E_2$};
    \draw[-{Latex[length=2mm]}] (0,0,0) --  (0,1,0);
    \node at (-.3,1,0) {$E_3$};
   \end{scope}
  \end{tikzpicture}
 }
 \subcaptionbox{\label{fig:preimage_b}}{
  \centering
  \begin{tikzpicture}
   \begin{scope}[rotate around y=14,scale=1.8]
      \draw  (-1,1,-1)-- (.5,-.5,-2.5) -- (1.5,-1.5,-1.5) -- (0,0,0);
      \draw (0,0,0) -- (0,-.5,0);
      \draw (0,-.5,0) -- (0,-1.5,0);
      \draw (-1,1,-1) -- (-1,-1,-1);
      \draw[dotted] (.5,-.5,-2.5) -- (.5,-2,-2.5);
      \draw (1.5,-1.5,-1.5) -- (1.5,-2.5,-1.5);
      
      \draw[dashed] (.3,-.3,-.3) -- (.3,-.3,-2.3);
      \translatepoint{.3,-.3,-1.3};
      \begin{scope}[shift=(middlepoint),draw=red,thick]
        \draw[dashed] (-1,-1,0) -- (1,-1,0);
        \draw (1,-1,0) -- (1,1,0);
        \draw (1,1,0) -- (-1,1,0);
        \draw[dashed] (-1,1,0) -- (-1,-1,0);
        \draw[dashed] (0,0,-1) -- (-1,-1,0);
        \draw[dashed] (-1,-1,0) -- (0,0,1);
        \draw[dashed] (0,0,-1) -- (1,1,0);
        \draw (1,1,0) -- (0,0,1);
        \draw[dashed] (0,0,-1) -- (-1,1,0);
        \draw (-1,1,0) -- (0,0,1);
        \draw[dashed] (0,0,-1) -- (1,-1,0);
        \draw (1,-1,0) -- (0,0,1);
    \end{scope}
    \draw (0,0,0) -- (-1,1,-1);
   \end{scope}
  \end{tikzpicture}
 }
 \caption{a) Sketch of a connected component $P$ of $\theta_1^{-1}(0)$ with normals $\nu_2^+$, $\nu_3^-$ and $\nu_3^+$. On the red face we get the information $\theta_2 \equiv 0$ or $\theta_2 \equiv 1$. In particular, we get it along the line $l$, which is parallel to $E_2$.
 b) Sketch of the polyhedron $Q$ that transports the information $\theta_2 \equiv 0$ or $\theta_2 \equiv 1$ along $l$ to the inside of $P$.}
 \label{fig:preimage}
\end{figure}
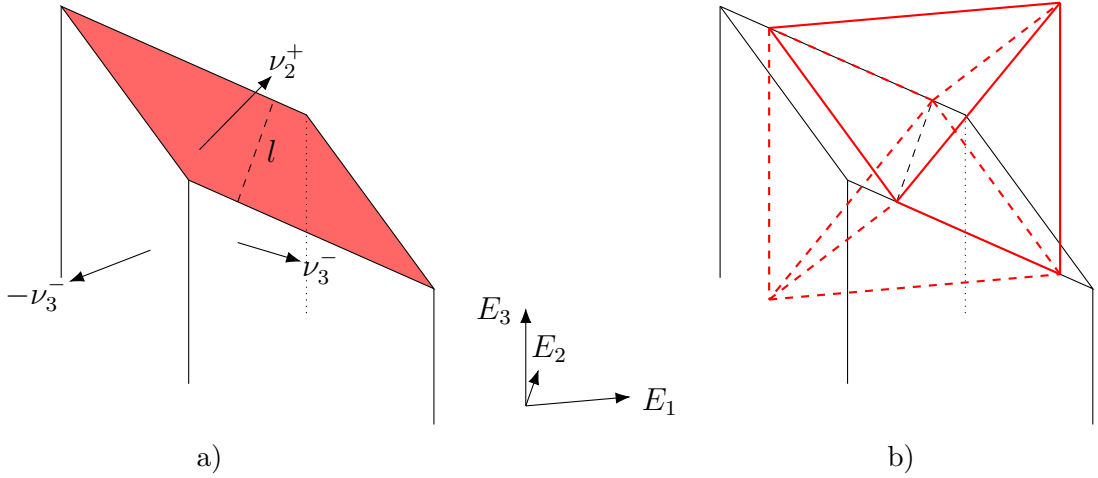

Using Lemma \ref{lemma: almost maxima on transversal lines, 3D} we get a polyhedron $Q$ that transports this information back inside $P$, see Figure \ref{fig:preimage_b}.
The goal is then to show that we can reach $x$ in order to get a contradiction to $e(u)(x)$ lying strictly between $e_2$ and $e_3$, which we will achieve by using the face of $P$ closest to $x$.

In order to turn this string of arguments into a proof in the case $f_\nu \in VMO$ for all $\nu \in N$ the key insight is that non-convex inclusions and approximation by convolutions interact very nicely for $VMO$-functions.
As has been pointed out to us by Radu Ignat, this elementary, if maybe a bit surprising fact has previously been used to in the degree theory for $VMO$-functions, see Brezis and Nirenberg \cite[Inequality (7)]{brezis1995degree}, who attribute it to L.\ Boutet de Monvel and O.\ Gabber.
For the convenience of the reader, we include the statement and present a proof later.

\begin{lemma}[L.\ Boutet de Monvel and O.\ Gabber]\label{lemma: Mean value theorem for VMO}
 Let $f \in VMO(U)$ with $f\in K$ almost everywhere for some open set $U \subset \R^n$ and a compact set $K \subset \R^d$, where we have $n,d \in \mathbb{N}$.
 Let $f_\delta(x) := \dashint_{\ball{x}{\delta}} f(y) \intd y$.
 Then $f_\delta$ is continuous and we have that $\dist (f_\delta, K) \to 0$ locally uniformly in $U$.
\end{lemma}

Unfortunately, formalizing the set $\{\theta_{1,\delta} \approx 0\}$ in such a way that connected components are polyhedra is a bit tricky.
We do get that they contain polyhedra on which the one-dimensional functions are close to affine ones, see Lemmas \ref{lemma: almost maxima on transversal lines, 3D} and \ref{lemma:almost_affine}.
However, we do not immediately get the other inclusion:
As the directions in the decomposition are linearly dependent, one of the one-dimensional functions deviating too much from their affine replacement does not translate into $\theta_1$ deviating too much from zero.

We side-step this issue by first working on hyperplanes $H(\alpha,(011))$.
In that case, the decomposition of $\theta_1$ simplifies to two one-dimensional functions
% , which will turn out to depend on $[11\overline1]$ and $[1\overline11]$,
and thus we do get that connected components of $\{\theta_{1,\delta} \approx 0\} \cap H(\alpha,(011))$ are parallelograms.
The goal is then to prove that at least some of them, let us call them $R_\delta$, do not shrink away in the limit $\delta \to 0$.
Making use of Lemma \ref{lemma: almost maxima on transversal lines, 3D} we can go back to a full dimensional ball and get that the set $\{\theta_1 =0\}$ has non-empty interior.
This allows the argument for continuous functions to be generalized to $VMO$-functions.

\begin{figure}
 \centering
   \begin{tikzpicture}[scale=3]
    \draw (0,0) -- (30:1) -- ($(30:1) + (330:1.25)$);
    \draw (0,0) -- (330:1.25);
    \draw[dashed] (30:.3) -- node[above]{$l$}($(30:1) + (330:.7)$);
    \draw[dotted] (30:.3) -- ++ (-.25,0);
    \draw[dotted] ($(30:1) + (330:.7)$) -- ++ (.5,0);
    \node[fill=black, circle, inner sep=1.5pt,label={$y_\delta$}] at (30:.3) {};
    \node[fill=black, circle, inner sep=1.5pt,label={$z_\delta$}] at ($(30:1) + (330:.7)$){};
%     
%     \draw[-{Latex[length=2mm]}] (-.85,-.15) -- ($(-.85,-.15) + (30:.5)$);
%     \node at ($(-.85,-.15) + (30:.7)$) {$[11\overline1]$};
%     \draw[-{Latex[length=2mm]}] (-.85,-.15) -- ($(-.85,-.15) + (330:.5)$);
%     \node at ($(-.85,-.15) + (330:.7)$) {$[1\overline11]$};
  \end{tikzpicture}
  \caption{Sketch of the parallelogram $R_\delta$. Along the dashed part of the line $l$, which intersects $\partial R_\delta$ at $y_\delta$ and $x_\delta$, the volume fraction $\theta_{2,\delta}$ is almost affine. At $x$ we have $c\leq\theta_{2,\delta}(x)\leq 1-c$ for some $c>0$.}
  \label{fig:does_not_shrink}
\end{figure}
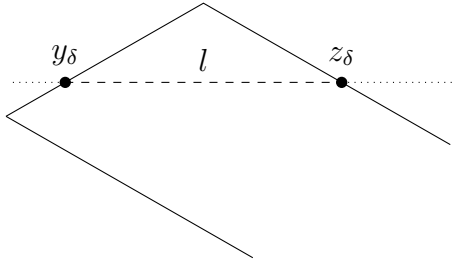

In order to prove that $R_\delta$ does not get too small we choose it such that we are in the situation depicted in Figure \ref{fig:does_not_shrink}.
We will show that $\theta_{2,\delta}(y_\delta) \approx 0$, $\theta_{2,\delta}(z_\delta) \approx 1$ or vice versa.
Together with the fact that $\theta_2 \circ l$ is close to an affine function in a strong topology by the following Lemma \ref{lemma:almost_affine}, the function $\theta_{2}$ would not have vanishing mean oscillation if $R_\delta$ shrank away, i.e., if $|y_\delta - z_\delta| \to 0$.

\begin{lemma}\label{lemma:almost_affine}
 There exists a number $C>0$ with the following property:
 Let $g\in L^\infty([0,1])$ and
 \[\eps := \sup_{ t, t+ h ,t+ \tilde h, t + h + \tilde h \in [0,1] } |g(t + h + \tilde h) - g(t+ h) - g(t+ \tilde h) + g(t)|.\]
 Then there exists an affine function $\tilde g$ such that
 \[||g-\tilde g||_\infty \leq C \left(||g||_\infty^{\frac{1}{2}}\eps^{\frac{1}{2}} + \eps \right).\]
\end{lemma}

This is closely related to the so-called Hyers-Ulam-Rassias stability of additive functions, on which there is a large body of literature determining rates for the closeness to linear functions, see e.g.\ Jung \cite{jung2011hyersulam}.
As such, this statement may well be already present in the literature.
However, as far as we can see, the corresponding community seems to be mostly concerned with the whole space case.

\subsection{Classification of planar configurations}
It remains to exploit the two-dimensionality that was the result of Proposition \ref{Prop: dimension reduction}.
It allowed us to reduce the complexity of the decomposition \eqref{decomposition} to three one-dimensional functions with linearly dependent normals and three affine functions.
We first deal with the easier case where one of the one-dimensional functions is affine and can be absorbed into the affine ones.

\begin{lemma}\label{lemma: two_functions}
 There exists a universal number $r>0$ with the following property:
 
 Let $e(u) \in \K$ almost everywhere.
 Let the configuration be planar with respect to the direction $d \in \{[111], [\overline111], [1\overline11],[11\overline1]\}$ and let it not be a two-variant configuration in $\ball{0}{r}$.
  Furthermore assume for $i$, $j\in \{1,2,3\}$ with $i\neq j$ that the function $f_{\nu_j}$ is affine and that
  \[\theta_i |_{H(\alpha,\nu_i)} = b \chi_{B} \numberthis \label{checkerboard_assumption}\]
  for some $\alpha \in (-r,r)$, a Borel-measurable set $B \subset H(\alpha,\nu_i)$ of non-zero $\mathcal{H}^2$-measure and $0<b<1$. 
 
 Then the configuration is a planar second-order laminate or a planar checkerboard on $\ball{0}{r}$.
\end{lemma}

While the preceding lemma is mostly an issue of efficient book-keeping to reap the rewards of previous work, we now have to make a last effort to prove the rather strong rigidity properties of planar triple intersections:

\begin{prop}\label{prop: six_corner}
 There exists a universal radius $r>0$ with the following property:
 
 Let $e(u) \in \K$ almost everywhere and let the configuration be planar with respect to the direction $d \in \{[111], [\overline111], [1\overline11],[11\overline1]\}$.
 Furthermore let all $f_{\nu_i}$ for $i=1,2,3$ be non-affine on $\ball{0}{\frac{r}{2}}$ and let $|\theta_i^{-1}(0) \cap \ball{0}{r}|>0$ and $|\theta_j^{-1}(0) \cap \ball{0}{r}| > 0$ for $i,j \in \{1,2,3\}$ with $i\neq j$.  

 Then the configuration is a planar triple intersection on $\ball{0}{r}$.
\end{prop}

The idea is to prove that the sets $\theta_i^{-1}(0)$ for $i=1,2,3$ take the form
\[\theta_i^{-1}(0)= \pi^{-1}_{i+1}(J_{i+1}) \cap \pi^{-1}_{i-1}(J_{i-1}),\]
where $J_j \subset \R$ and $\pi_j(x) := \nu_j \cdot x$ for $j=1,2,3$, i.e., they are product sets in suitable coordinates.
Expressing the condition $\bigcup_{i=1}^3 \theta_i^{-1}(0) = \ball{0}{1}$ in terms of these sets allows us to apply Lemma \ref{lemma: intervals} below to conclude that $J_j$ is an interval for $j=1,2,3$.
The actual representation of the strain is then straightforward to obtain.

\begin{lemma}\label{lemma: intervals}
 There exists a universal radius $0<r<\frac{1}{2}$ such that the following holds:
 Let $\nu_1, \nu_2, \nu_3 \subset \Sph^1$ be linearly dependent by virtue of $\nu_1 + \nu_2 + \nu_3 = 0$.
 Let $\pi_i(x) := x\cdot \nu_i$ for $x \in \R^2$ and $i=1,2,3$.
 Let $J_1,J_2,J_3 \subset [-1,1]$ be measurable such that
 \begin{enumerate}
  \item we have
   \begin{align}\label{intervals_assumption_main}
    \begin{split}
		 \big|\ball{0}{r}  \cap \left( \pi_1^{-1}(J_1)\cap\pi_2^{-1}(J_2)\cap\pi_3^{-1}(J_3)\right)\big| & =0,\\
		 \big|\ball{0}{r}  \cap \left( \pi_1^{-1}(\stcomp{ J_1})\cap\pi_2^{-1}(\stcomp{J_2})\cap\pi_3^{-1}(\stcomp{J_3})\right)\big| & = 0,
    \end{split}
   \end{align}
  \item and the two sets $J_1$ and $J_2$ neither have zero nor full measure, i.e., it holds that
    \begin{align}\label{intervals_assumption_non-empty}
      0 <  \left|J_1\cap\left[-\frac{r}{2},\frac{r}{2}\right]\right| , \left|J_2\cap\left[-\frac{r}{2},\frac{r}{2}\right]\right| < 2r.
    \end{align}
 \end{enumerate}    
  Then there exist a point $x_0 \in \R$ such that $x\cdot \nu_i \in (-r,r)$ for all $i=1,2,3$ and, up to sets of $\Leb^1$-measure zero, either
 \[J_i \cap [-r,r] = [-r,x_0\cdot \nu_i] \text{ for } i=1,2,3\]
 or
 \[J_i \cap [-r,r] = [-x_0\cdot \nu_i,r] \text{ for } i=1,2,3.\]
\end{lemma}

To illustrate the proof let us first assume that $J_1$ and $J_2$ are intervals of matching ``orientations'', e.g., we have $J_1 = J_2 = [-r,0]$, in which case
Figure \ref{fig:sketch_intervals_fitting} suggests that also $J_3 = [-r,0]$.

If they are not intervals of matching ``orientations'', we will see that, locally and up to symmetry, more of $J_1$ lies below, for example, the value $0$ than above, while the opposite holds for $J_2$.
The corresponding parts of $J_1$ and $J_2$ are shown in Figure \ref{fig:sketch_intervals_contradiction}.
One then needs to prove that sufficiently many lines $\pi_3^{-1}(s)$ for parameters $s$ close to $0$ intersect the ``surface'' of $\pi_1^{-1}(J_1) \cap \pi_2^{-1}(J_2)$, see Lemma \ref{lemma:surface_product_sets} below.
As a result less than half the parameters around $0$ are contained in $J_3$.
The same argument for the complements ensures that also less than half of them are not contained in $J_3$, which cannot be true.

\begin{figure}
	\centering
	\subcaptionbox{\label{fig:sketch_intervals_fitting}}{
%		\vskip 0pt
		\centering
		\begin{tikzpicture}[scale=.9]
			\begin{scope}
				\clip (-2.5,2) -- (2.5,2) -- (2.5,-2) -- (-2.5,-2);
				\fill[color=blue] (240:5) -- (0,0) -- (180:5) -- (-5,-5);
				\fill[color=red] (0:5) -- (0,0) -- (60:5) -- (5,5);
				\draw[-] (120:5) -- (300:5);
				\draw[pattern=north west lines, pattern color=gray] (120:5) -- (300:5) -- (5,5);
			\end{scope}
%			\node at (1,2.4) {$\pi_1^{-1}(\stcomp{J_1}) \cap \pi_2^{-1}(\stcomp{J_2})$};
%			\node at (-1,-2.4) {$\pi_1^{-1}(J_1) \cap \pi_2^{-1}(J_2)$};
%			\node at (2,-2.4) {$\pi_3^{-1}(J_3)$};
		\end{tikzpicture}
	}
	\subcaptionbox*{}{
%		\vskip 0pt
		\centering
		\begin{tikzpicture}
			 \draw[->] (0,0) -- (90:1)  node[right] {$\nu_1$};
			 \draw[->] (0,0) -- (210:1) node[left] {$\nu_3$};
			 \draw[->] (0,0) -- (330:1) node[right] {$\nu_2$};
		\end{tikzpicture}
	}
		\subcaptionbox{ \label{fig:sketch_intervals_contradiction}}[.37\linewidth]{
			
			\begin{tikzpicture}
				\begin{scope}[cm={1,0,0.5,0.866,(0,0)}]
					\begin{scope}
						\clip (-2.5,2) -- (2.5,2) -- (2.5,-2) -- (-2.5,-2);
						\fill[color=blue] (0,-.4) -- (.7,-.4) -- (.7,0) -- (0,0);
						\fill[color=blue] (1,-.4) -- (1.3,-.4) -- (1.3,0) -- (1,0);
						\fill[color=blue] (1.6,-.4) -- (2.5,-.4) -- (2.5,0) -- (1.6,0);
						\fill[color=blue] (0,-1.4) -- (.7,-1.4) -- (.7,-.7) -- (0,-.7);
						\fill[color=blue] (1,-1.4) -- (1.3,-1.4) -- (1.3,-.7) -- (1,-.7);
						\fill[color=blue] (1.6,-1.4) -- (2.5,-1.4) -- (2.5,-.7) -- (1.6,-.7);
						\fill[color=blue] (0,-2) -- (.7,-2) -- (.7,-1.6) -- (0,-1.6);
						\fill[color=blue] (1,-2) -- (1.3,-2) -- (1.3,-1.6) -- (1,-1.6);
						\fill[color=blue] (1.6,-2) -- (2.5,-2) -- (2.5,-1.6) -- (1.6,-1.6);
					
%					\fill[color=blue] (-1.8,.8) -- (-1.2,.8) -- (-1.2,1) -- (-1.8,1);				
				
						\fill[color=red] (-2.5,0) -- (-1.8,0) -- (-1.8,.8) -- (-2.5,.8);	
						\fill[color=red] (-1.2,0) -- (0, 0) -- (0,.8) -- (-1.2,.8);
						\fill[color=red] (-2.5,1) -- (-1.8,1) -- (-1.8,2) -- (-2.5,2);
						\fill[color=red] (-1.2,1) -- (0,1) -- (0,2) -- (-1.2,2);
				
%					\fill[color=red] (.7,-.4) -- (1,-.4) -- (1,-.7) -- (.7,-.7);
%					\fill[color=red] (1.3,-.4) -- (1.6,-.4) -- (1.6,-.7) -- (1.3,-.7);
%					\fill[color=red] (.7,-1.4) -- (1,-1.4) -- (1,-1.6) -- (.7,-1.6);
%					\fill[color=red] (1.3,-1.4)  -- (1.6,-1.4) -- (1.6,-1.6) -- (1.3,-1.6);
						\begin{scope}[xshift=-13, yshift=-13]
							\draw[-] (135:4) -- (315:4);
						\end{scope}
					\end{scope}
					\node at (205:1.2) {$\pi_3^{-1}(s)$};
				\end{scope}
			\end{tikzpicture}
	}
	\caption{Sketches illustrating the proof of Lemma \ref{lemma: intervals}. The arrows in the middle indicate the three linearly dependent directions $\nu_1$, $\nu_2$ and $\nu_3$. a) The set $\pi_3^{-1}(J_3)$ (hatched) may only intersect $\pi_1^{-1}(\stcomp{J_1}) \cap \pi_2^{-1}(\stcomp{J_2})$ (red) and its complement may only intersect $\pi_1^{-1}(J_1) \cap \pi_2^{-1}(J_2)$ (blue). b) The line $\pi_3^{-1}(s)$ intersects both a subset of $\pi_1^{-1}(J_1) \cap \pi_2^{-1}(J_2)$ (blue) and a subset of $\pi_1^{-1}(\stcomp{J_1}) \cap \pi_2^{-1}(\stcomp{J_2})$ (red). }
	\label{fig:triple_intersection_strategy}
\end{figure}
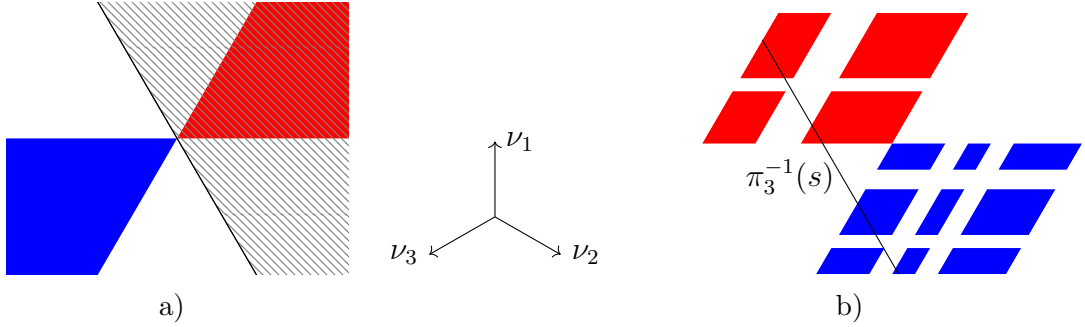

To link intersecting lines to the ``surface area'' we use that our sets are of product structure, i.e., they can be thought of as unions of parallelograms, and that the intersecting lines are not parallel to one of the sides of said parallelograms.
In the following and final lemma, we measure-theoretically ensure the line $\pi_3^{-1}(s)$ intersects a product set $\pi_1^{-1}(K_1) \cap \pi_2^{-1}(K_2)$ by asking
\[\int_{\{x\cdot\nu_3 = s\}} \chi_{K_1}(x\cdot \nu_1) \chi_{K_2}(x\cdot \nu_2) \intd \Hd^1(x) > 0.\]

\begin{lemma}\label{lemma:surface_product_sets}
 Let $\nu_1, \nu_2, \nu_3 \subset \Sph^1$ with $\nu_1 + \nu_2 + \nu_3 = 0$.
 Let $K_1, K_2 \subset \R$ be measurable with $|K_1|,|K_2|>0$.
 Then the set
 \[M:=\left\{s \in \R : \int_{\{x\cdot\nu_3 = s\}} \chi_{K_1}(x\cdot \nu_1) \chi_{K_2}(x\cdot \nu_2) \intd \Hd^1(x) > 0\right\}\]
 is measurable and satisfies $|M|\geq |K_1| + |K_2|$. 
\end{lemma}

\section{Proofs}\label{sec:proof}
 \subsection{The differential inclusion}
\begin{proof}[Proof of Lemma \ref{diff_incl}]
Fixing the sequence $(u_\eta, \chi_\eta)$ we interpret the energies
\[E_\eta(B) := \eta^{-\frac{2}{3}}\int_B \left| e(u)- \sum_{i=1}^3 \chi_ie_i\right|^2 \intd x+\eta^\frac{1}{3} \sum_{i=1}^3 |D\chi_i|(B)\]
as a sequence of finite Radon measures on $\ball{0}{1}$.

Let $y \in \ball{0}{1}$ and $r>0$ be such that $\ball{y}{r}\subset \Omega$.
By translation invariance we can assume $y=0$.
We rescale our functions to the unit ball by setting $\hat x := \frac{x}{r}$ and $\hat \eta := \frac{\eta}{r}$, and defining $\hat u_{\hat \eta} : \ball{0}{1} \to \R^3$ and $\hat \chi_{\hat \eta}: \ball{0}{1} \to \{0,1\}^3$ to be
\[\hat u_{\hat \eta}(\hat x):= \frac{1}{r}u_\eta\left(r \hat x\right)\text{, }\hat \chi_{\hat \eta}(\hat x):= \chi_\eta\left(r \hat x\right).\]
The energy of the rescaled functions is
\begin{align*}
E_{\hat \eta}(\hat u_{\hat \eta},\hat \chi_{\hat \eta}) & = \left(\frac{\eta}{r}\right)^{-\frac{2}{3}}\int_\ball{0}{1} \left| e(u_\eta)(r\hat x) - \sum_{i=1}^3 \chi_i(r\hat x) e_i \right|^2 \intd \hat x + \left(\frac{\eta}{r}\right)^{\frac{1}{3}}|D\hat \chi_{\hat \eta}|(\ball{0}{1})\\
& = r^{-3 + \frac{2}{3}}E_{\eta}(\ball{0}{r}).
\end{align*}
By the Capella-Otto rigidity result \cite{CO12} there exist a universal radius $0<s<1$ such that
\[\min\left\{||\hat \chi_{1,\hat \eta}||_{L^1(\ball{0}{s})},||\hat \chi_{2,\hat \eta}||_{L^1(\ball{0}{s})},||\hat \chi_{3,\hat \eta}||_{L^1(\ball{0}{s})}\right\} \lesssim \left(r^{-3 + \frac{2}{3}}E_{\eta}(\ball{0}{r})\right)^\frac{1}{2}.\]
Rescaling back to $\ball{0}{r}$ we get
\begin{align*}
  \frac{1}{r^3}\min\left\{||\chi_{1,\eta}||_{L^1(\ball{0}{sr})},||\chi_{2,\eta}||_{L^1(\ball{0}{sr})},||\chi_{3,\eta}||_{L^1(\ball{0}{sr})}\right\} \lesssim \left(r^{-3 + \frac{2}{3}}E_{\eta}(\ball{0}{r})\right)^\frac{1}{2}.
\end{align*}
After passing to a subsequence, we have $E_\eta \stackrel{*}{\rightharpoonup} E$ as Radon measures in the limit $\eta \to 0$.
Consequently weak lower semi-continuity of the $L^1$-norm and upper semi-continuity of the total variation on compact sets imply 
\begin{align*}
  \frac{1}{r^3}\min\left\{||\theta_{1}||_{L^1(\ball{0}{sr})},||\theta_{2}||_{L^1(\ball{0}{sr})},||\theta_{3}||_{L^1(\ball{0}{sr})}\right\} \lesssim \left(r^{-3 + \frac{2}{3}}E(\overline{\ball{0}{r}})\right)^\frac{1}{2}.
\end{align*}

By standard covering arguments one can see that \[\operatorname{dim}_H \left\{x \in \ball{0}{1} : \limsup_{r\to 0} r^{-3 + \frac{2}{3}}E\left(\overline{\ball{x}{r}}\right) >0\right\}\leq 3 - \frac{2}{3}.\]
Thus for almost every point $x \in \ball{0}{1}$ we have
\[\min\left\{\theta_{1}(x),\theta_{2}(x),\theta_{3}(x)\right\} = 0.\qedhere\]
\end{proof}

\subsection{Decomposing the strain}

\begin{proof}[Proof of Lemma \ref{lemma: decomposition}]
 The proof is essentially a translation of the proofs of Capella and Otto \cite[Lemma 4 and Proposition 1]{CO12} into our setting.
 To this end, we use the ``dictionary''
  \begin{align*}
   e(u)_{11}  & \longleftrightarrow \chi_1, \\
   e(u)_{22}  & \longleftrightarrow \chi_2, \\
   e(u)_{33}  & \longleftrightarrow \chi_3, \\
          0   & \longleftrightarrow \chi_0,
  \end{align*}
  where the left-hand side shows our objects and the right-hand side shows the corresponding ones of Capella and Otto.
 The two main changes are the following:
 \begin{enumerate}
   \item In our case all relevant second mixed derivatives vanish (see Lemma \ref{lemma: wave equations}), instead of being controlled by the energy.
   Furthermore, whenever Capella and Otto refer to their ``austenitic result'', we just have to use the fact that $e(u)_{11} + e(u)_{22} + e(u)_{33} \equiv 0$.
   \item We need to check at every step that boundedness of all involved functions is preserved.
 \end{enumerate}
 
 We will briefly indicate how boundedness of all functions is ensured.
 The functions in \cite[Lemma 4]{CO12} are constructed by averaging in certain directions.
 This clearly preserves boundedness.
 The proof of \cite[Proposition 1]{CO12} works by applying pointwise linear operations to all functions, which again preserves boundedness, and by identifying certain functions as being affine, which are also bounded on the unit ball.
\end{proof}

\begin{proof}[Proof of Corollary \ref{cor: theta constant}]
 By symmetry we can assume $i=1$.
 Applying \cite[Lemma 5]{CO12} to $\theta_1$ we see that the functions
 $f_{(101)}$, $ f_{(\overline101)}$, $ f_{(110)}$ and $ f_{(1\overline10)}$
 are affine on some ball $\ball{0}{r}$ with a universal radius $r>0$.
 Thus the decomposition reduces to
   \begin{align*}
 	\begin{split}
 		\theta_1 & \equiv 0, \\
 		 \theta_{2} & = \phantom{-}  f_{(011)} + f_{(01\overline1)} + g_{2}(x),\\
 		 \theta_{3} & = - f_{(011)} - f_{(01\overline1)} + g_{3}(x)\\
 	\end{split}
 \end{align*}
  on $\ball{0}{r}$.
  As the vectors $(011)$ and $(01\overline1)$ form a basis of the plane $H(0,E_1)$, we can absorb the parts of $g_2$ depending on $x_2$ and $x_3$ into $f_{(011)}$ and $f_{(01\overline1)}$.
  Due to $\theta_1 + \theta_2 + \theta_3 = 1$ we have
  \[g_2(x) + g_3(x) \equiv 1\]
  and the decomposition simplifies to
     \begin{align*}
 	\begin{split}
 		\theta_1 & \equiv 0, \\
 		 \theta_{2} & = \phantom{-}  f_{(011)} + f_{(01\overline1)} + \lambda x_1 +1,\\
 		 \theta_{3} & = - f_{(011)} - f_{(01\overline1)} -\lambda x_1\\
 	\end{split}
 \end{align*}
 for some $\lambda \in \R$.
\end{proof}

\begin{proof}[Proof of Lemma \ref{lemma: traces}]
 Let
 \[\phi(t):= \int_{\{x_1=t\}} \frac{1}{\Leb^n(\ball{0}{1})}\chi_\ball{0}{1}(t,x') \intd \Leb^{n-1}(x')\text{ and }\phi_\delta(t):= \frac{1}{\delta}\phi\left(\frac{t}{\delta}\right).\]
 For $x\in V$ and $\delta>0$ we have that
 \[\sum_{i=1}^P \phi_\delta * f_i(x\cdot \nu_i) = \dashint_\ball{x}{\delta}F(y) \intd \Leb^n(y) \in \mathcal C, \]
 since $\ball{0}{1}$ is invariant under rotation and $\mathcal{C}$ is convex.
 By standard statements about convolutions and sequences converging in $L^1$ we get a subsequence in $\delta$, which we will not relabel, and a measurable set $T \subset \R$ such that $\phi_\delta * f_i (t) \to f_i (t)$ for all $i=1,\ldots,P$ and all $t \in T$ with $\Leb(\R\setminus T) = 0$.
 Let $\tilde \nu \in V \cap \ball{0}{1}\setminus\{0\}$ be the orthogonal projection of $\nu_i$ onto $V$ for all $i=1,\dots,n$.
 A simple calculation implies that
 \[\Leb^k\left(\{x\in V: x\cdot \nu_i \in \R\setminus T\}\right) = \Leb^k\left(\{x\in V: x\cdot \tilde \nu_i \in \R\setminus T\}\right) = 0.\]
 Thus for almost all $x\in V$ we have that
 \[\dashint_\ball{x}{\delta}F(y) \intd \Leb^n(y) \to F|_V(x):= \sum_{i=1}^P f_i(x\cdot \nu_i) \in \mathcal{C}.\qedhere\]
\end{proof}

\begin{proof}[Proof of Lemma \ref{lemma: wave equations}]
 By symmetry it is sufficient to prove the equations involving $\theta_1$.
 We calculate
 \begin{align*}
  \partial_{[111]}\partial_{[\overline 111]}  & = -\partial^2_1 + \partial_1\partial_2 + \partial_1\partial_3 - \partial_1\partial_2 + \partial^2_2 + \partial_2\partial_3 - \partial_1\partial_3  + \partial_2\partial_3  + \partial^2_3 \\
  & = -\partial_1^2 +\partial_2^2 + \partial_3^2 + 2 \partial_2\partial_3
 \end{align*}
 and, similarly,
  \[\partial_{[1\overline11]}\partial_{[ 11\overline1]}  = \partial_1^2 - \partial_2^2 - \partial_3^2 - 2 \partial_2\partial_3.\]
 Due to $\frac{1}{2}(Du + Du^T) = e(u) \in S$ we have
 \begin{align*}
  (-\partial_1^2 +\partial_2^2 + \partial_3^2)u_{1} & = -\partial_1^2 u_{1} -\partial_2\partial_1 u_{2} -\partial_3\partial_1u_{3} \\
  & = -\partial_1 \operatorname{tr} Du \\
  & = 0.
 \end{align*}
 We also know
 \[\partial_2\partial_3 u_{1} = - \partial_2\partial_1 u_{3} = \partial_1\partial_3 u_{2}  = - \partial_2\partial_3 u_{1},\]
 which gives
 \[\partial_2\partial_3 u_{1} =0.\]
 Taking a further derivative we see
 \[\partial_{[111]}\partial_{[\overline 111]} \theta_1 = \partial_{[111]}\partial_{[\overline 111]} \partial_1 u_1 = 0\]
 and
 \[ \partial_{[1\overline11]}\partial_{[ 11\overline1]} \theta_1 = 0.\]
\end{proof}

\subsection{Planarity in the case of non-trivial blow-ups}

\begin{proof}[Proof of Proposition \ref{Prop: dimension reduction}]
\textit{Step 1: Identification of a suitable plane to blow-up at.}\\
 By symmetry, we may assume $\nu = \frac{1}{\sqrt 2} (011)$.
 We use two symbols for universal radii throughout the proof.
 The radius $\tilde r>0$, which will be the radius referred to in the statement of the proposition, will stay fixed throughout the proof and its value will be chosen at the end of the proof.
 In contrast, the radius $r>\tilde r$ may decrease from line to line.
 
 As $f_{(011)} \notin VMO([-\tilde r, \tilde r])$, there exist sequences $\alpha_k \in [-\tilde r, \tilde r]$ and $\delta_k > 0$ such that
 \begin{enumerate}
  \item \[ \lim_{k\to \infty} \dashint_{(\alpha_k -\delta_k, \alpha_k + \delta_k)} \left|f_{(011)}(s) - \dashint_{(\alpha_k -\delta_k, \alpha_k + \delta_k)} f_{(011)}(\tilde s)  \intd \tilde s \right| \intd s > 0,\numberthis \label{nontrivial blowup}\]
  \item $\lim_{k\to\infty}\delta_k = 0$,
  \item $\lim_{k\to \infty} \alpha_k = \alpha \in [-\tilde r, \tilde r]$.
 \end{enumerate}

 We parametrize the plane $H\left(\alpha_k,\frac{1}{\sqrt{2}}(011)\right)$ at which we will blow-up by 
 \[X_k(\beta,\gamma):= \alpha_k \frac{1}{\sqrt 2} (011) + \left(\beta - \frac{1}{2} \alpha_k\right)\frac{1}{\sqrt 2}[1\overline11] + \left(\gamma -\frac{1}{2}\alpha_k\right)\frac{1}{\sqrt 2}[11\overline1],\]
 where $\beta,\gamma \in \R$ such that $(\beta,\gamma)\in \ball{0}{r}$.
 For $r$ small enough we have $X_k(\beta,\gamma) \in \ball{0}{1}$.
 It is straightforward to see that then we have the following relations
 \begin{align}
  X_k(\beta,\gamma) \cdot \frac{1}{\sqrt 2} (011) & = \alpha_k, \label{parametrization k (011)}\\
  X_k(\beta,\gamma) \cdot \frac{1}{\sqrt 2} (101) & = \beta, \label{parametrization k (101)}\\
  X_k(\beta,\gamma) \cdot \frac{1}{\sqrt 2} (110) & = \gamma, \label{parametrization k (110)}\\
  X_k(\beta,\gamma) \cdot \frac{1}{\sqrt 2} (01\overline1) & = \gamma - \beta, \label{parametrization k (01-1)}\\
  X_k(\beta,\gamma) \cdot \frac{1}{\sqrt 2} (\overline101) & = \alpha_k - \gamma, \label{parametrization k (-101)}\\
  X_k(\beta,\gamma) \cdot \frac{1}{\sqrt 2} (1\overline10) & = \beta-\alpha_k. \label{parametrization k (-110)}
 \end{align}
 Note that they nicely capture the combinatorics we discussed in Remark \ref{rem:combinatorics}:
 The expression $ X_k(\beta,\gamma) \cdot \nu_1^+$ depends on neither $\beta$ nor $\gamma$, while $ X_k(\beta,\gamma) \cdot \nu_1^-$ depends on both.
 Furthermore, we see that $ X_k(\beta,\gamma) \cdot \nu_i^\pm$ for $i=2,3$ depend on precisely one of the two.
 For a sketch relating $H\big(\alpha_k,\frac{1}{\sqrt 2}(011)\big)$ with the normals $\nu \in N$ see Figure \ref{fig:relation_normals_a}.
 
 In the limit we get the uniform convergence
 \[X_k(\beta,\gamma) \to X(\beta,\gamma) =\alpha \frac{1}{\sqrt 2} (011) +\left( \beta - \frac{1}{2} \alpha\right) \frac{1}{\sqrt 2} [1\overline11] +\left( \gamma - \frac{1}{2} \alpha\right)\frac{1}{\sqrt 2}[11\overline1]\numberthis\label{differentiate_gamma}\]
 and the relations with the normals turn into
 \begin{align}
  X(\beta,\gamma) \cdot \frac{1}{\sqrt 2} (011) & = \alpha, \label{parametrization (011)}\\
  X(\beta,\gamma) \cdot \frac{1}{\sqrt 2} (101) & = \beta, \label{parametrization (101)}\\
  X(\beta,\gamma) \cdot \frac{1}{\sqrt 2} (110) & = \gamma, \label{parametrization (110)}\\
  X(\beta,\gamma) \cdot \frac{1}{\sqrt 2} (01\overline1) & = \gamma - \beta, \label{parametrization (01-1)}\\
  X(\beta,\gamma) \cdot \frac{1}{\sqrt 2} (\overline101) & = \alpha - \gamma, \label{parametrization (-101)}\\
  X(\beta,\gamma) \cdot \frac{1}{\sqrt 2} (1\overline10) & = \beta - \alpha. \label{parametrization (-110)}  
 \end{align} 
  For $\nu \in N$ we define the blow-ups to be
 \begin{align*}
  \theta_i^{(k)}(\beta,\gamma; \xi) & := \theta_i(X_k(\beta,\gamma) + \delta_k \xi),\\
  f_{\nu}^{(k)}(\beta,\gamma;\xi) & := f_\nu ( X_k(\beta,\gamma) + \delta_k \xi), \\
  g_i^{(k)}(\beta,\gamma; \xi) & := g_i ( X_k(\beta,\gamma) + \delta_k \xi)
 \end{align*}
 for $\xi \in \ball{0}{1}$ and $i=1,2,3$.
 
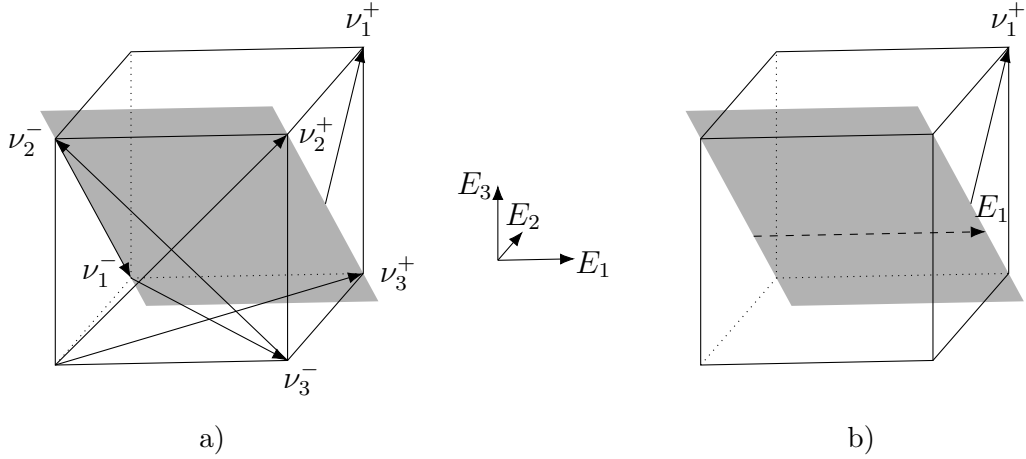
\begin{figure}
 \centering
 \begin{subfigure}{.37\linewidth}
   \centering
   \begin{tikzpicture}[scale=3, rotate around y=3]
    \fill[color=gray,opacity=.6] (0,1.2,.2) -- (1,1.2,.2) -- (1,-.2,-1.2) -- (0,-.2,-1.2);
    \draw (0,0,0) -- (0,1,0) -- (1,1,0) -- (1,0,0) -- cycle;
    \draw (0,1,0) -- (0,1,-1) -- (1,1,-1) -- (1,1,0);
    \draw (1,1,-1) -- (1,0,-1) -- (1,0,0);
    \draw[dotted] (0,0,0) -- (0,0,-1) -- (0,1,-1);
    \draw[dotted] (0,0,-1) -- (1,0,-1);
    
    \draw[-{Latex[length=2mm]}] (0,0,0) -- (1,0,-1);
    \draw[-{Latex[length=2mm]}] (0,0,-1) -- (1,0,0);
    \draw[-{Latex[length=2mm]}] (1,.5,-.5) -- (1,1,-1);
    \draw[-{Latex[length=2mm]}] (0,1,0) -- (0,0,-1);
    \draw[-{Latex[length=2mm]}] (0,0,0) -- (1,1,0);
    \draw[-{Latex[length=2mm]}] (1,0,0) -- (0,1,0);

    \node at (1.13,0,.215) {$\nu_3^-$};
    \node at (1.15,0,-1) {$\nu_3^+$};
    \node at (-.1,0.075,-.9) {$\nu_1^-$};
    \node at (1,1.125,-1) {$\nu_1^+$};
    \node at (1.125,1,0) {$\nu_2^+$};
    \node at (-.125,1,0) {$\nu_2^-$};
  \end{tikzpicture}
  \caption{}
  \label{fig:relation_normals_a}
 \end{subfigure}
 \begin{subfigure}{.15\linewidth}
  \begin{tikzpicture}[rotate around y=3]
   \draw[-{Latex[length=2mm]}] (0,0,0) --  (1,0,0);
   \node at (1.3,0,.2) {$E_1$};
   \draw[-{Latex[length=2mm]}] (0,0,0) --  (0,0,-1);
   \node at (0,.2,-1) {$E_2$};
   \draw[-{Latex[length=2mm]}] (0,0,0) --  (0,1,0);
   \node at (-.3,1,0) {$E_3$};
  \end{tikzpicture}
 \end{subfigure}
 \begin{subfigure}{.37\linewidth}
   \centering
   \begin{tikzpicture}[scale=3, rotate around y=3]
    \fill[color=gray,opacity=.6] (0,1.2,.2) -- (1,1.2,.2) -- (1,-.2,-1.2) -- (0,-.2,-1.2);
    \draw (0,0,0) -- (0,1,0) -- (1,1,0) -- (1,0,0) -- cycle;
    \draw (0,1,0) -- (0,1,-1) -- (1,1,-1) -- (1,1,0);
    \draw (1,1,-1) -- (1,0,-1) -- (1,0,0);
    \draw[dotted] (0,0,0) -- (0,0,-1) -- (0,1,-1);
    \draw[dotted] (0,0,-1) -- (1,0,-1);
    
    \draw[-{Latex[length=2mm]}] (1,.5,-.5) -- (1,1,-1);
    \draw[-{Latex[length=2mm]},dashed] (0,.3,-.7) -- (1,.3,-.7); 
    \node at (1.025,.4,-.7) {$E_1$};
    \node at (1,1.125,-1) {$\nu_1^+$};
    \node at (1.13,0,.215) {\phantom{$\nu_3^-$}};
  \end{tikzpicture}
  \caption{}
  \label{fig:relation_normals_b}
 \end{subfigure}
 \caption{a) Sketch relating planes $H(\tilde \alpha,(011))$ for $\tilde \alpha \in \R$ with all normals $\nu\in N$.
 b) Planes $H(\tilde \alpha,(011))$ for $\tilde \alpha \in \R$ contain lines parallel to $E_1$.}
 \label{fig:relation_normals}
\end{figure}
 
\textit{Step 2: There exists a subsequence, which we will not relabel, such that for almost all $(\beta,\gamma) \in \ball{0}{r}$ we have
\begin{flalign*}
  ||f_{\nu}^{(k)}(\beta,\gamma;\bullet) - f_\nu\circ X (\beta,\gamma) ||_{L^1(\ball{0}{1})} & \to 0 & &  \text { for } \nu \in N\setminus \left\{\frac{1}{\sqrt{2}} (011)\right\},& & \\
  ||g_i^{(k)}(\beta,\gamma; \bullet) - g_i\circ X (\beta,\gamma)||_{L^1(\ball{0}{1})} & \to 0  &  & \text { for } i=1,2,3,& &\\
  \int_\ball{0}{1} \psi\left(\left(-f_{(011)}^{(k)},f_{(011)}^{(k)}\right)(\beta,\gamma;\xi)\right)\mathrm{d} \xi & \to \int_\ball{0}{1} \psi(\hat f) \intd \mu(\hat f) & & \text{ for all } \psi \in C(\R^2).& & \numberthis \label{cumulative_Young}
\end{flalign*}
 Additionally, the probability measure $\mu$ on $\R^2$ is not a Dirac measure.}\\
 The combinatorics behind the first convergence can be found in Figure \ref{fig:relation_normals_a}.
 
 For $\nu \in N\setminus \{\frac{1}{\sqrt{2}} (011)\}$ we have
 \begin{align*}
  & \quad \int_\ball{0}{r} \dashint_\ball{0}{1} \left|f_\nu^{(k)}(\beta,\gamma;\xi) - f_\nu \circ X(\beta,\gamma) \right| \intd \xi \intd (\beta,\gamma) \\
  & = \int_\ball{0}{r} \dashint_\ball{0}{1} \left|f_\nu \left(\nu \cdot X_k(\beta,\gamma) + \delta_k \nu \cdot\xi\right) - f_\nu (\nu \cdot X(\beta,\gamma)) \right| \intd \xi \intd (\beta,\gamma) \\
  & \lesssim \int_\ball{0}{r} \dashint_{-1}^1 \left| f_\nu(\nu \cdot X_k(\beta,\gamma) + \delta_k s ) - f_\nu (\nu\cdot X(\beta,\gamma)) \right| \intd s \intd (\beta,\gamma).
 \end{align*}
 As $\nu \cdot X_k(\beta,\gamma)$ and $\nu \cdot X(\beta,\gamma)$ depend on at least $\beta$ or $\gamma$, see equations \eqref{parametrization k (101)}-\eqref{parametrization k (-110)} and \eqref{parametrization (101)}-\eqref{parametrization (-110)}, and we have the uniform convergence $X_k \to X$, we can apply Lemma \ref{lemma: shifts don't matter} to deduce that the integral in the last line vanishes in the limit.
 Passing to a subsequence, we get strong convergence in $\xi$ for almost all $(\beta,\gamma)\in \ball{0}{r}$.

 Also, for $i=1,2,3$ we have $g_i^{(k)}(\beta,\gamma;\xi) \to g_i \circ X(\beta,\gamma)$ pointwise and in $L^1$ by continuity of affine functions.
 
 Due to the fact that $X_k(\beta,\gamma)\cdot \frac{1}{\sqrt 2} (011) = \alpha_k$ we see that $f^{(k)}_{(011)}$ does not depend on $\beta$ and $\gamma$.
 Hence we may drop them in equation \eqref{cumulative_Young}.
 As $f_{(011)}$ is a bounded function, the sequence of push-forward measures defined by the left-hand side have uniformly bounded supports.
 Consequently, there exists a limiting probability measure $\mu$ such that along a subsequence we have
 \[\int_\ball{0}{1} \psi\left(\left(-f_{(011)}^{(k)},f_{(011)}^{(k)}\right)(\xi)\right)\intd \xi  \to \int_\ball{0}{1} \psi(\hat f) \intd \mu(\hat f)\]
 for all $\psi \in C(\R^2)$.
 Finally, if we had $\mu = \delta_{\hat f}$, then testing this convergence with the function $\psi(\hat g) =  |\hat g _2 - \hat f_2|$ we would see that
 \[\dashint_\ball{0}{1} \left| f_{(011)}^{(k)}(\xi) - \dashint_\ball{0}{1} f_{(011)}^{(k)}(\zeta) \intd \zeta \right| \intd \xi \lesssim \dashint_\ball{0}{1} | f_{(011)}^{(k)}(\xi) - \hat f_2| \intd \xi \to 0,\]
  because in $L^1$ the average is almost the constant closest to a function.
 However, this would contradict the convergence to a strictly positive number \eqref{nontrivial blowup} after undoing the rescaling.

\textit{Step 3: For all $(\beta,\gamma)$ as in Step 2 we have
  \[ \int_\ball{0}{1} \psi\left(\theta^{(k)}(\beta,\gamma;\xi)\right)\intd \xi  \to \int_\ball{0}{1} \psi\left((\theta_1\circ X,\hat f + (z_2,z_3))(\beta,\gamma)\right) \intd \mu(\hat f)\numberthis \label{cumulative_Young_support}\]
  for all $\psi \in C_0(\R^3)$ and where $z_2$, $z_3$ are defined by equations \eqref{shift_first_component} and \eqref{shift_second_component}.
  The measure $\bar \mu$ defined by the right-hand side is supported on $\tilde \K$, see definition \eqref{def_tilde_K}.}\\
 The previous calculations immediately give that $\theta_1^{(k)}$ converges strongly in $\xi$ to \begin{align}\label{decomposition_lower_dim}
  \theta_1\circ X(\beta,\gamma) = \left(f_{(101)} + f_{(\overline101)} - f_{(110)} - f_{(1\overline10)} + g_1\right) \circ X(\beta,\gamma).
 \end{align}
 Similarly, the blow-ups $(\theta_2^{(k)} + f^{(k)}_{(011)})(\beta,\gamma;\xi)$ and $(\theta_3^{(k)} - f^{(k)}_{(011)})(\beta,\gamma;\xi)$ converge strongly to
 \[z_2(\beta,\gamma) := \left(f_{(110)} + f_{(1\overline10)} - f_{(01\overline1)} + g_2\right) \circ X(\beta,\gamma),\numberthis \label{shift_first_component}\] 
 resp.\
 \[z_3(\beta,\gamma) := \left( f_{(01\overline1)} - f_{(101)} -f_{(\overline101)} + g_3 \right) \circ X(\beta,\gamma).\numberthis \label{shift_second_component}\]
 
 As the required convergence \eqref{cumulative_Young_support} is induced by a topology, we only have to identify the limit along subsequences, which may depend on $\beta$ and $\gamma$, of arbitrary subsequences.
 Thus we may extract a subsequence to obtain pointwise convergence a.e.\ of the sequences $\theta_1^{(k)}$, $(\theta_2^{(k)} + f^{(k)}_{(011)})(\beta,\gamma;\xi)$ and $(\theta_3^{(k)} - f^{(k)}_{(011)})(\beta,\gamma;\xi)$.
 Applying both Egoroff's and Lusin's Theorem, these convergences can be taken to be uniform and the limits to be continuous on sets of almost full measure.
 Consequently we get that
 \[ \int_\ball{0}{1} \psi\left(\theta^{(k)}(\beta,\gamma;\xi)\right)\intd \xi  \to \int_\ball{0}{1} \psi\left((\theta_1\circ X,\hat f + (0,z_2,z_3))(\beta,\gamma)\right) \intd \mu(\hat f)\]
 for all $\psi \in C_0(\R^3)$.
 Testing with $\psi = \dist(\bullet, \tilde \K)$ we see that the measure $\bar \mu$ has support in $\tilde \K$.

\textit{Step 4: \texorpdfstring{We have $\theta_1\circ X = b\chi_B$ for some $0<b<1$ and some measurable set $B \subset \R^2$ almost everywhere. Furthermore, the shift $(z_2,z_3)(\beta,\gamma)$ is constant on $B$ almost everywhere.}{Theta 1 is two valued, shift z is constant.}}\\
 Note that what we claim to prove in Step 4 is an empty statement if $\theta_1 \circ X \equiv 0$ a.e.\ in $\ball{0}{r}$.
 Let
 \[B:= \{(\beta,\gamma) \in \ball{0}{r}:\, \theta_1\circ X(\beta,\gamma) > 0 \text{ and the conclusion of Step 2 holds}\}.\]
 Let $T_z$ for $z\in \R^2$ be the translation operator acting on measures $\tilde \mu$ on $\R^2$ via the formula $(T_z \tilde \mu)(A) = \tilde \mu (A-z)$.
 Due to the support of $\bar \mu$ lying in $\tilde \K$ and $\tilde \K \cap \{\theta_1 = c\} = \{(c,0,1-c), (c,1-c,0)\}$ for $0<c\leq1$, see Figure \ref{fig:rigidity_K}, we have for any $(\beta,\gamma) \in B$ that
 \[\supp T_{-(z_2,z_3)(\beta,\gamma)}\mu \subset \left\{\left(0,1-\theta_1\circ X(\beta,\gamma)\right), \left(1-\theta_1\circ X(\beta,\gamma),0\right)\right\}.\]

 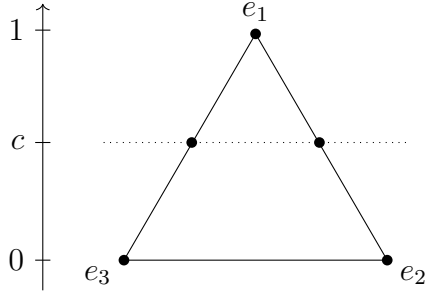
\begin{figure}
    \centering
      \begin{tikzpicture}[scale=2]
% 	  \node at (-1,1) {b)};
	  \fill (0.42,0.28) circle (1pt);
	  \fill (-0.42,0.28) circle (1pt);
	  \fill (90:1) circle (1pt);
	  \node at (90:1.15) {$e_1$};
	  \fill (210:1) circle (1pt);
	  \node at (210:1.2) {$e_3$};
	  \fill (330:1) circle (1pt);
	  \node at (330:1.2) {$e_2$};
	  \draw (90:1) -- (210:1);
	  \draw (210:1)-- (330:1);
	  \draw (330:1) -- (90:1);
	  \draw[dotted] (-1,0.28)  -- (1,0.28);
	  \draw[->] (-1.4, -0.7 ) -- (-1.4, 1.2);
	  \draw (-1.35, -0.5) -- (-1.45, -0.5) node[left] {0};
	  \draw (-1.35, 0.28) -- (-1.45, 0.28)  node[left]{$c$};
	  \draw (-1.35, 1.025) -- (-1.45, 1.025) node[left] {1};
      \end{tikzpicture}
     \caption{The dotted line $\{\theta_1 = c\}$ for $0<c\leq 1$ intersects $\K$ in the two points $ce_1 + (1-c)e_3$ and $ce_1 + (1-c)e_2$.}
      \label{fig:rigidity_K} 
 \end{figure}
 
 Thus we get that
 \[\mu = \lambda \delta_{\hat f} + (1-\lambda) \delta_{\hat g}\]
 with $0 <\lambda <1$ and $\hat f \neq \hat g$ since $\mu$ is not a Dirac measure by Step 2.
 Consequently, we get
 \[\{\hat f,\hat g\} - (z_2, z_3)(\beta,\gamma) = \left\{\left(0,1-\theta_1\circ X(\beta,\gamma)\right), \left(1-\theta_1\circ X(\beta,\gamma),0\right)\right\}.\]
 Both sets have the same diameter, which gives
 \[2\left(1-\theta_1\circ X(\beta,\gamma)\right) = |\hat f - \hat g|>0.\]
 Consequently we have $\theta_1\circ X(\beta,\gamma) < 1$ a.e.
 Furthermore, as $\mu$ is independent of $(\beta,\gamma)$  also $\hat f$ and $\hat g$ are, which implies that $\theta_1\circ X$ is constant on $B$.
 
 To see that $(z_2,z_3)$ is constant on $B$ note that the above implies
 \[\{\hat f,\hat g\} - (z_2, z_3)(\beta,\gamma) = \{\hat f,\hat g\} - (z_2, z_3)(\tilde \beta, \tilde \gamma)\]
 for $(\beta,\gamma), (\tilde \beta, \tilde \gamma) \in B$.
 As a non-empty set which is invariant under a single, non-vanishing shift has to at least be countably infinite, we see that $(z_2,z_3)$ has to be constant on $B$.

\textit{Step 5: If we have $|B|=0$, i.e., $\theta_1 \circ X(\beta,\gamma) \equiv 0$ for almost all $|(\beta,\gamma)| < r$, then the solution $u$ is a two-variant configuration.}\\
 As the plane $H(\alpha,(011))$ contains plenty of lines parallel to $E_1$, see Figure \ref{fig:relation_normals_b}, an application of Lemma \ref{lemma: almost maxima on transversal lines, 3D} ensures that $\theta_1 \equiv 0$ on $\ball{0}{r}$.
 Corollary \ref{cor: theta constant} then implies that we are dealing with a two-variant configuration.
 
\textit{Step 6: If $|B| > 0$, then there exists $d \in \{[111], [\overline1 11], [1\overline 1 1], [1 1 \overline1]\}$ such that the configuration is planar with respect to $d$.}\\
 By the decomposition of $\theta_1\circ X(\beta,\gamma)$, see equation \eqref{decomposition_lower_dim}, and its interplay with the coordinates $X$, see equations \eqref{parametrization (011)}-\eqref{parametrization (-110)}, we have
 \begin{align*}
  \theta_1 \circ X(\beta,\gamma)  & = f_{(101)}(\beta) + f_{(\overline101)}(\alpha-\gamma) - f_{(110)}(\gamma) - f_{(1\overline10)}(\beta - \alpha) + \lambda_1 \beta + \lambda_2 \gamma + c\\
  & = F_1(\beta) + F_2(\gamma), 
 \end{align*}
 where $\lambda_1$, $\lambda_2$, $c\in \R$ and
 \begin{align*}
  F_1(\beta) & := f_{(101)}(\beta)- f_{(1\overline10)}(\beta - \alpha) + \lambda_1 \beta,\\ F_2(\gamma) & := f_{(\overline101)}(\alpha-\gamma) - f_{(110)}(\gamma) + \lambda_2 \gamma + c.
 \end{align*}
 
 As by Step 4 the function $\theta_1\circ X(\beta,\gamma)$ takes at most two values almost everywhere we have that either $F_1$ is constant or $F_2$ is constant almost everywhere.
 
 We only deal with the case in which $F_2$ is constant.
 The argument for the other one works analogously.
 Consequently, we get a measurable set $D \subset (-r,r)$ such that $|D|>0$ and  $D\times (-r,r) \subset B$, see Figure \ref{fig:2d_differentiate}.

 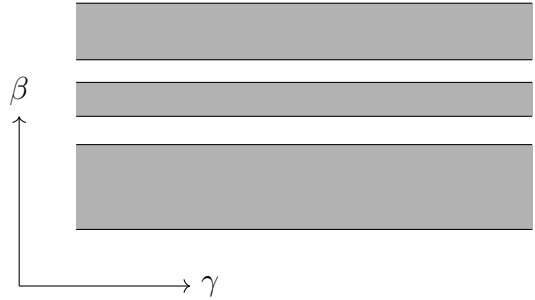
\begin{figure}
  \centering
  \begin{tikzpicture}[scale=1.5]
   \begin{scope}[shift={(-.5,-.5)}]
    \draw[->] (0,0) -- (0,1.5) node[above] {$\beta$};
    \draw[->] (0,0) -- (1.5,0) node[right] {$\gamma$};
   \end{scope}

   \fill[color=gray, opacity=.6] (0,0) -- (4,0) -- (4,.75) -- (0,.75) -- cycle;
   \draw (0,0) -- (4,0);
   \draw (0,.75) -- (4,.75);
   
   \fill[color=gray, opacity=.6] (0,1) -- (4,1) -- (4,1.3) -- (0,1.3) -- cycle;
   \draw (0,1) -- (4,1);
   \draw (0,1.3) -- (4,1.3);
   
   \fill[color=gray, opacity=.6] (0,1.5) -- (4,1.5) -- (4,2) -- (0,2) -- cycle;
   \draw (0,1.5) -- (4,1.5);
   \draw (0,2) -- (4,2);
  \end{tikzpicture}
  \caption{Sketch of the set $D\times(r,r)$. We take differences of the constant shifts $(z_2,z_3)$ in $\gamma$ and in $\beta$ in order to isolate a single function $f_\nu$ by Remark \ref{rem:combinatorics} and prove that it is affine.}
  \label{fig:2d_differentiate}
 \end{figure}

 We will follow the notation of Capella and Otto \cite{CO12} in writing discrete derivatives of a function $\phi(\gamma)$ as
 \[\partial_\gamma^h \phi(\gamma) : =  \phi(\gamma + h) - \phi(\gamma).\numberthis \label{discrete_derivatives}\]
 
 We proved in Step 4 that the shift $(z_2,z_3)$ is constant almost everywhere on $B$.
 Thus we get for $h\in (-r,r)$, $\beta \in D$ and almost all $\gamma \in (-r,r)$ that
 \begin{align*}
  0 & \overset{\phantom{\eqref{parametrization (110)}-\eqref{parametrization (-110)}}}{=} \partial_\gamma^h z_2 \circ X(\beta,\gamma) \overset{\eqref{shift_first_component}}{=} \partial_\gamma^h \left(f_{(110)} + f_{(1\overline10)} - f_{(01\overline1)} + g_2 \right)\circ X(\beta,\gamma) \\
  & \overset{\eqref{parametrization (110)} - \eqref{parametrization (-110)}}{=} \partial_\gamma^h \left( f_{(110)}(\gamma) + f_{(1\overline10)}(\beta - \alpha) - f_{(01\overline1)}(\gamma-\beta)\right) + \partial_\gamma^h g_2 \circ X(\beta,\gamma)\\
  & \overset{\phantom{\eqref{parametrization (110)} - \eqref{parametrization (-110)}}}{=} \partial_\gamma^h \left( f_{(110)}(\gamma) - f_{(01\overline1)}(\gamma-\beta)\right) + \partial_\gamma^h g_2 \circ X(\beta,\gamma) \numberthis\label{affine argument}.
 \end{align*}
 The fact that $g_2$ is affine implies that $\partial_\gamma^h g_2 \circ X$ is independent of $\beta$.
 Thus, ``differentiating'' again under the constraint $\beta$, $\tilde \beta \in D$ we get
 \[0 =  \partial_\gamma^h  f_{(01\overline1)}(\gamma - \beta) - \partial_\gamma^h f_{(01\overline1)}(\gamma - \tilde \beta).\]
 Even though in general we have $D\neq (-r,r)$, we can still apply \cite[Lemma 7]{CO12} due to $|D|>0$ to get
 \[\partial^h\partial^{\tilde h} f_{(01\overline1)}(t) = 0\]
 for almost all $t \in (-r,r)$ and shifts $h,\tilde h \in (-r,r)$.
 Consequently, the function $f_{(01\overline1)}$ is affine, see e.g.\ Lemma \ref{lemma:almost_affine}.
 Referring back to equation \eqref{affine argument} we see that also $f_{(110)}$ is affine.
 
 The upshot is that the decomposition for $\theta_2$ can be re-written as
 \[\theta_2 = -f_{(011)} + f_{(1\overline10)} + \tilde g_2 \numberthis \label{reduced_decomposition_1}\]
 in $\ball{0}{r}$
 with the affine function
 $\tilde g_2 := f_{(110)} - f_{(01\overline1)} + g_2.$
 By equation \eqref{affine argument} it furthermore satisfies
 \[\partial_\gamma \tilde g_2 \circ X = 0 \text{ in } \ball{0}{r}.\]
 In the standard basis of $\R^3$ this translates to
 \[\partial_{[11\overline1]} \tilde  g_2  = 0 \text{ on } \ball{\frac{\alpha}{\sqrt2}(011)}{r},\]
 since $\partial_\gamma$ corresponds to differentiating in the direction of $[11\overline1]$ by equation \eqref{differentiate_gamma}.
 At last we are in the position to choose $\tilde r := \frac{1}{2}r$, so that we get
 \[\partial_{[11\overline1]} \tilde  g_2   = 0 \text{ on } \ball{0}{\tilde r}.\]
% Using $\partial_{[11\overline1]} f_{(1\overline10)} = \partial_{[11\overline1]} f_{(011)} = 0$ we see that
% \[\partial_{[11\overline1]} \theta_2 = 0\]
% in $\ball{0}{r}$.
 
 The analogue of \eqref{affine argument} using $z_3$ rather than $z_2$ gives that $f_{(\overline101)}$ is affine and that we may find an affine function $\tilde g_3$ with $\partial_{[11\overline1]} \tilde g_3= 0$ such that
 \[\theta_3 = f_{(011)} - f_{(101)}  + \tilde g_3 \numberthis \label{reduced_decomposition_2}\]
% \[\partial_{[11\overline1]} \theta_3 = \partial_{[11\overline1]} \left(f_{(01\overline1)}- f_{(\overline101)} + g_3 \right) = 0\]
 in $\ball{0}{r}$.
 
 The relation $\theta_1 + \theta_2 + \theta_3 = 1$ and the two vanishing derivatives $\partial_{[11\overline1]} \theta_2 = \partial_{[11\overline1]} \theta_3 = 0$ imply $\partial_{[11\overline1]} \theta_1 = 0$.
 Therefore the affine function $\tilde g_1 := f_{(\overline101)} - f_{(110)} +g_1$ satisfies
 \[\partial_{[11\overline1]} \tilde g_1 = \partial_{[11\overline1]} \theta_1 =  0 \]
 on $\ball{0}{r}$ as well and we get the decomposition
 \[\theta_1 = f_{(101)} - f_{(1\overline10)} +  \tilde g_1.\numberthis \label{reduced_decomposition_3}\]
 Equations \eqref{reduced_decomposition_1}-\eqref{reduced_decomposition_3} together with the affine function $\tilde g_i$ being independent of the $[11\overline1]$-direction constitute planarity of the configuration, see Definition \ref{def:planar}.
\end{proof}

\begin{proof}[Proof of Lemma \ref{lemma: almost maxima on transversal lines}]
  Without loss of generality, we may assume
   \[\essinf_{x_1,x_2 \in [0,1]} f(x_1) + g(x_2) = c = 0.\]
  
\textit{Step 1: We have $\essinf f + \essinf g \geq 0$.}\\
  Let $\delta >0$.
  We know that
  \[\left|\left\{t \in (0,1) :  f(t) < \essinf f + \frac{\delta}{2}\right\}\right| >0 \]
  and 
  \[\left|\left\{t \in (0,1) :  g(t) < \essinf g + \frac{\delta}{2}\right\}\right| >0.\]
  Consequently, we have that
  \[\left|\left\{x \in (0,1)^2 : 0 \leq f(x_1) + g(x_2) < \essinf f + \essinf g + \delta \right\}\right| >0.\]
  As a result we know $-\delta \leq \essinf f + \essinf g$ for all $\delta >0$, which implies the claim.
  
\textit{Step 2: Statement 1 implies statement 3.}\\
  For almost all $x \in (0,1)^2$ we know that
  \[ \eps \geq f(x_1) + g(x_2) \geq \essinf f + g(x_2) \geq \essinf f + \essinf g \geq 0.\]
  In particular, we know 
  \[\essinf f + \essinf g \leq \eps.\]
  By Fubini's Theorem there exists an $x_2 \in (0,1)$ such that we
  have 
  \[\eps \geq f(x_1) + g(x_2) \geq \essinf f + g(x_2)\geq 0\]
  for almost all $x_1 \in (0,1)$.
  Thus we see
  \[ f(x_1) -\essinf f = f(x_1) + g(x_2) - ( \essinf f + g(x_2))\leq \eps. \]
  A similar argument ensures $g \leq \essinf g + \eps$.
  
% \textit{Step 3: If \[\left|\{(x_1,x_2) \in [0,1]^2: f(x_1) + g(x_2) = \operatorname{ess\ inf}_{(\tilde x_1,\tilde x_2) \in [0,1]^2 } f(\tilde x_1) + g(\tilde x_2)\}\right|>0,\]
%  then we also have $\left|\{x_1\in [0,1]: f =\operatorname{ess\ inf} f \}\right|>0$ and $\left|\{x_2 \in [0,1] : g(x_2) =\operatorname{ess\ inf} g \}\right|>0$.}\\
%   For almost all $(x_1,x_2) \in \{f + g = \operatorname{ess\ inf} f + g\}$ we get
%   \[ 0 = f(x_1) + g(x_2) \geq \essinf f + g(x_2) \geq \essinf f + \essinf g = 0.\]
%   By again a similar argument as above we have
%   $f = \essinf f$ and $g=\essinf g$ on subsets of positive measure.
%   
\textit{Step 3: Conclusion.}\\
  The proof for the implication ``2 $\implies$ 3 '' is very similar to Step 2.
  Lastly, if $\eps = 0$, the implications ``3 $\implies$ 1, 2'' are trivial.
\end{proof}

\begin{proof}[Proof of Lemma \ref{lemma: almost maxima on transversal lines, 3D}]
 The radius $r>0$ is only required to ensure that $P \subset \ball{0}{1}$.
 We may thus translate, re-scale and use the symmetries of the problem to only work in the case $i=1$, $x_0 = 0$, $I=(-1,1)$. 
 These additional assumptions imply
 \[\nu \cdot l(I) = \sqrt{2}E_i\cdot \nu (-1,1) = (-1,1)\]
 for $\nu \in N_2 \cup N_3$ and, consequently, $P = \bigcap_{\nu \in N_{2} \cup N_{3}} \{x \in \R^3: |\nu \cdot x| < 1\}$.
 Furthermore, we only have to deal with the case $\theta_1\circ l \leq \epsi$, as the other one can be dealt with by working with $\tilde \theta_1 := 1-\theta_1$.
 We remind the reader that Figure \ref{fig:polyhedron_proof} depicts the general strategy of the proof.
 
\textit{Step 1: Extend $0 \leq \theta_1 \leq \eps$ to the plane $H\big(0,\frac{1}{2}(011)\big)$.}\\
 We parametrize the plane via
 \[X(\alpha,\beta):=\alpha \frac{1}{\sqrt 2}[11\overline1] + \beta \frac{1}{\sqrt 2}[1\overline11] .\]
 By the decomposition into one-dimensional functions, see Lemma \ref{lemma: decomposition}, and the existence of traces, see Lemma \ref{lemma: traces}, we have for almost all $(\alpha,\beta) \in (-1,1)^2 $ that
 \[0\leq \theta_1 \circ X \left(\alpha , \beta  \right) = f_{(\overline101)}(-\alpha) - f_{(110)}( \alpha) + f_{(101)}(\beta) - f_{(1\overline10)}(\beta) \leq 1.\]
 As $X \cdot l(t) = t (1,1)$ parametrizes the diagonal, the assumption \eqref{max trans lines, main ass} of $\theta_1$ almost achieving its minimum along $l$ and the two-dimensional statement Lemma \ref{lemma: almost maxima on transversal lines} imply that for almost all points $\alpha,\beta \in (-1,1)$ we have
 \begin{align*}
  f_{(\overline101)}(-\alpha) - f_{(110)}( \alpha)  \leq & \operatorname{ess\ inf}_{\tilde \alpha} \left( f_{(\overline101)}(-\tilde \alpha) - f_{(110)}(\tilde \alpha) \right) + \eps, \\
  f_{(101)}(\beta) - f_{(1\overline10)}(\beta)  \leq & \operatorname{ess\ inf}_{\tilde \beta} \left(f_{(101)}\left(\tilde \beta\right) - f_{(1\overline10)}\left(\tilde \beta\right) \right) + \eps.
 \end{align*}
 Consequently, we have
 \begin{align*}
  \left|f_{(\overline101)}(-\alpha) - f_{(110)}( \alpha) - \left(f_{(\overline101)}(-\tilde \alpha) - f_{(110)}(\tilde \alpha) \right) \right| & \leq \eps,\\
  \left|f_{(101)}(\beta) - f_{(1\overline10)}(\beta) - \left( f_{(101)}(\tilde \beta) - f_{(1\overline10)}(\tilde \beta) \right) \right| & \leq \eps
 \end{align*}
 for all $\alpha,\tilde \alpha, \beta, \tilde \beta \in (-1,1)$.
 These inequalities together with the assumption \eqref{max trans lines, main ass} imply for almost all $(\alpha,\beta) \in (-1,1)^2$ that
 \[0 \leq \theta_1 \circ X \left(\alpha, \beta\right) \leq 3\eps.\]
 Changing coordinates to $y := \frac{1}{2} \left(\alpha + \beta\right)$, $z := \frac{1}{2} \left(\alpha - \beta\right)$ we see that
 \[0 \leq \theta_1\left(\sqrt 2 (y, z,- z)\right) \leq 3\eps\]
 for almost all $(y,z)\in \R^2$ with $y + z$, $y - z \in (-1,1)$.
 
\textit{Step 2: Prove inequality \eqref{close to maximum on 3D set} on a suitable subset of $P$.}\\
 Fubini's theorem implies that for almost all $z \in (-1,1)$ we have 
 \[0\leq \theta_1\left(\sqrt 2(y,z,-z)\right) \leq 3\eps\numberthis \label{offset_close_to_max}\] for almost all $y \in \R$ with $y + z$, $y - z \in (-1,1)$.
 Furthermore, this condition for $y$ is equivalent to $y \in I(z) := \left(-1 + |z|, 1 - |z|\right)$.
%   \item The decomposition \eqref{decomposition_max_trans_lines} together with $\theta_1 \leq 1$ holds $\mathcal{H}^2$-almost everywhere on the set $P \cap \left\{\frac{1}{\sqrt 2}(01\overline1)\cdot x =  2 z \right\}$.
 We may thus repeat the above argument for almost all $z \in (-1,1)$ with $\tilde l(t) = \sqrt{2}t E_i + \sqrt{2}(0,z,-z)$ and the plane $H\big(2z,\frac{1}{\sqrt 2}(01\overline1)\big)$ to see that
 \[0 \leq \theta_1\left(\sqrt 2 (0,z,-z) +  \alpha \frac{1}{\sqrt 2}[111] +  \beta \frac{1}{\sqrt 2}[\overline111]\right) \leq 6\eps\]
 for almost all $\alpha,\beta \in I(z)$. 
 Due to measurability of $\theta_1$ another application of Fubini's theorem implies that we have the above inequality for almost all $(z,\alpha,\beta) \in \R^3$ with $z\in (-1,1)$ and $\alpha, \beta \in I(z)$.
 
 The proof so far ensured that the argument of $\theta_1$ in this inequality lies in $P$.
 We now need to prove that we did not miss significant parts.
 
%  
%  We use the trace of $\theta_1$ on the plane 
%  by observing that for $\alpha,\beta$ with $-z + \alpha$, $z + \alpha$, $-z + \beta$, $-z - \beta \in (-1,1)$ we have
%  \begin{align*}
%   & \quad \theta_1\left(\sqrt 2 (0,z,-z) + \alpha \frac{1}{\sqrt 2}[111] + \beta \frac{1}{\sqrt 2}[\overline111]\right)\\
%   & = f_{(101)}(- z + \alpha) - f_{(110)}( z + \alpha) + f_{(\overline101)}(-z + \beta) + f_{(1\overline10)}(-z - \beta).  
%  \end{align*}
%  The conditions on $\alpha$ and $\beta$ can be rewritten as $\alpha,\beta \in I(z)$.
%  As we have $\sqrt 2 (0,z,-z) + y \frac{1}{\sqrt 2}[111] + y \frac{1}{\sqrt 2}[\overline111] = \sqrt 2(y,z,-z)$ and the estimate \eqref{offset_close_to_max}, yet another application of Lemma \ref{lemma: almost maxima on transversal lines} implies that
%  \begin{align*}\label{trans_lines_1D_almost_constant}
%   \begin{split}
%     f_{(101)}(- z + \alpha) - f_{(110)}( z + \alpha) & \geq \esssup_{\tilde \alpha \in I(z)} f_{(101)}(- z + \tilde \alpha) - f_{(110)}( z + \tilde \alpha) - 3 \eps,\\
%     f_{(\overline101)}(-z + \beta) + f_{(1\overline10)}(-z - \beta) & \geq \esssup_{\tilde \beta \in I(z)} f_{(\overline101)}(-z + \tilde \beta) + f_{(1\overline10)}(-z - \tilde \beta) - 3 \eps
%   \end{split}
%  \end{align*}
%  for almost all $z \in (-1,1)$ and for almost all $\alpha$, $\beta \in I(z)$.
%  As a result we get for almost all $z \in (-1,1)$ and almost all $\alpha$, $\beta \in I(z)$ that

\textit{Step 3: Prove that the estimate $0 \leq \theta_1 (x) \leq 6\eps$ holds for $x\in P$.}\\
 To this end, we exploit that $\overbar P = \bigcap_{\nu \in N_{2} \cup N_{3}} \{x \in \R^3: |\nu \cdot x| \leq 1\}$ is a three-dimensional polyhedron.
 A fundamental result in the theory of bounded, non-empty polyhedra, see Br{\o}ndsted \cite[Corollary 8.7 and Theorem 7.2]{brondsted2012introduction}, is that they can be represented as the convex hull of their extremal points.
 Following Br{\o}ndsted \cite[Chapter 1, \S5]{brondsted2012introduction}, extremal points $x\in \overbar P$ are defined to leave $\overbar P\setminus\{x\}$ still convex, see also Figure \ref{fig:extremal_points}.
 Thus, in order to prove $0 \leq \theta_1 (x) \leq 6\eps$ holds for $x\in P$ we only have to argue that the closure of the set
 \[Q:= \left\{\sqrt 2 (0,z,-z) + \alpha \frac{1}{\sqrt 2}[111] + \beta \frac{1}{\sqrt 2}[\overline111]: z \in (-1,1)\text{ and } \alpha ,\beta \in I(z) \right\}\]
 contains all extremal points and is convex.
 
  \begin{figure}
   \centering
  \begin{tikzpicture}
   \begin{scope}[rotate around y=14,scale=2.5]
%     \draw (1,0,0) -- (0,1,1) -- (-1,0,0) -- (0,-1,-1) -- cycle;
%     \draw (1,0,0) -- (0,1,-1) -- (-1,0,0) -- (0,-1,1) -- cycle;
      \draw (0,-1,1) -- (0,1,1) -- (0,1,-1);
      \draw[dotted] (0,1,-1) -- (0,-1,-1) -- (0,-1,1);
      \draw (1,0,0) -- (0,1,1) -- (-1,0,0);
      \draw (1,0,0) -- (0,-1,1) -- (-1,0,0);
      \draw (1,0,0) -- (0,1,-1) -- (-1,0,0);
      \draw[dotted]  (1,0,0) -- (0,-1,-1) -- (-1,0,0);
      
      \fill (1,0,0) circle (1pt);
      \node at (1.25,0,0) {$\sqrt{2} E_1$};
      \fill (-1,0,0) circle (1pt);
      \node at (-1.3,0,0) {$-\sqrt{2} E_1$};
      \fill (0,1,1) circle (1pt);
      \node at (.36,1,1) {$-2 \nu_1^-$};
      \fill (0,1,-1) circle (1pt);
      \node at (0,1.15,-1) {$2 \nu_1^+$};      
      \fill (0,-1,-1) circle (1pt);
      \node at (.1,-0.85,-1) {$2 \nu_1^-$};
      \fill (0,-1,1) circle (1pt);
      \node at (0,-1.175,1) {$-2 \nu_1^+$};
      
      \translatepoint{-1.5,-.75,1};
      \begin{scope}[shift=(middlepoint)]
	\draw[-{Latex[length=2mm]}] (0,0,0) --  (.75,0,0);
	\node at (.875,0,0) {$E_1$};
	\draw[-{Latex[length=2mm]}] (0,0,0) --  (0,0,-.75);
	\node at (0,0,-.95) {$E_2$};
	\draw[-{Latex[length=2mm]}] (0,0,0) --  (0,.75,0);
	\node at (0,.85,0) {$E_3$};
      \end{scope}
      
    \end{scope}
   \end{tikzpicture}
  \caption{Sketch showing the extremal points of the polyhedron $P$.}
  \label{fig:extremal_points}
\end{figure}
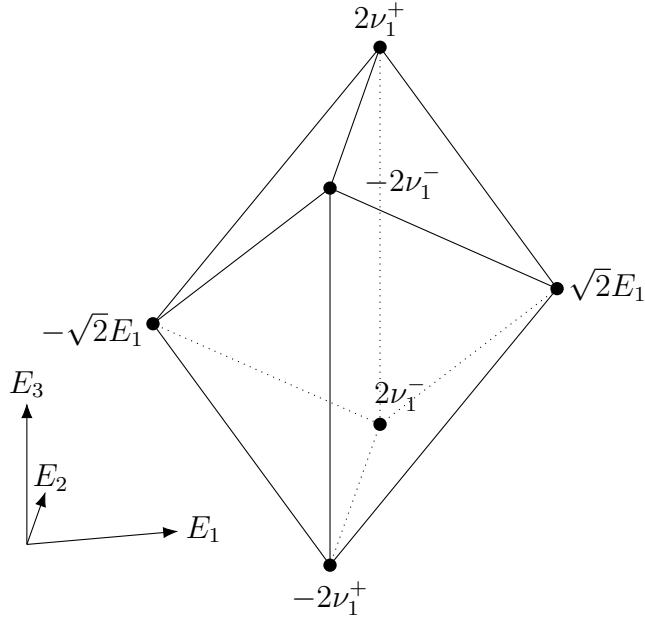
 
 The extremal points can be computed in a straightforward manner by finding all intersections of three of its two-dimensional faces still lying in $\bar P$.
 The resulting points are $\pm\sqrt 2 E_1$, $\pm  2 (011) = \pm 2 \nu_1^+$ and $\pm \sqrt 2 (01\overline1) = \pm \sqrt{2} \nu_1^-$, see Figure \ref{fig:extremal_points}.
 These can be presented as
 \begin{alignat*}{8}
  \pm \sqrt 2 e_1 & = \pm \left( \frac{1}{\sqrt 2} [111] - \frac{1}{\sqrt 2} [\overline 111]\right), & \text{ for } z & = 0,\  & \alpha & = - \beta & = & \pm 1,\\
  \pm \sqrt 2 (011) & = \pm \left( \frac{1}{\sqrt 2} [111] + \frac{1}{\sqrt 2} [\overline 111]\right), & \text{ for }  z & = 0,\ & \alpha & =  \beta & = & \pm 1,\\
  \pm \sqrt 2 (01\overline1), & & \text{ for } z & = \pm 1,\  & \alpha & = \beta & = & \ 0 
 \end{alignat*}
 and thus they lie in $\overbar Q$.
 
 Furthermore, in order to see that $\overbar Q$ is convex, we only have to prove
 \[\lambda I(z_1) + (1-\lambda) I(z_2) \subset I\left(\lambda z_1 + (1-\lambda) z_2\right)\]
  for all $ -1 \leq z_1, z_2 \leq 1$.
 Indeed, by the triangle inequality we have 
 \begin{align*}
  & \quad \lambda I(z_1) + (1- \lambda) I(z_2) \\
  & = \lambda \big(-1 + |z_1|, 1-|z_1|\big) + (1-\lambda) \big(-1 + |z_2|, 1-|z_2|\big) \\
  & =  \big(-1 + \lambda |z_1| + (1-\lambda) | z_2|, 1- \lambda |z_1| - (1-\lambda) |z_2| \big) \\
  & \subset \big(-1 + |\lambda z_1 + (1-\lambda) z_2|, 1- |\lambda z_1 + (1-\lambda) z_2| \big)\\
  & = I\left(\lambda z_1 + (1-\lambda) z_2\right).
 \end{align*}

\textit{Step 4: Prove that $f_\nu$ is almost affine for $\nu \in N_2 \cup N_3$  if the one-dimensional functions are continuous.}\\
We will only deal with $\nu = \frac{1}{\sqrt 2}(101)$.
The advantage of working with continuous functions is that we do not have to bother with sets of measure zero.
Let $(s,h,\tilde h) \in \R^3$ be such that $s$, $s + h$, $s + \tilde h$, $s + h + \tilde h \in (-1,1)$.
In order to exploit Remark \ref{rem:combinatorics} we set
\begin{align*}
 x_1 &:= \sqrt{2}sE_1 ,\\
 x_2 &:= \sqrt{2}sE_1 + h \frac{1}{\sqrt 2} [111],\\
 x_3 &:=  \sqrt{2}sE_1 + \tilde h \frac{1}{\sqrt 2} [1\overline11],\\
 x_4 &:=  \sqrt{2}sE_1 + h \frac{1}{\sqrt 2} [111] + \tilde h \frac{1}{\sqrt 2} [1\overline11].
\end{align*}
To prove $x_j \in P$ for all $j=1,2,3,4$ we go through the cases:
\begin{itemize}
 \item The facts $x_0 \cdot \nu = s$ and $\frac{1}{\sqrt 2}[111] \cdot \nu = \frac{1}{\sqrt 2}[1\overline11] \cdot \nu =1 $ clearly implies $x_j\cdot \nu \in (-1,1)$ for $j=1,2,3,4$.
 \item In contrast, for $\tilde \nu = \frac{1}{\sqrt 2} (\overline101)$ we have $x_0 \cdot \tilde \nu = - s$ and $\frac{1}{\sqrt 2}[111] \cdot \tilde \nu = \frac{1}{\sqrt 2}[1\overline11] \cdot \tilde \nu =0 $, which still implies $x_j \cdot \tilde \nu \in (-1,1)$.
 \item For $\tilde \nu \in N_3$ we have $x_0 \cdot \tilde \nu = s$ and 
 \[\left\{\frac{1}{\sqrt 2}[111] \cdot \nu , \frac{1}{\sqrt 2}[1\overline11] \cdot \nu \right\}=\{0,1\},\]
 which also implies $x_j \cdot \tilde \nu \in (-1,1)$.
\end{itemize}
By Step 3 have
\[|\theta_1(x_4) + \theta_1(x_1) - \theta_1(x_2) - \theta_1(x_3) | \leq 24 \eps.\]
Inserting the decomposition into the one-dimensional functions and making use of the combinatorics above we see that
\[\left|f_{(101)}(s + h + \tilde h) + f_{(101)}(s) - f_{(101)}(s + h) - f_{(101)} (s + \tilde h)\right| \leq 24 \eps.\qedhere\]
\end{proof}
%
%\begin{cor}\label{cor: almost affine on polyhedron}
% In the same setting as in Lemma \ref{lemma: almost maxima on transversal lines, 3D}, if the 1D functions a
%\end{cor}
%
%
%
%
%

\begin{proof}[Proof of Lemma \ref{lemma: shifts don't matter}]
Density of continuous functions with compact support in $L^p$ implies
\[\lim_{|h| \to 0} \int_{\R^n} |f(x+h) - f(x)|^p  \intd x  = 0.\]
For $y \in \ball{0}{1}$ setting $h=z+\tau y$ we thus get
\[\lim_{|z|,\tau \to 0} \int_{\R^n} |f(x+z+\tau y) - f(x)|^p  \intd x  = 0\]
uniformly in $y$.
After integration in $y$ we obtain the claim
\[\lim_{|z|, \tau \to 0}\int_{\R^n} \dashint_{\ball{0}{1}} |f(x + z + \tau y) - f(x)|^p \intd y \intd x = 0.\qedhere\]
\end{proof}

\subsection{\texorpdfstring{The case $f_\nu \in VMO$  for all $\nu \in N$}{The case f in VMO for all normals}}

\begin{proof}[Proof of Proposition \ref{Prop: Rigidity VMO}]
 Throughout the proof let $\tilde r>0$ be a universal, fixed radius, which we will choose later.
 We will denote generic radii with $r > \tilde r$.
 These may decrease from line to line.

 Applying the mean value theorem for $VMO$-functions, Lemma \ref{lemma: Mean value theorem for VMO}, we get that if $\theta \in \{e_1,e_2,e_3\}$ almost everywhere on $\ball{0}{\tilde r}$, then it holds that $\theta \equiv e_i$ for some $i \in \{1,2,3\}$ on $\ball{0}{\tilde r}$, which implies degeneracy by Corollary \ref{cor: theta constant}.
 Thus we may additionally assume that on $\ball{0}{\tilde r}$, exploiting symmetry of the problem, that
 \[\left|\left\{x\in \ball{0}{\tilde r}: \theta_1(x) = 0 , 0 < \theta_2(x), \theta_3(x) < 1\right\}\right| > 0.\numberthis \label{middle_of_edge}\]
\textit{Step 1: Find a set $A \subset \ball{0}{\tilde r}$ with $|A|>0$ and $\eps = \eps(\delta) \ssearrow 0$ as $\delta \ssearrow 0$ such that the following hold:
\begin{itemize}
 \item On $A$ we have
 	\begin{align}\label{average in middle of edge}
 		\theta_1 = 0\text{ and } & \frac{\eta}{2}   < \theta_{2} \text{, } \theta_{3} < 1 - \frac{\eta}{2},\\
 	 	\theta_{1,\delta}   < \eps\text{ and } & \frac{\eta}{2}  < \theta_{2,\delta} \text{, } \theta_{3,\delta} < 1 - \frac{\eta}{2} .
	 \end{align}
 \item On $\ball{0}{r}$ we have
 	\[\theta_\delta \subset \tilde K_\eps := \tilde \K + \ball{0}{\eps} \cap \operatorname{conv}(\tilde K) \text{ on } \ball{0}{r},\numberthis \label{approx_diff_incl}\]
 where $\operatorname{conv}(\tilde K)$ denotes the convex hull, see Figure \ref{fig:approx_K}.
\end{itemize}
%  and
% 
% In particular, given $x\in \ball{0}{r}$ we have $\theta_{i,\delta}(x) \leq \epsi$ for some $i\in \{1,2,3\}$.
 We may furthermore assume
 \begin{align}
  0 \in A \label{density one}  
 \end{align}
 to be a point of density one in the sense that $\frac{ |A\cap \ball{0}{\kappa}| }{    |\ball{0}{\kappa}| } \to 1$ as $\kappa \to 0$.}\\
 Recall that we defined $\theta_\delta(x) = \dashint_{\ball{x}{\delta}} \theta(y) dy$.
 As convolutions are convex operations we obtain $\theta_\delta \in \operatorname{conv}(\tilde \K)$ a.e.
 Another application of Lemma \ref{lemma: Mean value theorem for VMO} gives the fuzzy inclusion \eqref{approx_diff_incl} with $\eps = \eps(\delta) \to 0$ as $\delta \to 0$. 
 The additional assumption \eqref{middle_of_edge} implies that there exists $\eta >0$ such that on $\ball{0}{\tilde r}$ we have
 \begin{align}\label{WLOG_centered_at_zero}
  \left|\left\{x\in \ball{0}{\tilde r}: \theta_1(x) = 0 , \eta < \theta_2(x), \theta_3(x) < 1-\eta \right\}\right| > 0.
 \end{align}

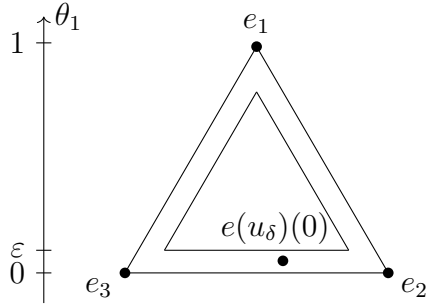
\begin{figure}
 \centering
 \begin{tikzpicture}[scale=2]
% 	  \node at (-1,1) {b)};
	  \fill (90:1) circle (1pt);
	  \node at (90:1.15) {$e_1$};
	  \node at (210:1.2) {$e_3$};
	  \node at (330:1.2) {$e_2$};
	  \draw (90:1) -- (210:1);
	  \draw (210:1)-- (330:1);
	  \draw (330:1) -- (90:1);
	  
	  \draw (90:.7) -- (210:.7);
	  \draw (210:.7)-- (330:.7);
	  \draw (330:.7) -- (90:.7);
	  
	  \fill (330:1) circle (1pt);
	  \fill (210:1) circle (1pt);
	  \fill ($(330:1)!.4!(210:1) + (0,.08)$) circle (1pt);
	  \node at ($(330:1)!.4!(210:1) + (-.05,.3)$) {$e(u_\delta)(0)$};
	  \draw[->] (-1.4, -0.7 ) -- (-1.4, 1.2) node[right] {$\theta_1$};
	  \draw (-1.35, -0.5) -- (-1.45, -0.5) node[left] {0};
	  \draw (-1.35, -.35) -- (-1.45, -.35) node[left] {$\eps$};
	  \draw (-1.35, 1.025) -- (-1.45, 1.025) node[left] {1};
 \end{tikzpicture}
  \caption{Sketch of the strains taking the form $e= \sum_{i=1}^3 \theta_ie_i$ for $\theta \in \tilde \K_\eps$. The strain $e(u_\delta)(0)= \sum_{i=1}^3\theta_{i,\delta} e_i$ essentially lies strictly between $e_2$ and $e_3$.}
 \label{fig:approx_K}
\end{figure}

 Lebesgue point theory implies that $\theta_\delta \to \theta$ pointwise almost everywhere.
 Using Egoroff's Theorem, we may upgrade this convergence to uniform convergence on some set
 \[A \subset \left\{x\in \ball{0}{\tilde r}: \theta_1(x) = 0 , \eta < \theta_2(x), \theta_3(x) < 1-\eta \right\}\]
 with $|A| >0$ and such that all points in $A$ have density one.
 Using both uniform convergences above we get that for $\delta > 0$ small enough we have
 \begin{align*}
  \theta_{1,\delta}  < \eps\text{, } \frac{\eta}{2}  < \theta_{2,\delta} \text{, } \theta_{3,\delta} < 1 - \frac{\eta}{2} \text{ on } A
 \end{align*}
 with $\eps = \eps(\delta) \to 0$ as $\delta \to 0$.

 To see that we may assume property \eqref{density one}, namely $0 \in A$, let $\bar r \leq 1$ be a universal radius with which the conclusion of the proposition holds under the assumption that we indeed have $0\in A$.
 We may then choose the radius $\tilde r = \frac{1}{4} \bar r$ in inequality \eqref{WLOG_centered_at_zero} so that $A \subset \ball{0}{\frac{1}{4}\bar r}$.
 For any point $x\in A$ we then clearly have $\ball{x}{\frac{1}{2}} \subset\ball{0}{1}$.
 Shifting and rescaling said ball to $\ball{0}{1}$ and applying the conclusion in the new coordinates, we see that the configuration only involves two variants on $\ball{x}{\frac{1}{2}\bar r}$.
 Consequently, it is a two-variant configuration on $\ball{0}{\frac{1}{4}\bar r} \subset \ball{x}{\frac{1}{2}\bar r}$.
%  
%  small enough such that $A \subset \ball{0}{r_1}$ and such that there exists a universal $r_2 >0$ with $\ball{0}{r_2} \subset \ball{x}{r_3}$ for all $x \in \ball{0}{r_1}$, where $r_3$ is a radius on which the conclusion of the theorem holds under the assumption $0 \in A$.
%  Then we apply the following arguments to $\theta$ shifted in such a way that $0 \in A$.
% 
% , with $\alpha \in (-\tilde r, \tilde r)$, for which \[\mathcal{H}^2\left(A \cap \left\{\frac{1}{\sqrt 2}(011)\cdot x = \alpha\right\}\right)>0,\]
 
\textit{Step 2: On the plane $H\big(0,\frac{1}{\sqrt 2}(011)\big)$ we split up $\theta_1$ into two one-dimensional functions and find maximal intervals on which they are essentially constant.}\\
 Similarly to the proof of Proposition \ref{Prop: dimension reduction} we parametrize the plane $H\big(0,\frac{1}{\sqrt 2}(011)\big)$ via
 \[X_k(\beta,\gamma):= \beta \frac{1}{\sqrt 2} [1\overline11] + \gamma \frac{1}{\sqrt 2} [11\overline1],\]
 which gives the relations
 \begin{align}
  X(\beta,\gamma) \cdot \frac{1}{\sqrt 2} (011) & = 0, \label{VMO parametrization (011)}\\
  X(\beta,\gamma) \cdot \frac{1}{\sqrt 2} (101) & = \beta, \label{VMO parametrization (101)}\\
  X(\beta,\gamma) \cdot \frac{1}{\sqrt 2} (110) & = \gamma, \label{VMO parametrization (110)}\\
  X(\beta,\gamma) \cdot \frac{1}{\sqrt 2} (01\overline1) & = \gamma - \beta, \label{VMO parametrization (01-1)}\\
  X(\beta,\gamma) \cdot \frac{1}{\sqrt 2} (\overline101) & =  - \gamma, \label{VMO parametrization (-101)}\\
  X(\beta,\gamma) \cdot \frac{1}{\sqrt 2} (1\overline10) & = \beta. \label{VMO parametrization (-110)}  
 \end{align}
 Absorbing the affine function $g_1$ in decomposition \eqref{decomposition} into the four one-dimensional functions $f_\nu$ for $\nu \in N_2 \cup N_3$ we may assume
 \begin{align}\label{VMO_simplified_decomposition}
  \theta_1 = f_{(101)} + f_{(\overline101)} - f_{(110)} - f_{(1\overline10)}.
 \end{align}
 
 As before, we exploit the combinatorial structure of the normals discussed in Remark \ref{rem:combinatorics} and sort these according to their dependence on $\beta$ or $\gamma$  on the plane $H(0,(011))$ by defining
 \begin{align*}
  \begin{split}
    F_{1}(\beta) & := f_{(101)}(\beta) - f_{(1\overline10)}(\beta), \\
    F_{2}(\gamma) & := f_{(\overline101)}(-\gamma) - f_{(110)}(\gamma).
  \end{split}
 \end{align*}
 As a result of Lemma \ref{lemma: almost maxima on transversal lines} we may shuffle around some constant so that we can assume
 \begin{align}\label{F_positive}
	  F_{1}, F_{1}\geq 0.
\end{align} 
 The decomposition then turns into
 \begin{align*}
  \theta_{1,\delta} \circ X(\beta,\gamma) = F_{1,\delta}(\beta) + F_{2,\delta}(\gamma)
 \end{align*}
 after averaging.
 
 Due to our assumption that $0 \in A$ and the fact that inequality \eqref{average in middle of edge} is an open condition, continuity of $\theta_\delta$ implies that there exists $\kappa = \kappa(\delta) >0$ such that 
 \begin{align}\label{average in middle of edge in small ball}
  \theta_{1,\delta}  < \eps\text{ and } \frac{\eta}{2} < \theta_{2,\delta} \text{, } \theta_{3,\delta} < 1 - \frac{\eta}{2} \text{ on }\ball{0}{\kappa}.
 \end{align}
 As $\theta_{1,\delta}$ is a sum of two one-dimensional functions that is small due to the first inequality of \eqref{average in middle of edge in small ball} the individual terms are small by Lemma \eqref{lemma: almost maxima on transversal lines}, i.e., we have
 \begin{align*}
    F_{1,\delta}(\beta) - \min_{[-\kappa, \kappa]} F_{1,\delta} & \leq \eps \text{ on } [-\kappa,  \kappa]  , \\
    F_{2,\delta}(\gamma) - \min_{[-\kappa, \kappa]} F_{2,\delta} & \leq \eps \text{ on } [-\kappa,  \kappa],
 \end{align*}
 where we used continuity to replace the essential infima.
 In particular, for the oscillations on closed intervals $I$, defined as
 \begin{align*}
  \osc_{I} F_{1,\delta} & :=   \max_{ I} F_{1,\delta} - \min_{[-\kappa, \kappa] }  F_{1,\delta}   , \\
  \osc_{I} F_{2,\delta} & :=   \max_{ I} F_{2,\delta} - \min_{[-\kappa, \kappa] }  F_{2,\delta},
 \end{align*}
 we have that
 \begin{align*}
  0 \leq \osc_{[-\kappa, \kappa]} F_{1,\delta} &  \leq \eps , \\
  0 \leq \osc_{[ -\kappa,  \kappa]} F_{2,\delta} & \leq \eps.
 \end{align*}
 
 By continuity of $F_{1,\delta}$ and $F_{2,\delta}$ the oscillations are continuous when varying the endpoints of the involved intervals.
 Thus there exist unique maximal intervals 
 \[[-\kappa, \kappa]\subset I^{\delta}_1 \subset [-r,r]\text{ and } [ -\kappa,  \kappa] \subset I^{\delta}_2 \subset [-r,r]\]
 such that
 \begin{align*}
  \osc_{I^{\delta}_1} F_{1,\delta} \leq \eps \text{ and } \osc_{I^{\delta}_2} F_{2,\delta}   \leq \eps.
 \end{align*} 
 We would like to prove that $[-\tilde r,\tilde r] \subset I_{1}^\delta, I_2^\delta$, but for the next couple of steps we will be content with making sure they do not shrink away as $\delta \to 0$, see Figure \ref{fig:VMO_outline} for an outline of the argument.
 Note that we will drop the dependence of $I_1$ and $I_2$ on $\delta$ in the following as long as we keep it fixed.
 
 \begin{figure}
 \centering
 \begin{tikzpicture}[scale=1.4]
%   \draw (-.5,-.5) -- (-.5,.5) -- (.5,.5) -- (.5,-.5) -- cycle;
%   \node at (1.2,0) {$[-\kappa,\kappa]^2$};
  \draw (-2,-1) -- (3,-1) -- (3, 2) -- (-2,2) -- cycle;
  \node at (2.4,-.6) {$I_1\times I_2$};
  \draw[dashed] (-1,-1) -- node[above]{$l$} (2,2);
  \draw[dotted] (-1.5,-1.5) -- (-1,-1);
  \draw[dotted] (2,2) -- (2.5,2.5);

  \fill (0,0) circle (1.5pt);
  \node at (-.2,.2) {$l(0)$};
  
  \fill (-1,-1) circle (1.5pt);
  \node at (-1.4,-.8) {$l(t_{min})$};
  \fill (2,2) circle (1.5pt);
  \node at (1.6,2.2) {$l(t_{max})$};
 \end{tikzpicture}
  \caption{Sketch relating $I_1\times I_2$ and the line $l(t) = t(1,1)$. Step 3 ensures that $\min(\theta_{2,\delta}, \theta_{3,\delta}) <\eps$ on $\partial (I_2\times I_3)$. In Step 4 we will show that $\theta_2$ is almost affine along the dashed part of $l$, which we will exploit in Step 5 to argue that $\theta_2\circ l(t_{min}) \approx 0$ and $\theta_2\circ l(t_{min}) \approx 1$ or vice versa due to $\theta_2\circ l(0) \not \approx 0,1$.
  The function $\theta_2$ being of vanishing mean oscillation allows us then to deduce that $t_{max}$ and $t_{min}$ cannot get too close as $\delta \to 0$.}
 \label{fig:VMO_outline}
\end{figure}
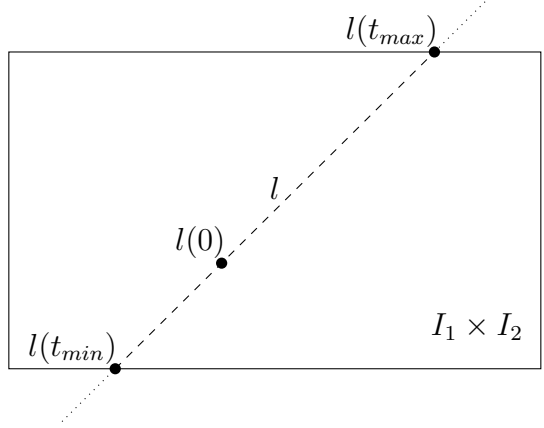
 
 \textit{Step 3: Prove $\min\{\theta_{2,\delta},\theta_{3,\delta}\}<\eps$ on $\partial \left(I_1\times I_2\right) \cap (-r,r)^2$.}\\
 For $\beta \in  \partial I_1 \cap (-r,r)$ we have
 \[  F_{1,\delta}(\beta) - \min_{[-\kappa, \kappa]} F_{1,\delta}  = \eps.\]
 Together with \eqref{F_positive} we obtain for $\gamma \in   I_2 \cap (-r,r)$ that
 \[\theta_{1,\delta}(\beta, \gamma) = F_{1,\delta}(\beta) + F_{2,\delta}(\gamma) = \eps + \min_{[-\kappa, \kappa]} F_{1,\delta} +  F_{2,\delta}(\gamma) \geq \eps .\]
 Swapping the roles of $\beta$ and $\gamma$ and using Step 1 and the definition of $\tilde \K$ we thus see
 \begin{align} \label{not_too_far}
  \min\{\theta_{2,\delta},\theta_{3,\delta}\}<\eps 
 \end{align}
 on the set $\partial (I_1 \times I_2 ) \cap (-r,r)^2$.

\textit{Step 4: The functions $f_{\nu,\delta}\circ X$ for $\nu \in N_2\cup N_3$, $\theta_{2,\delta}\circ X$ and $\theta_{3,\delta} \circ X$ are almost affine along $l(t):=  t (1,1)$ as long as $t^{\delta}_{min} < t < t^{\delta}_{max}$.}\\
 Here $t^{\delta}_{min} < 0< t^{\delta}_{max}$ are the two parameters for which $l$ intersects $\partial (I^{\delta}_1 \times I^{\delta}_2)$, see Figure \ref{fig:VMO_outline}. We again drop the superscripts in the notation of these objects as well as long as we keep $\delta$ fixed.
 
 For parameters $\bar \beta \in \operatorname{arg\ min}_{[ -\kappa,  \kappa] } F_{1,\delta}$ and $\bar \gamma \in \operatorname{arg\ min}_{[ -\kappa,  \kappa] } F_{2,\delta}$ we have
 \[\theta_{1,\delta} \circ X(\bar\beta,\bar\gamma) \leq \eps.\]
 Consequently we have for any $(\beta,\gamma) \in I_1\times I_2$ and for a generic constant $c>0$ which may change from line to line that
 \begin{align}\label{VMO theta_1 approx 0}
  0 \leq \theta_{1,\delta} \circ X(\beta,\gamma)  \leq \theta_{1,\delta} \circ X(\bar\beta,\bar\gamma) + \osc_{I_1} F_{1,\delta} + \osc_{I_2} F_{2,\delta} \leq c\eps.
 \end{align}
 
 As we have that $X \circ l(t) = \sqrt{2} t E_1$ is parallel to $E_1=[100]$ and $l(t)\in I_1\times I_2$ for $t \in [t_{min},t_{max}]$
 we can apply Lemma \ref{lemma: almost maxima on transversal lines, 3D}
 to see that $f_{\nu,\delta}$ is almost affine
 \begin{align*}
  \begin{split}
   & \big| f_{\nu,\delta}\circ X\circ l(t + h + \tilde h)  +  f_{\nu,\delta} \circ X \circ l(t) \\
   & \quad -f_{\nu,\delta}\circ X \circ l(t + h)   -  f_{\nu,\delta}\circ X \circ l(t + \tilde h) \big| < C \eps
  \end{split}
 \end{align*}
 for $t$, $h$, $\tilde h \in \R$ such that $t$, $t + h$, $t + \tilde h$, $t + h + \tilde h \in [t_{min},t_{max}]$ and $\nu \in N_2\cup N_3$.
 Plugging this into the decomposition \eqref{decomposition} of $\theta_2$ and $\theta_3$ and observing that affine functions drop out in second discrete derivatives and that $f_{(011)}$ and $f_{(01\overline1)}$ drop out as the line $X\circ l$ is parallel to $E_1$, we obtain
 \begin{align}\label{VMO_almost_affine}
  \begin{split}
   & \big| \theta_{2,\delta}\circ X \circ l(t + h + \tilde h) + \theta_{2,\delta} \circ X \circ l(t) \\
   & \quad -\theta_{2,\delta}\circ X \circ l(t + h) - \theta_{2,\delta}\circ X \circ l(t + \tilde h) \big| < C \eps,\\
   & \big| \theta_{3,\delta}\circ X \circ l(t + h + \tilde h) + \theta_{3,\delta} \circ X \circ l(t) \\
   & \quad -\theta_{3,\delta}\circ X \circ l(t + h) - \theta_{3,\delta}\circ X \circ l(t + \tilde h)\big| < C \eps,
  \end{split}
 \end{align}
 for $t$, $h$, $\tilde h \in \R$ such that $t$, $t + h$, $t + \tilde h$, $t + h + \tilde h \in [t_{min},t_{max}]$.
%  Lemma \ref{lemma: almost maxima on transversal lines, 3D} implies that
%  \begin{align}\label{1D functions almost affine}
%   \left| f_{\nu,\delta}\left(t + h + \tilde h\right) - f_{\nu,\delta}(t + h ) -f_{\nu,\delta}\left(t  + \tilde h\right) + f_{\nu,\delta}(t )\right| < c \eps
%  \end{align}
%  for $\nu \in N_2 \cup N_3$; $t$, $t+ h$, $t+\tilde h$, $t+h + \tilde h \in \nu \cdot X(I_1\times I_2)$ and some universal $c>0$.

%  Let $I^{\alpha,\delta}_1:=[\beta^{\alpha,\delta}_{min},\beta^{\alpha,\delta}_{max}]$ and $I^{\alpha,\delta}_2:=[\gamma^{\alpha,\delta}_{min}, \gamma^{\alpha,\delta}_{max}]$.

\textit{Step 5: If $\delta>0$ is sufficiently small and we have $ -r < t_{min} < t_{max} < r$, then
 \[\theta_{2,\delta} \circ X \circ l(t_{min}) < \eps \text{ and }  \theta_{3,\delta} \circ X \circ l(t_{max}) <\eps\]
 or
 \[\theta_{3,\delta} \circ X \circ l(t_{min}) < \eps \text{ and }  \theta_{2,\delta} \circ X \circ l(t_{max}) <\eps .\]
 By inequality \eqref{not_too_far} the statement $\theta_{3,\delta} \circ X \circ l(t_{max}) <\eps$ implies $\theta_{2,\delta} \circ X \circ l(t_{max}) >1-\eps$.
 We also get the same implication at $t_{min}$.
%  Furthermore, with a suitable choice of $\tilde r$, we have one of the following:
%  \begin{enumerate}
%   \item $\beta_{min}, \gamma_{min} \leq - \tilde r$,
%   \item $\beta_{max}, \gamma_{max} \geq \tilde r$,
%   \item $ [-\tilde r,\tilde r] \subset I_1$,
%   \item $ [-\tilde r,\tilde r] \subset I_2$.
%  \end{enumerate}
 }\\
 Aiming for a contradiction we assume that
 \begin{align} \label{theta_in_corner}
  \begin{split}
    \theta_{3,\delta} \circ X(l(t_{min})) & < \eps,\\
    \theta_{3,\delta} \circ X(l(t_{max})) & < \eps.
  \end{split}
 \end{align}
 Recalling Step 3 we see that the only other undesirable case is $\theta_{2,\delta} \circ X(l(t_{min}))<\eps$, $\theta_{2,\delta} \circ X(l(t_{max})) < \eps$, which can be dealt with in the same manner.
 
 In order to transport this information to the point $l(0)$ we use that $\theta_{3,\delta} \circ X $ is almost affine along $l(t)$, see \eqref{VMO_almost_affine}, to get
 \begin{align*}
    \big| \theta_{3,\delta} \circ X \circ l(t_{max}) - \theta_{3,\delta} \circ X\circ l(0)  -\theta_{3,\delta}\circ X\circ l(t_{min} + t_{max}) + \theta_{3,\delta} \circ X \circ l(t_{min}) \big| < C \eps.
 \end{align*}
 with $t : = t_{min}$, $h:= -t_{min}$ and $\tilde h : = t_{max}$.

 Combining this inequality with $\theta_{3,\delta}\circ X \circ l(t_{min} + t_{max})  \geq 0 $ and the supposedly incorrect assumption \eqref{theta_in_corner} we arrive at
 \begin{align*}
   \quad \theta_{3,\delta} \circ X \circ l(0) & < \theta_{3,\delta}\circ X \circ l(t_{max})  + \theta_{2,\delta}\circ X \circ l(t_{min})  - \theta_{2,\delta}\circ X \circ l(t_{min} + t_{max}) + C\eps \\
  & \leq  C\eps 
 \end{align*}
 However, this is in contradiction to the strain lying strictly between two martensite strains at $0$ for small $\delta$, see \eqref{average in middle of edge in small ball}, which proves the claim.
 
\textit{Step 6: We do not have $ \liminf_{\delta \to 0} t^{\delta}_{max} - t^{\delta}_{min} = 0$.}\\
 Towards a contradiction we assume that the difference does vanish in the limit. 
 Let $g_\delta(s) := \left(f_{(101),\delta} + f_{(\overline101), \delta} \right) ( (1-s)t^\delta_{min} + st^\delta_{max} ) $ for $s \in [0,1]$.
 By Lemma \ref{lemma:almost_affine} the sequence $g_\delta$ converges uniformly to an affine function $g$.
 As by Step 5 we know that the linear part of $g$ has to be nontrivial, recall that $f_{(011),\delta}$ and $f_{(01\overline1),\delta}$ drop out in the decomposition of $\theta_2$ along $X\circ l$, we get that
 \[\int_0^1 \left| g(s) - \int_0^1 g(\tilde s) \intd \tilde s \right| \intd s > 0.\]
 Undoing the rescaling we conclude that
 \[\lim_{\delta \to 0} \dashint_{t^{\delta}_{min}}^{t^{\delta}_{max}} \left|\left(f_{(101),\delta} + f_{(\overline101), \delta} \right) (t ) - \dashint_{t^{\delta}_{min}}^{t^{\delta}_{max}}\left(f_{(101),\delta} + f_{(\overline101), \delta} \right) ( \tilde t ) \intd \tilde t \right| \intd t>0.\]
 Due to Jensen's inequality this implies
 \[\liminf_{\delta \to 0} \dashint_{t^{\delta}_{min}-\delta}^{t^{\delta}_{max}+\delta} \left|\left(f_{(101)} + f_{(\overline101)} \right) (t ) - \dashint_{t^{\delta}_{min}-\delta}^{t^{\delta}_{max} + \delta }\left(f_{(101)} + f_{(\overline101)} \right) ( \tilde t ) \intd \tilde t \right| \intd t>0.\]
 However, this is a contradiction to our assumption that $f_{(101)}, f_{(\overline101)} \in VMO$ since we have $ t^\delta_{max} - t^\delta_{min} + 2\delta \to 0$. 
 
\textit{Step 7: The open set
\[\interior{\left\{x\in \ball{0}{r}: x \text{ is a Lebesgue point of } \theta_1 \text{ with } \theta_1(x) = 0\right\}}\]
has a connected component $P$ such that $0 \in \overline{P}$.
Furthermore, the set $P$ satisfies
\[P \cap \ball{0}{r} = \bigcap_{\nu \in N_2 \cup N_3} \{\nu \cdot x \in  I_\nu\} \cap \ball{0}{r}\]
for open, non-empty intervals $I_\nu \subset \R$, i.e., up to localization it is a polyhedron whose faces' normals are contained in $N_2\cup N_3$.}\\
 By Step 6 and Lemma \ref{lemma: almost maxima on transversal lines, 3D} we find a connected component $P$ of the above set such that $0 \in \overline{P}$ in the limit $\delta \to 0$.
 In the following, we will choose the precise representatives of all involved functions, see Evans and Gariepy \cite[Chapter 1.7.1]{evans2015measure}, so that we can evaluate $\theta_1$ in a pointwise manner.
 
 By distributionally differentiating the condition
 \[f_{(101)} + f_{(\overline101)} - f_{(110)} - f_{(1\overline10)} \overset{\eqref{VMO_simplified_decomposition}}{=} \theta_1 \equiv 0
 \]
 on $P$ in two different directions $d, \tilde d \in \mathcal{D}$, see Subsection \ref{subsec:notation}, and making use of Remark \ref{rem:combinatorics} we see that $f_\nu$ is locally affine on $P$ for $\nu \in N_2 \cup N_3$.
 By connectedness of $P$, they must be globally affine:
 
 Let $\nu \in N_2 \cup N_3$ and let $G:=\{g:\R^3 \to \R: g \text{ is affine}\}$.
 Let
 \[U_g:= \{x \in P:  f_\nu(\nu\cdot y) \equiv g(y) \text{ for } y \in \ball{x}{\kappa} \text{ for some } \kappa >0 \}.\]
 By construction, these sets are open.
 They are also disjoint because two affine functions agreeing on a non-empty open set have to coincide globally.
 Finally, we have $P = \bigcup_{g\in G} U_g$ by assumption.
 Therefore, there exists a single affine function $g$ such that $f_\nu = g$ on $P$.
 We may thus re-define $f_\nu$ for $\nu \in N_2 \cup N_3$ to satisfy 
 \[f_\nu \equiv 0 \text{ on } P.\numberthis \label{VMO_some_directions_constant}\]
 
 The image $I_\nu := \nu \cdot P$ is open and connected, and thus an interval.
 It is also clearly non-empty and by construction we have
 \[P \subset \bigcap_{\nu \in N_2 \cup N_3} \{\nu \cdot x \in  I_\nu\} \cap \ball{0}{r}.\]
 As it holds that $f_{\tilde \nu} =0$ on $\bigcap_{\nu \in N_2 \cup N_3} \{\nu \cdot x \in  I_\nu\} \cap \ball{0}{r}$ for all $\tilde \nu \in N_2 \cup N_3$ we get the other inclusion
 \[\bigcap_{\nu \in N_2 \cup N_3} \{\nu \cdot x \in I_\nu\} \cap \ball{0}{r} \subset P,\]
 which proves the claim.
 
\textit{Step 8: Let $F$ be a face of $P$ with normal $\nu \in N_i$ for $i\in \{2,3\}$ and $F\cap \ball{0}{r} \neq \emptyset$.
  Then $\theta_i \equiv 0$ or $\theta_i \equiv 1$ on $F$.}\\
  The claim is meaningful by Lemma \ref{lemma: traces}. 
  In order to keep notation simple, we assume that $\nu = \frac{1}{\sqrt 2} (101)$ and that $\nu$ is the outer normal to $P$ at $F$, i.e., we have $P\subset \{x\cdot \nu < b\}$ with $\{b\} = \nu \cdot F$.
  A two-dimensional sketch of this situation can be found in Figures \ref{fig:VMO_boundary_a}, while a less detailed three-dimensional one is shown in Figure \ref{fig:preimage_a}.
  
  Furthermore, we only have to prove the dichotomy $\theta_2 \equiv 0$ or $\theta_2 \equiv 1$ locally on $F$, i.e., on $\ball{x_0}{\kappa}$ for all $x_0\in F$ and some $\tilde \kappa = \tilde\kappa(x_0)>0$ such that
  \[\ball{x_0}{\tilde\kappa} \cap H(b,\nu) \subset F \cap \ball{0}{r}\text{ and }\ball{x_0}{\tilde \kappa} \cap \{x\cdot \nu < b\} \subset P: \numberthis \label{relative_interior}\]
  By Lemma \ref{lemma: traces} and $f_\nu \in VMO$ for all $\nu \in N$ we have $\theta_i \circ X \in VMO\big(\tilde F\big)$, where $ X : \tilde F \to F$ is an affine parametrization of $F$.
  An application of the mean value theorem for VMO-functions, Lemma \ref{lemma: Mean value theorem for VMO}, gives the ``global'' statement on $F$ due to connectedness of $F$.
 
 \begin{figure}
 \centering
 \subcaptionbox{\label{fig:VMO_boundary_a}}{
  \begin{tikzpicture}[scale=1]
    \draw(-3,-1) -- (-3,0) --(3,0) --  (3,-1);
    \node[fill=black, circle, inner sep=1.5, label=-90:{$x_0$}] at (0,0) {};
    \draw (0,0) circle (2);
    \draw[{Latex[length=2mm]}-{Latex[length=2mm]}] (0,0) --node[right]{$\kappa$} (320:2);
    \node[label={$F$}] at (-1,-.1){};
    
    \draw[-{Latex[length=2mm]}] (-2.5,0) -- (-2.5,1) node[above]{$(101)$};
  \end{tikzpicture}
 }
 \subcaptionbox{\label{fig:VMO_boundary_b}}{
    \begin{tikzpicture}[scale=1]
    \draw(-3,-1) -- (-3,0) --(3,0) --  (3,-1);
    \node[fill=black, circle, inner sep=1.5, label=-90:{$x_0$}] at (0,0) {};
    \draw (0,0) circle (2);
    \draw[dashed] (170:2) -- (10:2);
    \draw[dotted] (170:2) -- ++ (-1.26,0);
    \draw[dotted] (10:2) -- ++ (1.1,0)node[above] {$H(b_\delta,(101))$};
    
    \path[name path=A] (170:2) -- ++(2,2); 
    \path[name path=B] (10:2) -- ++(-2,2);
    \path[name intersections={of=A and B}];
    
    \coordinate (a) at (intersection-1);
    
    \path[name path=C] (170:2) -- ++(2,-2); 
    \path[name path=D] (10:2) -- ++(-2,-2);
    \path[name intersections={of=C and D}];
    \draw (170:2) -- (a) -- (10:2) --  (intersection-1)  -- cycle;
    \fill[color=gray, opacity=.6] (170:2) -- (a) -- (10:2) --  (intersection-1);
    
    \draw[dotted] (-3.7,.8) -- (1.9,.8);
    \draw[dotted] (-3.7,-.8) -- (3.2,-.8);
    \draw [{Latex[length=2mm]}-{Latex[length=2mm]}] (-3.5,-.8) -- node[left]{2c} (-3.5,.8);
  \end{tikzpicture}
 }
 \caption{a) Inside $\ball{x_0}{\kappa}$, the polyhedron $P$ looks like a half-space with boundary $F$ and exterior normal $\nu_2^+$. 
b) The dichotomy $\theta_{2,\delta} \approx 0$ or $\theta_{2,\delta} \approx 1$ on the dashed line $H(b_\delta,(101)) \cap \ball{x_0}{\kappa}$ can be propagated to the gray neighborhood of $x_0$ as long as we have $\dist(x_0,H(b_\delta,(101)))<c$.}
 \label{fig:VMO_boundary}
\end{figure}
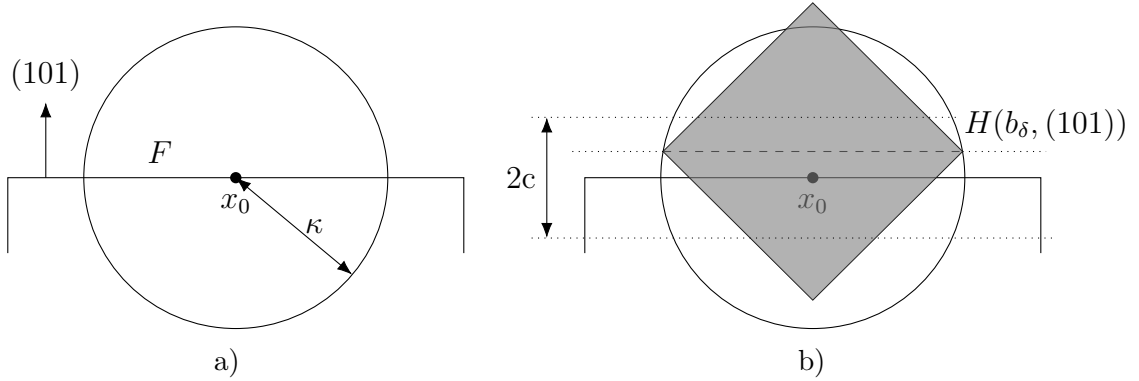
 
Let $x_0 \in F$ be such that there exists $\kappa>0$ with the inclusions  \eqref{relative_interior} being satisfied for $\tilde \kappa = \kappa$,
where in the following $\kappa$ may decrease from line to line in a universal manner.
We can use the identities \eqref{VMO_some_directions_constant} to conclude $f_{(101)} \equiv 0$ on $\ball{x_0}{2\kappa} \cap \{x\cdot \nu < b\}\subset P$ and $f_\nu \equiv 0$ on $\ball{x_0}{2\kappa}$ for $\nu \in N_2 \cup N_3 \setminus
 \left\{\frac{1}{\sqrt 2} (101)\right\}$.
 Consequently, we get
 \[f_{(101),\delta}\left(b - c\right) =0 \text{ and } f_{\nu,\delta} \equiv 0 \text{ on } \ball{x_0}{\kappa} \numberthis \label{(101) interior_2}\]
 after averaging provided we have $\delta <c$ for a constant $0<c<\kappa$ to be chosen later.
 In particular, the latter together with the decomposition \eqref{VMO_simplified_decomposition} implies
 \[\theta_{1,\delta} = f_{(101),\delta} \numberthis \label{component_theta_1d}.\]
 
 Therefore, we cannot have $f_{(101)} \equiv 0$ on the larger set $\ball{x_0}{\kappa} \cap \{x\cdot \nu < b + c\}$  as otherwise we would get the contradiction \[\ball{x_0}{\kappa} \cap \{x\cdot \nu < b + c\} \subset  P \subset \{x\cdot \nu < b\}.\]
 Written in terms of the approximation $f_{(101),\delta}$, recalling that $\eps = \eps(\delta) \to 0$ as $\delta \to 0$, this gives
 $f_{(101),\delta}(b_{\delta}) \geq \eps$
 for some $b_{\delta} \in  [b-c, b+c] \cap  [-r,r]$ and $\delta>0$ small enough.
 By equation \eqref{(101) interior_2} and continuity we may additionally assume that
 $f_{(101),\delta}(b_{\delta}) = \eps$ 
 which due to equation \eqref{component_theta_1d} implies that
 \[\theta_{1,\delta}(x) = \eps \numberthis \label{(101) changes} \]
 for all $x\in \tilde H := H\big(b_\delta,\frac{1}{\sqrt 2}(101)\big) \cap \ball{x_0}{\kappa}$, see Figure \ref{fig:VMO_boundary_b}.
 
 Combining this with the inclusion $\theta_\delta \in \tilde K + \ball{0}{\eps}$ we consequently get
 \[\min\{\theta_{2,\delta}(x), \theta_{3,\delta}(x) \} < \eps\]
 on $\tilde H$.
 Due to $\theta_1 + \theta_2 + \theta_3 \equiv 1$ we convert this into
 \[\min\{\theta_{2,\delta}(x), 1- \theta_{2,\delta}(x) \} < 2\eps\]
 for all $x \in \tilde H$.
 Continuity implies the dichotomy we have
 \[\text{ either } \theta_{2,\delta}(x) < 2\eps \text{ for } x\in \tilde H \text{ or } \theta_{2,\delta}(x) < 2\eps \text{ for } x\in \tilde H.\]
 
 In order to propagate this information back to $x_0$ let $x_\delta := x_0 + \left(b_\delta - b \right) \frac{1}{\sqrt 2} (101)$.
 The line $l(t):= x_\delta +  \sqrt 2 t E_2$ satisfies $l(t)\cdot \frac{1}{\sqrt 2} (101) = b_\delta$ by $x_0\in F \subset H\big(b,\frac{1}{\sqrt 2}(101)\big)$.
 We also have $l(t) \in \ball{x_0}{\kappa}$ for $t\in \left[-\frac{\kappa}{2},\frac{\kappa}{2}\right]$ provided we choose $c\leq \frac{\kappa}{2}$.
 Therefore, the above dichotomy holds along $l$.
 Consequently, Lemma \ref{lemma: almost maxima on transversal lines, 3D} implies that
 \[\min\{\theta_{2,\delta}(x), 1- \theta_{2,\delta}(x) \} < 12 \eps \numberthis \label{VMO_face_conclusion} \]
 on $\ball{x_\delta}{\kappa}$.
 By definition of $x_0$ and $b_\delta$ we have $|x_0- x_\delta| = |b-b_\delta| \leq c$.
 As a result, the choice $c\leq \frac{\kappa}{2}$ ensures that estimate \eqref{VMO_face_conclusion} holds on $\ball{x_0}{\tilde \kappa}$ for $\tilde \kappa = \kappa - c$.
 By Lemma \ref{lemma: traces} we see that in the limit $\delta \to 0$ we obtain $\theta_2 \equiv 0$ or $\theta_2 \equiv 1$ on $\ball{x_0}{\tilde \kappa} \cap F$, which concludes Step 8.

\textit{Step 9: Transport the information $\theta_i \equiv 0$ or $\theta_i \equiv 1$ on the face $F$ closest to the origin back into $P$.}\\
 Let $ I_\nu = (a_\nu,b_\nu)$ be the intervals obtained in Step 7.
 The proposition is proven once we can show that $a_\nu \geq -\tilde r <  \tilde r \leq  b_\nu $ for all $\nu \in N_2 \cup N_3$.
 Towards a contradiction we assume otherwise.
  Furthermore, for the sake of concreteness we assume that $b := b_{(101)} = \min_{\nu \in N_2 \cup N_3} \{-a_\nu,b_\nu\} < \tilde r$, i.e., we assume the face $F$ of $P$ we considered in the previous step to be the one closest to the origin.
  All other cases work the same.
  
  For $l(t):= b\frac{1}{\sqrt 2}(101) + \sqrt{2} tE_2$ we know by Step 8 that
 \[\theta_2 \circ l(t) = 0\]
  for almost all $t \in J:= l^{-1}(F\cap \ball{0}{r})$.
 Lemma \ref{lemma: almost maxima on transversal lines, 3D} implies that $\theta_2 \equiv 0$ on  the convex polyhedron
 \[Q:= \bigcap_{\nu \in N_1 \cup N_3} \left\{x\cdot \nu = \nu \cdot l(t) \text{ for some } t \in J\right\},\]
 see Figure \ref{fig:polyhedron_b} for a sketch relating $P$ and $Q$ in three dimensions.
 As any point of the closure $\overbar Q$ has positive density, we only have to prove $0 \in \overbar Q$ to get a contradiction to $0$ being a point of density one of the set
 \[\{\theta_1=0 \text{, } 0 < \theta_2, \theta_3 <1\},\]
 see Step 1.
 Furthermore, we may suppose that $b >0$ as that would imply $0 \in F$, which by $F \subset \overbar Q$ trivially gives the statement.
%  Note that in the case that $P$ is thin we are forced to use the closest face to propagate the dichotomy from, see Figure \ref{fig:why_close}.

\textit{Step 10: Prove $0 \in \overbar Q$, i.e., we can transport $\theta_2 = 0$ or $\theta_2 = 1$ to the origin.}\\
  To this end, let $x_\alpha := \sqrt{2} \alpha (101)$ for $\alpha >0$.
  In order to check $x_\alpha \in Q$ we calculate
  \begin{align*}
   x_\alpha \cdot \frac{1}{\sqrt 2} (011) = & \alpha,\\
   x_\alpha \cdot \frac{1}{\sqrt 2} (01\overline1) = & - \alpha,\\
   x_\alpha \cdot \frac{1}{\sqrt 2} (110) = & \alpha,\\
   x_\alpha \cdot \frac{1}{\sqrt 2} (1\overline10) = & \alpha.
  \end{align*}
  Consequently, we have
  \begin{align*}
   & \quad x_\alpha \cdot \frac{1}{\sqrt 2} (011)  = l(t) \cdot \frac{1}{\sqrt 2} (011)\text{ and } x_\alpha \cdot \frac{1}{\sqrt 2} (110)  = l(t) \cdot \frac{1}{\sqrt 2} (110) \\
   & \iff \alpha = \frac{1}{2} b + t \\
   & \iff t = \alpha - \frac{1}{2}b.
  \end{align*}
  For the $t \in \R$ in the previous line we have indeed $l(t) = b \frac{1}{\sqrt{2}}(1\overline11) + \sqrt{2} \alpha E_2\in \ball{0}{r}$ and $l(t) \in F$ due to the first equivalence above and
  \begin{alignat*}{4}
   a_{(01\overline1)} & < l(t) \cdot \frac{1}{\sqrt 2} (01\overline1) &  = - b + \alpha & < b_{(01\overline1)}\\
   a_{(1\overline10)} & < l(t) \cdot \frac{1}{\sqrt 2} (1\overline10)  &\, = b - \alpha & < b_{(1\overline10)}
  \end{alignat*}
  due to $a_{(01\overline1)} \leq - b < 0 \leq  b_{(01\overline1)}$ and $a_{(1\overline10)} \leq 0 <  b \leq b_{(1\overline10)}$.
  This proves
  \[x_\eps \in \left\{ \nu \cdot x =  \nu \cdot l(t) \text{ for some } t \in J \right\}\]
   for $\nu = \frac{1}{\sqrt 2}(110)$ and $\nu = \frac{1}{\sqrt 2} (011)$.
  
  Furthermore, we compute
  \begin{align*}
   & \quad x_\alpha \cdot \frac{1}{\sqrt 2} (01\overline1)  = l(t) \cdot \frac{1}{\sqrt 2} (01\overline1) \text{ and } x_\alpha \cdot \frac{1}{\sqrt 2} (1\overline10)  = l(t) \cdot \frac{1}{\sqrt 2} (1\overline10) \\
   & \iff \alpha= \frac{1}{2} b - t \\
   & \iff t = \frac{1}{ 2} b-\alpha
  \end{align*} 
  and, again for the $t \in \R$ given by the previous line, we have $l(t) = b\frac{1}{\sqrt 2}(111) - \sqrt{2}\alpha E_2 \in \ball{0}{r}$.
  We also have $l(t) \in F$ by the equivalence above and
  \begin{alignat*}{4}
   a_{(110)} & < l(t) \cdot \frac{1}{\sqrt 2} (110)  &  = b  -\alpha & < b_{(110)},\\
   a_{(011)} & < l(t) \cdot \frac{1}{\sqrt 2} (011)  &\, = b -\alpha & < b_{(011)},
  \end{alignat*}
  where we used $ a_{(110)} \leq 0 <b \leq  b_{(110)}$ and $a_{(011)} \leq 0 < b \leq b_{(011)}$.
  We thus have $x_\eps \in \left\{ \nu \cdot x =  \nu \cdot l(t) \text{ for some } t \in J \right\}$ for $\nu = \frac{1}{\sqrt 2}(01\overline1)$ and $\nu = \frac{1}{2}(1\overline10)$.
  As a result, we have $x_\alpha \in Q$, which ensures $0 \in \bar Q$ and finally concludes the proof.
\end{proof}

\begin{proof}[Proof of Lemma \ref{lemma: Mean value theorem for VMO}]
 The fact that $f_\delta= \dashint_{\ball{0}{\delta}} f(y)\intd y$ is continuous follows easily from the observation that $f_\delta$ is the convolution of $f$ with $\frac{1}{|\ball{0}{\delta}|} \chi_{\ball{0}{\delta}}$.
 
 As long as $\ball{x}{\delta} \subset U$, we have that
 \[\dist(f_\delta, K) = \inf_{\hat f \in K} | f_\delta(x) - \hat f| = \dashint_{\ball{x}{\delta}} \inf_{\hat f \in K} | f_\delta(x) - \hat f| \intd y \leq \dashint_{\ball{x}{\delta}} | f_\delta(x) - f(y)| \intd y  \to 0 \]
 uniformly in $x$ by definition of $VMO$.
\end{proof}

\begin{proof}[Proof of Lemma \ref{lemma:almost_affine}]
 By convolution (and restriction to a slightly smaller interval) we may suppose that $g$ is continuous.
 Without loss of generality we may additionally assume $g(0) = 0$.
 Recall $\eps := \sup_{t,t+h, t+\tilde h, t+ h + \tilde h \in [0,1]} | g(t+h+\tilde h) - g(t+ h) - g(t+\tilde h ) +g(t)|$.
 
 By induction, we can prove that for $x_i \geq 0$ with $1\leq i \leq n$ such that $\sum_{i=1}^n x_i \leq 1$ we have
 \[\left|g\left(\sum_{i=1}^n x_i\right) - \sum_{i=1}^n g(x_i)\right| \leq (n-1) \eps.\]
 Indeed, the case $n=1$ is trivial and the crucial part of the induction step is
  \[\left|g\left(\sum_{i=1}^{n-1} x_i + x_n\right) - \sum_{i=1}^{n-1} g(x_i) - g(x_n)\right| \leq \left|g\left(\sum_{i=1}^{n-1} x_i\right) - \sum_{i=1}^{n-1} g(x_i)\right| + \eps.\]
 
 In particular, for $x \in [0,1]$ and $n\in \mathbb{N}$ such that $nx \in [0,1]$ we have that
 \[\left|g\left(n x \right) -n g(x)  \right| \leq \left(n -1 \right) \eps,\numberthis \label{workhorse1}\]
 which implies
  \[\left|g\left(x \right) - \frac{1}{n} g(nx)  \right| \leq \eps.\numberthis \label{workhorse2}\]
 Choosing $|x| \leq \frac{1}{2}$ and $n = \left\lfloor \frac{1}{x}\right\rfloor$ in this inequality gives
 \[|g(x)| \leq \left\lfloor \frac{1}{x}\right\rfloor ^{-1} \left|g\left( \left\lfloor \frac{1}{x}\right\rfloor x\right) \right|  + \eps \leq 2x||g||_\infty + \eps,\]
where we used $\left\lfloor \frac{1}{x}\right\rfloor x  \geq \left( \frac{1}{x} -1 \right) x  = 1-x \geq \frac{1}{2}$.
%  \[\left|g\left(\left\lfloor \frac{1}{x}\right\rfloor x \right) -\left\lfloor \frac{1}{x}\right\rfloor g(x)  \right| \leq \left(\left\lfloor \frac{1}{x}\right\rfloor -1 \right) \eps.\] 
	For $x,y \in [0,1]$ with $|x-y|\leq \frac{1}{2}$ therefore get
 \[|g(x) - g(y)| \leq |g(|x-y|)| + \eps \leq 2|x-y|\,||g||_\infty + 2\eps.\]
 Plugging $x = \frac{1}{m}$, $n = k$ into estimate \eqref{workhorse1} and $x = \frac{1}{m}$, $n=m$ into estimate \eqref{workhorse2} for numbers $k, m \in \mathbb{N}$ with $k\leq m$ gives
 \[\left|g\left(\frac{k}{m}\right) - \frac{k}{m} g(1) \right| \leq \left|g\left(\frac{k}{m}\right) - k g\left(\frac{1}{m}\right) \right| + \left| k g\left(\frac{1}{m}\right) - \frac{k}{m} g(1) \right| \leq (2k-1)\eps \leq 2m\eps.\] 
 Additionally note that for $x\in [0,1]$ and $N \in \mathbb{N}$ we have
 \[\left| x - \frac{1}{N}\left\lfloor Nx \right\rfloor  \right| \leq \frac{1}{N}.\]
 
 Collecting all of the above, we have for $N \geq 2$ and $x\in [0,1]$ that
 \begin{align*}
  & \quad \left|g(x) - x g(1) \right| \\
  & \leq \left| g(x) - g\left(\frac{1}{N} \left\lfloor Nx\right\rfloor\right) \right| + \left| g\left(\frac{1}{N} \left\lfloor Nx\right\rfloor\right) - \frac{1}{N}\left\lfloor Nx \right\rfloor g(1) \right| + \left|\frac{1}{N}\left\lfloor Nx \right\rfloor - x \right| |g(1)| \\
  & \leq 2 \left| x - \frac{1}{N}\left\lfloor Nx \right\rfloor  \right|\, ||g||_\infty  + 2 \eps + 2N\eps + \left| x - \frac{1}{N}\left\lfloor Nx \right\rfloor  \right| \, ||g||_\infty\\
  & \leq \frac{3}{N}||g||_\infty +  4 N \eps.
 \end{align*}
 If $||g||_\infty ^\frac{1}{2} \eps^{-\frac{1}{2}} \geq 2$ we may choose $N\in \mathbb{N}$ with $N \geq 2$ such that
 \[||g||_\infty ^\frac{1}{2} \eps^{-\frac{1}{2}} \leq N < ||g||_\infty ^\frac{1}{2} \eps^{-\frac{1}{2}}+1\]
 and $\tilde g(x) := x g(1)$, which gives
 \[\left|\left|g - \tilde g \right|\right|_\infty \lesssim ||g||_\infty ^\frac{1}{2}\eps^{\frac{1}{2}} + \eps \leq \frac{3}{2} ||g||_\infty ^\frac{1}{2}\eps^{\frac{1}{2}}.\]
 If instead we have $||g||_\infty ^\frac{1}{2} \eps^{-\frac{1}{2}} < 2$ we set $\tilde g \equiv 0$ and get
 \[\left|\left|g - \tilde g \right|\right|_\infty \leq 2 \eps. \qedhere\]
\end{proof}

\subsection{Classification of planar configurations}

\begin{proof}[Proof of Lemma \ref{lemma: two_functions}]
 Without loss of generality, we may assume that $f_{\nu_1}$ is affine and that
 \[\theta_2 |_{H(\alpha,\nu_2)} = b \chi_{ B},\numberthis \label{2d_two_valued_2}\]
 where $B$ has non-vanishing measure.
 Absorbing $f_{\nu_1}$ into $g_2$ and $g_3$, as well as absorbing $g_1-1$ into $f_{\nu_2}$ and $f_{\nu_3}$, which we can do because $\partial_d g_1 = 0$ and the remaining variables are spanned by $\nu_2$ and $\nu_3$, we are left with
 \begin{align*}
  \theta_1(x) & = \phantom{{}-{}} f_{\nu_2}(x\cdot \nu_2) - f_{\nu_3}(x\cdot \nu_3) + 1,  \\
  \theta_2(x) & = \phantom{{}-{} f_{\nu_2}(x\cdot \nu_2) {}+{}} f_{\nu_3}(x\cdot \nu_3)  +g(x),\\
  \theta_3(x) &  = {}-{} f_{\nu_2}(x\cdot \nu_2) \phantom{{}+{} f_{\nu_3}(x\cdot \nu_3)}  - g(x) 
 \end{align*}
 for an affine function $g$ with $\partial_d g =0$.
 One of the two functions $f_{\nu_2}$ and $f_{\nu_3}$ cannot be affine as otherwise we would be dealing with a two-variant configuration by Proposition \ref{Prop: Rigidity VMO}.
 Therefore, there are two cases: Precisely one of the two remaining one-dimensional functions is affine, or both are not.
 
 Let us first deal with $f_{\nu_2}(x)$ being affine.
 We cannot have $|\theta_3^{-1}(0)|>0$, because two affine functions agreeing on a set of positive measure have to agree everywhere, which would imply $\theta_3 \equiv 0$ and thus there would only be two martensite variants present.
 We thus have $|\theta_1^{-1}(0)|>0$ and $|\theta_2^{-1}(0)|>0$. 
 The same argument applied to the $x\cdot \nu_2$-dependence of $\theta_1$ and $\theta_2$ implies that $f_{\nu_2}$ is constant and $g$ only depends on $x \cdot \nu_3$.
 Consequently, there exist $a$, $b \in \R$ such that the decomposition simplifies to
 \begin{align*}
  \theta_1(x) & =  {}- f_{\nu_3}(x\cdot \nu_3) + 1,  \\
  \theta_2(x) & =  \phantom{{}+{}} f_{\nu_3}(x\cdot \nu_3)  - a\, x\cdot \nu_3 - b,\\
  \theta_3(x) &  = \phantom{{}+{} f_{\nu_3}(x\cdot \nu_3)  +{}} a\, x\cdot \nu_3 + b. 
 \end{align*}
 For $x \in \ball{0}{r}$ such that $f_{\nu_3}(x) \neq 1$ we must have $\theta_2(x) = 0$, which implies that
 \[f_{\nu_3}(x\cdot \nu_3) = \chi_A(x\cdot \nu_3) + (a\, x\cdot \nu_3 +b)\chi_{\stcomp{A}}(x\cdot \nu_3)\]
 for some measurable set $A \subset \R$.
 Plugging this into the decomposition gives
 \begin{align*}
  \theta_1(x) & =  (1 - a\, x\cdot \nu_3 - b)\chi_{\stcomp{A}}(x\cdot \nu_3),  \\
  \theta_2(x) & =  (1 - a\, x\cdot \nu_3 - b)\chi_{A}(x\cdot \nu_3),\\
  \theta_3(x) &  = \phantom{(1 {}-{} }   a\, x\cdot \nu_3 + b, 
 \end{align*}
 i.e., the decomposition is a planar second-order laminate according to Definition \ref{def:second-order_laminate}. 
 The argument for $f_{\nu_3}$ being affine is the same.
 
 Finally, let us work with the case that both functions are not affine.
 Using the two-valuedness \eqref{2d_two_valued_2} on $H(\alpha,\nu_2)$, we may split up $g(x) = \tilde g_2 (x \cdot \nu_2) + \tilde g_3 (x\cdot \nu_3)$ into two affine functions such that $\tilde g_2(\alpha)=0$ and
 \[f_{\nu_3}(x\cdot \nu_3) + \tilde g_3(x\cdot \nu_3) = \theta_2(x) = b \chi_{ B}(x)\]
  for $x\in \ball{0}{r}$ with $x\cdot \nu_2 = \alpha$.
 Therefore $\chi_{B}$ captures the entire dependence on $x\cdot \nu_3$ and we abuse the notation in writing 
 \begin{align*}
  \theta_1(x) & = \phantom{{}-{}} f_{\nu_2}(x\cdot \nu_2) - f_{\nu_3}(x\cdot \nu_3)\,\, + 1,  \\
  \theta_2(x) & = \phantom{{}-{} f_{\nu_2}(x\cdot \nu_2) {}+{}} b\chi_B(x\cdot \nu_3) + \tilde g_2(x\cdot \nu_2),\\
  \theta_3(x) &  = {}-{} f_{\nu_2}(x\cdot \nu_2) \phantom{{}+{} f_{\nu_3}(x\cdot \nu_3)}  \,\,-g(x).
 \end{align*}
 As $f_{\nu_3}$ is not affine, the set $B$ has neither zero nor full measure.
 Choosing $x$ such that $\chi_B(x\cdot \nu_3)= 0$ we see that $\tilde g_2 \geq 0$.
 Thus it is an affine function which achieves its minimum at $\tilde g_2(\alpha) =0$, which in turn makes sure that $\tilde g_2 \equiv 0$.
 Consequently, we can re-define the functions on the right-hand side to get
 \begin{align*}
  \theta_1(x) & = \phantom{{}-{}} f_{\nu_2}(x\cdot \nu_2) -b\chi_B(x\cdot \nu_3) + 1,  \\
  \theta_2(x) & = \phantom{{}-{} f_{\nu_2}(x\cdot \nu_2) {}+{}} b\chi_B(x\cdot \nu_3) ,\\
  \theta_3(x) &  = {}-{} f_{\nu_2}(x\cdot \nu_2) \phantom{{}+{} b\chi_B(x\cdot \nu_3)}  - \tilde g_3(x\cdot \nu_3).
 \end{align*}
 For $x$ such that $x\cdot \nu_3 \in B$ we see that
 \[\theta_1(x) = 1-b,\, \theta_3(x) = 0 \text{ or } \theta_1(x) = 0,\, \theta_3(x) = 1-b.\]
 This implies $f_{\nu_2}  = - (1-b)\chi_A$ for a measurable set $A$ of neither zero nor full measure, since $f_{\nu_2}$ is not affine.
 On the set $\{x\cdot \nu_2 \in \stcomp{A}\} \cap \{x\cdot \nu_3 \in B\}$ of positive measure we get that $\theta_3(x) = 0$ due to our assumption that $0<b<1$, resulting in $\tilde g_3 \equiv 0$.
 Hence the decomposition can be written as
 \begin{align*}
  \theta_1(x) & = - (1-b)\chi_A(x\cdot \nu_2) -b\chi_B(x\cdot \nu_3) + 1,  \\
  \theta_2(x) & = \phantom{- (1-b)\chi_A(x\cdot \nu_2) {}+{}} b\chi_B(x\cdot \nu_3) ,\\
  \theta_3(x) &  =\phantom{-} (1-b)\chi_A(x\cdot \nu_2),
 \end{align*}
 meaning the configuration is a planar checkerboard according to Definition \ref{def:checkerboard}.
\end{proof}

\begin{proof}[Proof of Proposition \ref{prop: six_corner}]
We denote the fixed radius for which the assumptions of the lemma hold by $\tilde r$, while $r > \tilde r$ is a generic radius that may decrease from line to line.

\textit{Step 1: Rewrite the problem in a two-dimensional domain and bring the decomposition \eqref{decomposition} into an appropriate form.}\\
 Using the specific form of the normals $\nu_i$ and the fact that they are linearly independent, we can find orientations $\tilde \nu_i = \pm \nu_i$ for $i=1,2,3$ which satisfy $\tilde \nu_1+\tilde \nu_2+\tilde \nu_3 =0$.
 Furthermore, the strain $e(u)$ only depends on directions in $V := \operatorname{span}
 (\tilde \nu_1,\tilde \nu_2,\tilde \nu_3)$.
 Thus we can rotate the domain of definition such that $V = \R^2$ and treat $e(u)$ as a function defined on $\ball{0}{1} \subset \R^2$.
 In the following we will abuse the notation by writing $\nu_i$ for the images of $\tilde \nu_i$ under this rotation.
 
 The condition $\nu_1+\nu_2 + \nu_3 =0$ implies that
 \begin{align*}
  - \nu_1 \cdot \nu_2 - \nu_1\cdot \nu_3 \phantom{ {}-{} \nu_2\cdot \nu_3 } & = 1, \\
  - \nu_1 \cdot \nu_2 \phantom{ {}-{} \nu_1\cdot \nu_3} - \nu_2\cdot \nu_3  & = 1, \\
  \phantom{ {}-{} \nu_1 \cdot \nu_2} - \nu_1\cdot \nu_3 - \nu_2\cdot \nu_3  & = 1,
 \end{align*}
 which by elementary calculation gives $\nu_i\cdot \nu_j = -\frac{1}{2}$ for $i,j=1,2,3$ and $i\neq j$.
 Thus $\{\nu_i, \nu_j\}$ is a basis of $\R^2$ and the angle between the two vectors is universally bounded away from zero.
 In fact, it is given by $120\degree$, see Figure \ref{fig:sketch_intervals_fitting}.
  
  Furthermore, we rewrite the decomposition \eqref{reduced_decomposition} as
 \begin{align}
  \theta_1(x) & = \phantom{f_1^{(2)}(x\cdot \nu_1) {}+{} }f_2^{(1)}(x\cdot \nu_2) + f_3^{(1)}(x\cdot \nu_3), \notag \\
  \theta_2(x) & = f_1^{(2)}(x\cdot \nu_1) \phantom{{}+{} f_2^{(1)}(x\cdot \nu_2)} + f_3^{(2)}(x\cdot \nu_3), \label{decompostion_3_func_no_aff} \\
  \theta_3 (x) &  = f_1^{(3)}(x\cdot \nu_1) + f_2^{(3)}(x\cdot \nu_2), \phantom{{}+{} f_3^{(1)}(x\cdot \nu_3)} \notag    
 \end{align}
  where $f_k^{(i)} + f_k^{(j)}$ is affine almost everywhere for $\{i,j,k\}=\{1,2,3\}$ and all one-dimensional functions are non-constant in $L^\infty(\ball{0}{\tilde r})$.
  We may do so since for all $i=1,2,3$ the functions $g_i$ only depend on variables in $V$ for which any two of the three normals $\nu_j$, $j=1,2,3$, form a basis.
  
 \textit{Step 2: If $\left|\theta_i^{-1}(0)\cap\ball{0}{r}\right| > 0$ for some $i=1,2,3$ we re-define $f^{(i)}_{i+1}$ and $f^{(i)}_{i-1}$ to satisfy $f^{(i)}_{i+1},f^{(i)}_{i-1}\geq 0$ on $[- r, r]$ and $f^{(i)}_{i+1} = f^{(i)}_{i-1} =  0$ on $\theta_i^{-1}(0)\cap\ball{0}{r}$.}\\
  For almost all $x\in \theta_i^{-1}(0)\cap\ball{r}{0}$ we have
  \begin{align*}
  0 & = f^{(i)}_{i+1}(x\cdot \nu_{i+1}) + f^{(i)}_{i-1}(x\cdot \nu_{i-1}) \\
  & \geq \essinf_{[-r,r]} f^{(i)}_{i+1} + f^{(i)}_{i-1} (x\cdot \nu_{i-1}) \\
 & \geq \essinf_{[-r,r]} f^{(i)}_{i+1} + \essinf_{[-r,r]} f^{(i)}_{i-1} \\
 & \geq 0,
  \end{align*}
  where in the last step we used Lemma \ref{lemma: almost maxima on transversal lines} for large $\eps>0$.
  Fubini's Theorem thus implies
 $f^{(i)}_{i+1} = \essinf_{[-r,r]} f^{(i)}_{i+1}$ and $f^{(i)}_{i-1} = \essinf_{[-r,r]} f^{(i)}_{i-1}$ on sets of positive measure.
 Shuffling around some constant, we may assume that $\essinf f^{(i)}_{i+1} = \essinf f^{(i)}_{i-1} = 0$.

\textit{Step 3: There exist measurable sets $J_j \subset \R$ for $j=1,2,3$ such that  \[\theta_i^{-1}(0)\cap \ball{0}{r}  =  \pi_{i+1}^{-1}\left(J_{i+1}\right) \cap \pi_{i-1}^{-1}\left(\stcomp{J_{i-1}}\right) \cap \ball{ r}{0}\]
up to null-sets and the two sets
\begin{align*}
 \ball{0}{r} & \cap \left( \pi_1^{-1}(J_1)\cap\pi_2^{-1}(J_2)\cap\pi_3^{-1}(J_3)\right),\\
 \ball{0}{r} & \cap \left( \pi_1^{-1}(\stcomp{ J_1})\cap\pi_2^{-1}(\stcomp{J_2})\cap\pi_3^{-1}(\stcomp{J_3})\right)
\end{align*}
have measure zero.}\\ 
 If $\left|\theta_i^{-1}(0)\cap\ball{0}{r}\right| > 0$ we set 
 \begin{align}\label{f_vanish}
  I^{(i)}_{i+1}:= \left(f^{(i)}_{i+1}\right)^{-1}(0) \cap [-r, r]\text{, } I^{(i)}_{i-1}:= \left(f^{(i)}_{i-1}\right)^{-1}(0) \cap [-r, r].
 \end{align}
 Otherwise we set $I^{(i)}_{i+1}=I^{(i)}_{i-1}= \emptyset$.
 In any case we have
 \[\theta_i^{-1}(0) \cap \pi_{i+1}^{-1}([- r, r]) \cap \pi_{i-1}^{-1}([- r, r]) = \pi_{i+1}^{-1}\left(I^{(i)}_{i+1}\right) \cap \pi_{i-1}^{-1}\left(I^{(i)}_{i-1}\right)\]
 up to null-sets.
 
 \textit{Claim 3.1: We have $\left|I^{(i)}_k \cap I^{(j)}_k\right| = 0$ for $\{i,j,k\} = \{1,2,3\}$.}\\ 
 If $|\theta_i^{-1}(0)|=0$ or $|\theta_j^{-1}(0)|=0$ then there is nothing to prove.
 Otherwise we assume towards a contradiction that
 \[\left|I^{(i)}_k \cap I^{(j)}_k\right| > 0.\]
 In that case the affine function $f^{(i)}_{k}+f^{(j)}_{k}$ vanishes on a set of positive measure.
 Thus we have $f^{(i)}_{k} \equiv -f^{(j)}_{k}$.
 Since both functions are non-negative on $[-r, r]$ we get $f^{(i)}_{k} \equiv f^{(j)}_{k} \equiv 0$ on $[-r, r]$.
 However, this contradicts our assumption that they are non-constant.
 Thus we have
 \[\left|I^{(i)}_k \cap I^{(j)}_k \right| = 0,\]
  which proves Claim 3.1.

 Consequently we get, up to null-sets,
 \[\theta_i^{-1}(0)\cap \ball{0}{r}  \subset  \pi_{i+1}^{-1}\left(I^{(i)}_{i+1}\right) \cap \pi_{i-1}^{-1}\left( \stcomp{\left(I^{(i+1)}_{i-1}\right)}\right),\]
 which in terms of
  \begin{align}\label{def_J}
    J_j:= I_{j}^{(j-1)}\text{ for } j=1,2,3
 \end{align}
 reads, up to null-sets,
  \[\theta_i^{-1}(0)\cap \ball{0}{r}  \subset  \pi_{i+1}^{-1}\left(J_{i+1}\right) \cap \pi_{i-1}^{-1}\left( \stcomp{J_{i-1}}\right).\]
 
 Since the sets $\pi_{i+1}^{-1}\left(J_{i+1}\right) \cap \pi_{i-1}^{-1}\left( \stcomp{J_{i-1}}\right)$ are pairwise disjoint for $i=1,2,3$ and, again up to null-sets, we have $\bigcup_{i=1,2,3} \theta_i^{-1}(0) \cap \ball{0}{r}= \ball{ r}{0}$ we get that
 \[\theta_i^{-1}(0)\cap \ball{0}{r} =  \pi_{i+1}^{-1}\left(J_{i+1}\right) \cap \pi_{i-1}^{-1}\left( \stcomp{J_{i-1}}\right) \cap \ball{0}{r}\]
 up to null-sets.
 
 Some straightforward combinatorics ensure that
 \begin{align*}
    \ball{0}{r} & \setminus  \left(\bigcup_{i=1,2,3} \pi_{i+1}^{-1}\left(J_{i+1}\right) \cap \pi_{i-1}^{-1}\left(\stcomp{J_{i-1}}\right) \right)\\
     = & \, \ball{0}{r} \cap \left( \left( \pi_1^{-1}(J_1)\cap\pi_2^{-1}(J_2)\cap\pi_3^{-1}(J_3)\right)\cup\left( \pi_1^{-1}(\stcomp{ J_1})\cap\pi_2^{-1}(\stcomp{J_2})\cap\pi_3^{-1}(\stcomp{J_3})\right)\right).
 \end{align*}
 Thus we have
 \begin{align*}
  \left| \ball{0}{r} \cap \left( \pi_1^{-1}(J_1)\cap\pi_2^{-1}(J_2)\cap\pi_3^{-1}(J_3)\right)\right| & = 0, \\
  \left| \ball{0}{r} \cap \cup\left( \pi_1^{-1}(\stcomp{ J_1})\cap\pi_2^{-1}(\stcomp{J_2})\cap\pi_3^{-1}(\stcomp{J_3})\right)\right| & = 0.
 \end{align*}
 This finishes the proof of Step 3.

 \textit{Step 4: The conclusion of the lemma holds.}\\
 We now make sure that we can apply Lemma \ref{lemma: intervals}.
 To this end, we choose $\tilde r$ small enough such that we can use Lemma \ref{lemma: intervals} after rescaling $\ball{0}{r}$ to $\ball{0}{1}$.
 
 By assumption there are $i,j = 1,2,3$ with $i\neq j$ such that
 \begin{align*}
   \left|\theta_i^{-1}(0)\cap \pi_i^{-1}\left(\left[-\frac{\tilde r}{2},\frac{\tilde r}{2}\right]\right) \cap \pi_j^{-1}\left(\left[-\tilde r,
   \tilde r\right]\right)\right| & >0, \\
   \left|\theta_j^{-1}(0)\cap \pi_i^{-1}\left(\left[-\frac{\tilde r}{2}, \frac{\tilde r}{2}\right]\right) \cap \pi_j^{-1}([-\tilde r,\tilde r]) \right| & >0.
 \end{align*}
 By relabeling we may suppose $i=3$ and $j=1$.
 Consequently we get
 \[|J_1| , | \stcomp{J_2}| , |J_2| , |\stcomp{J_3}| >0.\]
 As $f_{\nu_1}^{(3)} = 0$ on $J_1$ and $f_{\nu_1}^{(3)}\not\equiv 0$ we must have $|\stcomp{J_1}| >0$.
 The upshot is that we have $0 < \left|J_1\cap \left[-\frac{\tilde r}{2}, \frac{\tilde r}{2} \right]\right| < \tilde r = \left| \left[-\frac{\tilde r}{2}, \frac{\tilde r}{2} \right] \right|  $ and $0< \left|J_2\cap \left[-\frac{\tilde r}{2},\frac{\tilde r}{2}\right]\right| < \tilde r $.

 Lemma \ref{lemma: intervals} implies that there exists a point $x_0 \in \ball{0}{1}$ such that $x_0 \cdot \nu_i \in (-\tilde r,\tilde r)$ for all $i=1,2,3$, up to sets of measure zero, we have either
 \[J_i \cap [-\tilde r, \tilde r] = [-\tilde r,x_0\cdot \nu_i] \text{ for } i=1,2,3\]
 or
 \[J_i \cap [-\tilde r,\tilde r] = [-x_0\cdot \nu_i,\tilde r] \text{ for } i=1,2,3.\]
 Let $K_i := J_i\cap [-\tilde r,\tilde r]$.
 Tracing back the definitions using \eqref{def_J}, Claim 3.1 and \eqref{f_vanish}, we see that on $K_{i+1}$ we have $f_{i+1}^{(i)}=0$ and on $[-r,r]\setminus K_{i-1}$ we have $f_{i-1}^{(i)}=0$.
 As a result we can rewrite the decomposition \eqref{reduced_decomposition} of $\theta$ on $\ball{0}{r}$ to be
  \begin{align}
  \theta_1(x) & = \phantom{f_1^{(2)}(x\cdot \nu_1)\chi_{\stcomp{K_1}}{}+{}}  f_2^{(1)}(x\cdot \nu_2)\chi_{\stcomp{K_2}}(x\cdot\nu_2) + f_3^{(1)}(x\cdot \nu_3)\chi_{K_3}(x\cdot\nu_3) ,\notag\\
  \theta_2(x) & = f_1^{(2)}(x\cdot \nu_1)\chi_{K_1}(x\cdot\nu_1)  \phantom{{}+{} f_2^{(1)}(x\cdot \nu_2)\chi_{\stcomp{\stcomp{K_2}}}} + f_3^{(2)}(x\cdot \nu_3) \chi_{\stcomp{K_3}}(x\cdot\nu_3) ,\label{reduced_decomposition_intermediate}\\
  \theta_3 (x) &  = f_1^{(3)}(x\cdot \nu_1)\chi_{\stcomp{K_1}}(x\cdot\nu_1) + f_2^{(3)}(x\cdot \nu_2)\chi_{K_2}(x\cdot\nu_2) \phantom{{}+{} f_3^{(1)}(x\cdot \nu_3)\chi_{\stcomp{\stcomp{K_3}}}} .\notag
 \end{align}
 The condition that certain sums of the one-dimensional functions are affine turns into
 \[\left(f_i^{(i-1)}\chi_{\stcomp{K}_{i-1}} + f_i^{(i+1)}\chi_{K_{i+1}}\right)(t) = a_it + b_i\] for $t\in (-r,r)$, $a_i, b_i \in \R$ and $i=1,2,3$.

 Due to $\sum_{i=1}^3\theta_1 \equiv 1$, summing the equations in the decomposition \eqref{reduced_decomposition_intermediate} gives
 \[\sum_{i=1}^3 a_i x\cdot \nu_i + b_i = 1\]
 for all $x \in  \ball{0}{r}$.
 Comparing the coefficients of both polynomials we see that
 \[\sum_{i=1}^3 b_i =1\text{, } \sum_{i=1}^3 a_i \nu_i = 0.\]
 Subtracting $a_1(\nu_1 + \nu_2 + \nu_3) = 0$ from the second equation and remembering from Step 1 that  $\nu_2$ and $\nu_3$ are linearly independent, we see that $a:= a_1= a_2 = a_3$.
 We thus get
 \begin{align*}
  \theta_1(x) & = \phantom{(ax\cdot \nu_1 + b_1)\chi_{\stcomp{K_1}}{}+{}}  (a x\cdot \nu_2 +b_2)\chi_{\stcomp{K_2}}(x\cdot\nu_2) + (ax\cdot \nu_3 + b_3)\chi_{K_3}(x\cdot\nu_3) ,\\
  \theta_2(x) & = (ax\cdot \nu_1+b_1)\chi_{K_1}(x\cdot\nu_1)  \phantom{{}+{} (ax\cdot \nu_2 + b_2)\chi_{\stcomp{\stcomp{K_2}}}} + (ax\cdot \nu_3 + b_3) \chi_{\stcomp{K_3}}(x\cdot\nu_3) ,\\
  \theta_3 (x) &  = (ax\cdot \nu_1 + b_1)\chi_{\stcomp{K_1}}(x\cdot\nu_1) + (ax\cdot \nu_2 + b_2)\chi_{K_2}(x\cdot\nu_2) \phantom{{}+{} (ax\cdot \nu_3+ b_3)\chi_{\stcomp{\stcomp{K_3}}}}
 \end{align*}
 with $\sum_{i=1}^3 b_i = 1$.
\end{proof}

\begin{proof}[Proof of Lemma \ref{lemma: intervals}]
 Let $r>0$ be small enough such that
 \[\pi_i^{-1}([-r,r])\cap \pi_j^{-1}([-r,r]) \subset \ball{0}{1}\]
 for $i,j=1,2,3$ with $i\neq j$.
 Let $K_i:= J_i\cap [-r,r]$.
 
 \textit{Claim 1: There exist $a_1, a_2 \in (-r,r)$ such that, up to null-sets, either
 \[K_1 = [-r,a_1]\text{ and } K_2 = [-r,a_2]\]
 or
 \[K_1 = [a_1,r]\text{ and } K_2 = [a_2,r].\]}
 \vspace{-\baselineskip}\\
 Towards a contradiction we assume the negation of Claim 1.
 
 \textit{Step 1.1: Up to symmetries of the problem, find Lebesgue points $-r < p_1 < p_2 < r$ of $\chi_{K_1}$ and $-r < q_1 < q_2 < r$ of $\chi_{K_2}$ such that
 \[\chi_{K_1}(p_1) = \chi_{K_2}(q_2)=  1\text{ and }\chi_{K_1}(p_2) = \chi_{K_2}(q_1) = 0.\]}
 \vspace{-\baselineskip}\\
 The negation of Claim 1 implies that there exist Lebesgue points $-r < p_1 < p_2 < r$ of $\chi_{K_1}$ and $-r < q_1 < q_2 < r$ of $\chi_{K_2}$ such that
%  PICTURE!!!!!!
 \begin{align*}
   \chi_{K_1}(p_1) & \neq \chi_{K_1}(p_2),\\
   \chi_{K_2}(q_1) & \neq \chi_{K_2}(q_2),\\
   \chi_{K_1}(p_1) & \neq \chi_{K_2}(q_1),\\
   \chi_{K_1}(p_2) & \neq \chi_{K_2}(q_2):\\
 \end{align*} 
 If, up to null-sets, both are intervals with one having an endpoint at $-r$ and the other one having an endpoint at $r$, then one may take, for $\delta >0$ small enough, $p_1 := a_1 - \delta$, $p_2 := a_1 + \delta$, $q_1 := a_2 -\delta$ and $q_2 := a_2 + \delta$, see Figure \ref{fig:triple_boundary_intervals}.
 
 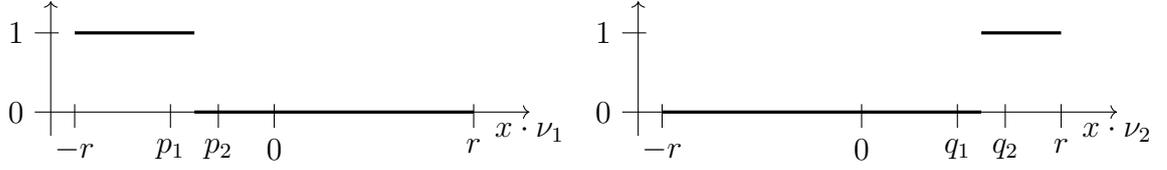
\begin{figure}
 \centering
%  \subcaptionbox{\label{fig:triple_pointvalues_a}}{
  \begin{tikzpicture}[scale=1.05]
    \draw[->] (-3,0)node[left]{$0$} -- (3.2,0) node[below]{$x\cdot \nu_1$};
    \draw (-2.5,-.2)node[below]{$-r$} -- (-2.5,.1);
    \draw (0,-.2)node[below]{$0$} -- (0,.1);
    \draw (2.5,-.2)node[below]{$r$} -- (2.5,.1);
    \draw[->] (-2.8,-.3) -- (-2.8,1.4);
    \draw (-3,1)node[left]{$1$} -- (-2.7,1);
    
    \draw[very thick] (-2.5,1) -- (-1,1);
    \draw[very thick] (-1,0) -- (2.5,0);
    \draw (-1.3,-.2)node[below]{$p_1$} -- (-1.3,.1);
    \draw (-.7,-.2)node[below]{$p_2$} -- (-.7,.1);    
  \end{tikzpicture}
%  }
%  \subcaptionbox{\label{fig:triple_pointvalues_b}}{
  \begin{tikzpicture}[scale=1.05]
    \draw[->] (-3,0)node[left]{$0$} -- (3.2,0) node[below]{$x\cdot \nu_2$};
    \draw (-2.5,-.2)node[below]{$-r$} -- (-2.5,.1);
    \draw (0,-.2)node[below]{$0$} -- (0,.1);
    \draw (2.5,-.2)node[below]{$r$} -- (2.5,.1);
    \draw[->] (-2.8,-.3) -- (-2.8,1.4);
    \draw (-3,1)node[left]{$1$} -- (-2.7,1);
    
    \draw[very thick] (-2.5,0) -- (1.5,0);
    \draw[very thick] (1.5,1) -- (2.5,1);
    \draw (1.2,-.2)node[below]{$q_1$} -- (1.2,.1);
    \draw (1.8,-.2)node[below]{$q_2$} -- (1.8,.1);        
  \end{tikzpicture}
%  }
 \caption{Graphs of $\chi_{K_1}$ and $\chi_{K_2}$ in the case that $K_1$ and $K_2$ are intervals such that one of them has an endpoint at $-r$ and the other one at $r$. In this case we choose $p_1, p_2$ and $q_1, q_2$ on opposite sides of the respective other endpoint.
 }
 \label{fig:triple_boundary_intervals}
\end{figure}
  
 If $K_1$ is not an interval with one endpoint at $-r$ or $r$, see Figure \ref{fig:triple_boundary_general}, there exist three Lebesgue points $\bar p_1<\bar p_2< \bar p_3$ such that $\theta_1(\bar p_1)\neq \theta_1(\bar p_2) \neq \theta_1(\bar p_3)$.
 Since $K_2$ has neither full nor zero measure, there exist Lebesgue points $q_1< q_2$ with $\chi_{K_2}(q_1) \neq \chi_{K_2}(q_2)$.
 In the case $\chi_{K_2}(q_1)\neq \chi_{K_1}(\bar p_1)$, set $p_1 := \bar p_1$ and $p_2 := \bar p_2$.
 Otherwise set $p_1 := \bar p_2$ and $p_2:= \bar p_3$.
  
 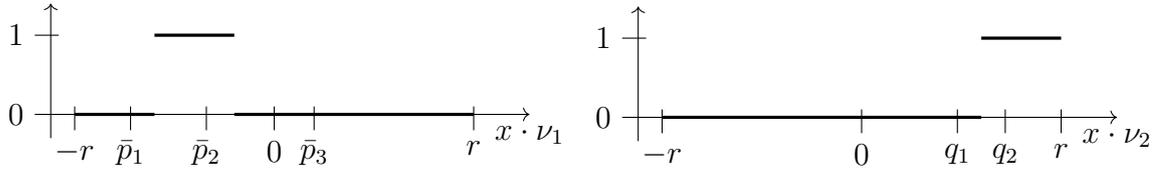
\begin{figure}
 \centering
%  \subcaptionbox{\label{fig:triple_pointvalues_a}}{
  \begin{tikzpicture}[scale=1.05]
    \draw[->] (-3,0)node[left]{$0$} -- (3.2,0) node[below]{$x\cdot \nu_1$};
    \draw (-2.5,-.2)node[below]{$-r$} -- (-2.5,.1);
    \draw (0,-.2)node[below]{$0$} -- (0,.1);
    \draw (2.5,-.2)node[below]{$r$} -- (2.5,.1);
    \draw[->] (-2.8,-.3) -- (-2.8,1.4);
    \draw (-3,1)node[left]{$1$} -- (-2.7,1);

    \draw[very thick] (-2.5,0) -- (-1.5,0);
    \draw[very thick] (-1.5,1) -- (-.5,1);
    \draw[very thick] (-.5,0) -- (2.5,0);
    \draw (-1.8,-.2)node[below]{$\bar p_1$} -- (-1.8,.1);
    \draw (-.85,-.2)node[below]{$\bar p_2$} -- (-.85,.1);
    \draw (.5,-.2)node[below]{$\bar p_3$} -- (.5,.1);
  \end{tikzpicture}
%  }
%  \subcaptionbox{\label{fig:triple_pointvalues_b}}{
  \begin{tikzpicture}[scale=1.05]
    \draw[->] (-3,0)node[left]{$0$} -- (3.2,0) node[below]{$x\cdot \nu_2$};
    \draw (-2.5,-.2)node[below]{$-r$} -- (-2.5,.1);
    \draw (0,-.2)node[below]{$0$} -- (0,.1);
    \draw (2.5,-.2)node[below]{$r$} -- (2.5,.1);
    \draw[->] (-2.8,-.3) -- (-2.8,1.4);
    \draw (-3,1)node[left]{$1$} -- (-2.7,1);
    
    \draw[very thick] (-2.5,0) -- (1.5,0);
    \draw[very thick] (1.5,1) -- (2.5,1);
    \draw (1.2,-.2)node[below]{$q_1$} -- (1.2,.1);
    \draw (1.8,-.2)node[below]{$q_2$} -- (1.8,.1);        
  \end{tikzpicture}
%  }
 \caption{Graphs of $\chi_{K_1}$ and $\chi_{K_2}$ in the case that $K_1$ is not an interval with one endpoint at $-r$ or $r$. In this specific instance we choose $p_1 = \bar p_2$ and $p_2 := \bar p_3$.
 }
 \label{fig:triple_boundary_general}
\end{figure}
  
 If $K_2$ is not an interval with one endpoint at $-r$ or $r$, the same reasoning applies.
 
 Furthermore, we may assume $\chi_{K_1}(p_1) = 1$ because the statement of the lemma is clearly invariant under replacing all sets by their complements.
 The above collection of unordered inequalities then turns into $\chi_{K_1}(p_1) = \chi_{K_2}(q_2) = 1$ and $\chi_{K_1}(p_2) = \chi_{K_2}(q_1) = 0$.
 
 \textit{Step 1.2: Find $\delta >0$ and $s_1, s_2 \in (-r+ \delta, r-\delta)$ such that for $K_1^<:= K_1 \cap (s_1 - \delta, s_1)$, $K_1^> := K_1 \cap (s_1, s_1 + \delta)$, $K_2^<:= K_2 \cap (s_2 - \delta, s_2)$ and $K_2^> := K_2 \cap (s_2, s_2 + \delta)$ we have
 \[|K_1^<| > | K_1^>| \text{ and } |K_2^> | > |K_2^<|,\] see Figure \ref{fig:triple_approx_boundary}.}\\ 
 By the virtue of $p_i$ and $q_i$ being Lebesgue points, there exists $\tilde \delta > 0$ such that we have $p_i \pm 3\tilde \delta$, $q_i \pm 3\tilde \delta \in [-r,r]$ and
 \begin{align*}
  \dashint_{p_1-\tilde \delta}^{p_1+\tilde \delta} \chi_{K_1} \intd t, &  \dashint_{q_2-\tilde \delta}^{q_2+\tilde \delta} \chi_{K_2} \intd t > \frac{3}{4}, \\
  \dashint_{p_2-\tilde \delta}^{p_2+\tilde \delta} \chi_{K_1} \intd t, &  \dashint_{q_1-\tilde \delta}^{q_1+\tilde \delta} \chi_{K_2} \intd t < \frac{1}{4}.
 \end{align*}
 Since the map $s\mapsto \dashint_{s-\tilde \delta}^{s+\tilde \delta} \chi_{K_1} \intd t$ is continuous, there exists
 \[\tilde s_1 := \max\left\{p_1 \leq  s \leq  p_2 : \dashint_{s-\tilde \delta}^{s+\tilde \delta} \chi_{K_1} \intd t = \frac{1}{2} \right\}.\]

\begin{figure}
 \centering
%  \subcaptionbox{\label{fig:triple_pointvalues_a}}{
  \begin{tikzpicture}[scale=1.05]
    \draw[->] (-3,0)node[left]{$0$} -- (3.2,0) node[below]{$x\cdot \nu_1$};
    \draw[->] (-2.8,-.3) -- (-2.8,1.4);
    \draw (-3,1)node[left]{$1$} -- (-2.7,1);
    
    \draw[very thick, color=gray] (-2.5,1) --(-1.5,1);
    \draw[very thick] (-1.5,1)--  (-1.2,1);
    \draw[very thick] (-1.2,0) -- (-1,0);
    \draw[very thick] (-1,1) -- (-.5,1);
    \draw[very thick] (-.5,0) -- (.1,0);
    \draw[very thick] (.1,1) -- (.3,1);
    \draw[very thick] (.3,0) -- (.5,0);
    \draw[very thick, color=gray] (.5,0) -- (2.5,0);
    
    \node at (-1,1.3) {$\chi_{K_1^<}$};
    \node at (0,1.3) {$\chi_{K_1^>}$};
    \draw[{Parenthesis[scale=1.75]}-{Parenthesis[scale=1.75]}] (-1.5,0) node[below]{$s_1-\delta$}--(-.5,0) node[below]{\phantom{$\delta$}$s_1$\phantom{$\delta$}};
    \draw[{Parenthesis[scale=1.75]}-{Parenthesis[scale=1.75]}] (-.5,0)-- (.5,0) node[below]{$s_1+\delta$};
%     \draw (-1.3,-.2)node[below]{$p_1$} -- (-1.3,.1);
%     \draw (-.7,-.2)node[below]{$p_2$} -- (-.7,.1);    
  \end{tikzpicture}
%  }
%  \subcaptionbox{\label{fig:triple_pointvalues_b}}{
  \begin{tikzpicture}[scale=1.05]
    \draw[->] (-3,0)node[left]{$0$} -- (3.2,0) node[below]{$x\cdot \nu_2$};
    \draw[->] (-2.8,-.3) -- (-2.8,1.4);
    \draw (-3,1)node[left]{$1$} -- (-2.7,1);
    
    \draw[very thick, color=gray] (-2.5,0) -- (-.7,0);
    \draw[very thick] (-.7,0) --  (-.5,0);
    \draw[very thick] (-.5,1) -- (-.35,1);
    \draw[very thick] (-.35,0) -- (0.3,0);
    \draw[very thick] (0.3,1) -- (.55,1);
    \draw[very thick] (.55,0) -- (.7,0);
    \draw[very thick] (.7,1) -- (1.3,1);
    \draw[very thick, color=gray] (1.3,1) -- (2.5,1);
    
    \node at (-.2,1.3) {$\chi_{K_2^<}$};
    \node at (.8,1.3) {$\chi_{K_2^>}$};
    \draw[{Parenthesis[scale=1.75]}-{Parenthesis[scale=1.75]}] (-.7,0) node[below]{$s_2-\delta$}-- (.3,0) node[below]{\phantom{$\delta$}$s_2$\phantom{$\delta$}};
    \draw[{Parenthesis[scale=1.75]}-{Parenthesis[scale=1.75]}] (.3,0)--(1.3,0)node[below]{$s_2+\delta$};
%     \draw (1.2,-.2)node[below]{$q_1$} -- (1.2,.1);
%     \draw (1.8,-.2)node[below]{$q_2$} -- (1.8,.1);        
  \end{tikzpicture}
%  }
 \caption{The sets $K^<_1$, $K^>_1$, $K^<_2$ and $K^>_2$ locally split up $K_1$ and $K_2$. The irrelevant parts of the graphs of $\chi_{K_1}$ and $\chi_{K_2}$ are shown in gray.
 }
 \label{fig:triple_approx_boundary}
\end{figure}
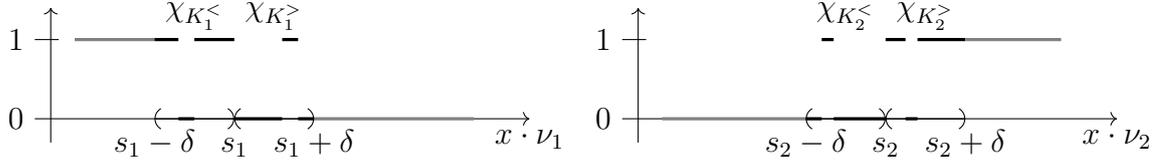
 
 Let $s_1:= \tilde s_1 +\tilde \delta$ and $\delta := 2 \tilde \delta$.
 Then we have 
 \[|K_1 \cap (s_1 - \delta, s_1)|= \frac{1}{2} > |K_1 \cap (s_1, s_1 + \delta)|,\]
 which with the notation $K_1^<= K_1 \cap (s_1 - \delta, s_1)$ and $K_1^> = K_1 \cap (s_1, s_1 + \delta)$ reads
 \[|K_1^< | > |K_1^>|.\]
 Using the same reasoning we can find $s_2 \in [-r+\delta , r-\delta]$ such that for $K_2^< = K_2 \cap (s_1-\delta , s_1)$ and $K_2^> = K_2 \cap (s_2 , s_2 + \delta )$ we get
 \[|K_2^>| > |K_2^<|.\]
 
 \textit{Step 1.3: Derive the contradiction.}\\ 
 Let $C_1:=\pi_1^{-1}(s_1-\delta,s_1)\cap \pi_2^{-1}(s_2, s_2 +\delta)$ and $C_2:=\pi_1^{-1}(s_1,s_1 +\delta)\cap \pi_2^{-1}(s_2 -\delta, s_2)$.
 In Figure \ref{fig:triple_intersection_strategy}, which illustrates the strategy of the argument, the set $C_1$ is colored blue, while $C_2$ is shown in red.
 From $\nu_1+\nu_2 + \nu_3 = 0$ it follows that
 \begin{align*}
  x \cdot \nu_1 + x \cdot \nu_2 + x \cdot \nu_3 = 0.
 \end{align*}
 As a result $\pi_3(C_1) = (-s_1- s_2 - \delta, -s_1 - s_2 + \delta) = \pi_3(C_2)$.
 Let
 \[M_1:= \left\{s \in \pi_3(C_1) : \int_{x\cdot \nu_3 = s} \chi_{K_1^<}(x\cdot \nu_1) \chi_{K_2^>}(x\cdot \nu_2) \intd \mathcal{H}^1(x) > 0 \right\}\]
 and
 \[M_2:= \left\{s \in \pi_3(C_1) : \int_{x\cdot \nu_3 = s} \chi_{[s_1,s_1+\delta]\setminus K_1^>}(x\cdot \nu_1) \chi_{[s_2 -\delta,s_2]\setminus K_2^<}(x\cdot \nu_2) \intd \mathcal{H}^1(x) > 0 \right\}.\]
 By Lemma \ref{lemma:surface_product_sets} we have 
 \[|M_1| \geq |K_1^<| + |K_2^>|\]
 and
 \[|M_2| \geq |[s_1,s_1+\delta]\setminus K_1^>| + |[s_2 -\delta,s_2]\setminus K_2^<| = 2\delta - |K_1^>| - |K_2^<|.\]
 Summing these two inequalities and using the strict inequalities of Step 1.2 we see that
 \begin{align*}
  |M_1| + |M_2| & \geq 2\delta + |K_1^<| - |K_1^>| + |K_2^>| - |K_2^<| > 2\delta = |\pi_3(C_1)|.
 \end{align*}
 As we also have $M_1, M_2 \subset \pi_3(C_1)$ we get that
 \[|M_1 \cap M_2| >0.\]
 By assumption \eqref{intervals_assumption_main} and Fubini's Theorem we have
 \[\int_{M_1 \cap K_3} \int_{\{x\cdot \nu_3 =s\}} \chi_{K_1}(x\cdot \nu_1) \chi_{K_2}(x\cdot \nu_2) \intd \Hd^1(x) \intd s = | \pi_1^{-1}(K_1) \cap \pi_2^{-1}(K_2) \cap \pi_3^{-1}(M_1 \cap K_3) | = 0.\]
 As the inner integral is positive on $M_1$, we must have $|M_1 \cap K_3| = 0$.
 Similarly, we get $|M_2 \cap \stcomp{K_3}| = 0$.
 However, this would imply
 \[0< |M_1\cap M_2| = |M_1\cap M_2 \cap K_3| + |M_1 \cap M_2 \cap \stcomp{K_3}| = 0,\]
 which clearly is a contradiction.
 We thus have either
 \[K_1 = [-r,a_1]\text{ and } K_2 = [-r,a_2]\]
 or
 \[K_1 = [a_1,r]\text{ and } K_2 = [a_2,r]\]
 up to sets of measure zero.
%  Remembering the definition $K_i = J_i \cap [-r,r]$ we get the statement of the lemma for $J_1$ and $J_2$.
 
 \textit{Claim 2: There exists $x_0 \in \ball{0}{r}$ with $x_0 \cdot \nu_1 = a_1$ and $x_0 \cdot\nu_2 = a_2$. Depending on the ``orientation'' of $K_1$ and $K_2$ we either have $J_3 \cap [-r,r] = [-r, x_0 \cdot \nu_3]$ or $J_3\cap [-r,r] = [x_0\cdot \nu_3 ,r]$ up to sets of measure zero.}\\
 Also here Figure \ref{fig:sketch_intervals_fitting} offers in illustration of the argument.
 
 Assumption \eqref{intervals_assumption_non-empty} immediately implies $a_1, a_2 \in \left(-\frac{1}{2}r, \frac{1}{2}r \right)$.
 As $\{\nu_1,\nu_2\}$ is a basis of $\R^2$, see Step 1 in the proof of Proposition \ref{prop: six_corner}, for $r>0$ small enough there exists $x_0 \in \ball{0}{1}$ with $x_0 \cdot \nu_1 = a_1$ and $x_0 \cdot\nu_2 = a_2$.
 This ensures that $J_1$ and $J_2$ have the form advertised in the statement of the Lemma.
 
 Let us assume we are in the case 
 \[K_1 = [-r,a_1]\text{ and } K_2 = [-r,a_2]\]
 up to sets of measure zero, the other case being similar.
 As before we get
 \[\int_{K_3} \int_{\{x\cdot \nu_3 =s\}} \chi_{[-r,a_1]}(x\cdot \nu_1) \chi_{[-r,a_2]}(x\cdot \nu_2) \intd \Hd^1(x) \intd s = | \pi_1^{-1}(K_1) \cap \pi_2^{-1}(K_2) \cap \pi_3^{-1}(K_3) | = 0.\]
 Due to $x\cdot \nu_3 = - x\cdot \nu_1 - x\cdot \nu_2$ we see 
 \[\int_{\{x\cdot \nu_3 =s\}} \chi_{[-r,a_1]}(x\cdot \nu_1) \chi_{[-r,a_2]}(x\cdot \nu_2) \intd \Hd^1(x) > 0\]
 for $s\in (-a_1-a_2,r)$.
 Therefore we get
 $|J_3 \cap [-r,r] \cap [-a_1 - a_2, r] | = | K_3\cap[-a_1 -a_2,  r]|=0$.
 Similarly we can see $|\stcomp{J_3} \cap [-r,r] \cap [-r, -a_1-a_2]| = 0$.
 As a result, we obtain
 \[J_3  \cap [-r,r] = [-r,-a_1-a_2]\]
 up to sets of measure zero.
 Finally, the computation
 \[x_0 \cdot \nu_3 = - x_0 \cdot \nu_1 - x_0 \cdot \nu_2 = -a_1 -a_2 \in (-r,r)\]
 yields the desired statement for $J_3$.
\end{proof}

\begin{proof}[Proof of Lemma \ref{lemma:surface_product_sets}]
 Measurability of
 \[M=\left\{s \in \R : \int_{\{x\cdot\nu_3 = s\}} \chi_{K_1}(x\cdot \nu_1) \chi_{K_2}(x\cdot \nu_2) \intd \Hd^1(x) > 0\right\}\]
 is a consequence of Fubini's theorem.
 By monotonicity of the Lebesgue measure it is sufficient to prove the statement for bounded $K_1$ and $K_2$.
 
 \textit{Step 1: If $t_1$ is a point of density one of $K_1$ and $t_2$ is point of density one of $K_2$, then $-t_1-t_2$ is a point of density one of $M$.}\\ 
 For convenience, we may assume $t_1 = t_2 =0$.
 Let $\pi_i(x) := x\cdot \nu_i$ for $x \in \R^2$ and $i=1,2,3$.
 Let $D_\eps:= \pi_1^{-1}(-\eps,\eps)\cap \pi_2^{-1}(-\eps,\eps)$.
 As, in some transformed coordinates, sets of the form $\pi_1^{-1}(A)\cap \pi_2^{-1}(B)$ are product sets, we can compute
 \begin{align*}
   1- \frac{1}{|D_\eps|} |\pi_1^{-1}(K_1) \cap \pi_2^{-1}(K_2) \cap D_\eps| 
  = &\,\frac{1}{|D_\eps|} \left(|D_\eps| - \left|\pi_1^{-1}(K_1) \cap \pi_2^{-1}(K_2) \cap D_\eps  \right|\right) \\
  = & \, \frac{1}{|D_\eps|} \left|\left(\pi_1^{-1}(\stcomp{K_1})\cap D_\eps\right) \cup \left(\pi_2^{-1}(\stcomp{K_2})\cap D_\eps\right)\right|\\
  \lesssim & \frac{1}{\eps^2}\left(\eps |\stcomp{K_1}\cap(-\eps,\eps)| + \eps |\stcomp{K_2}\cap (-\eps,\eps)|  \right)\\
  = & \, \frac{1}{\eps} (|\stcomp{K_1}\cap(-\eps,\eps)|+ |\stcomp{K_2}\cap(-\eps,\eps)|).
 \end{align*}
 If we take the limit $\eps \to 0$ we see that
 \[\lim_{\eps \to 0} 1- \frac{1}{|D_\eps|} |\pi_1^{-1}(K_1) \cap \pi_2^{-1}(K_2) \cap D_\eps| = 0.\]
 
 \begin{figure}
 \centering
 \setbox8=\hbox{
 \begin{tikzpicture}[scale=1.9]
  \node[above] at (90:.87)  {\phantom{$\eps\nu_1$}};
   \clip (-3,1) -- (1.8,1) -- (1.8,-2) -- (-3,-2);
  \draw[loosely dashed] ($(300:-2) + (210:.4)$)-- ($(300:1.7)+ (210:.4)$)node[right] {$l$};
  \fill[color=white]($(-1,0) + (60:-1)$) -- ($(1,0) + (60:-1)$) -- ($(1,0) + (60:1)$) -- ($(-1,0) + (60:1)$);
  
  \begin{scope}
   \clip ($(-1,0) + (60:-1)$) -- ($(1,0) + (60:-1)$) -- ($(1,0) + (60:1)$) -- ($(-1,0) + (60:1)$) -- cycle;
   \draw[densely dashed] ($(300:-1.3) + (210:.4)$)-- ($(300:1.7)+ (210:.4)$);   
  \end{scope}

  \draw ($(-1,0) + (60:-1)$) -- ($(1,0) + (60:-1)$) -- ($(1,0) + (60:1)$) -- ($(-1,0) + (60:1)$) -- cycle;

  \draw[dotted] ($(300:-2) + (210:1)$)-- ($(300:1.7)+ (210:1)$);
  \draw[dotted] ($(300:-1.3) + (210:-1)$)-- ($(300:1.7)+ (210:-1)$);
  \draw[{Latex[length=2mm]}-{Latex[length=2mm]}] ($(210:1)+(300:1.4)$) -- ($(210:-1)+(300:1.4)$);
  \node at (300:1.25) {$2c\eps$};

  \begin{scope}[shift={(-2.4,-.7)}]
   \draw[->] (0,0) -- (90:.4) node[above] {$\nu_1$};
   \draw[->] (0,0) -- (210:.4) node[left] {$\nu_3$};
   \draw[->] (0,0) -- (330:.4) node[right] {$\nu_2$};
  \end{scope}

 \end{tikzpicture}
 }
  \subcaptionbox{\label{fig:D_a}}{\raisebox{\dimexpr\ht8-\height}{
 \begin{tikzpicture}[scale=1.9]
  \draw ($(-1,0) + (60:-1)$) -- ($(1,0) + (60:-1)$) -- ($(1,0) + (60:1)$) -- ($(-1,0) + (60:1)$) -- cycle;
  \node at (-.3,-.3) {$D_\eps$};
  
  \draw[-{Latex[length=2mm]}] (0,0) -- (90:.87) node[above] {$\eps\nu_1$};
  \draw[-{Latex[length=2mm]}] (0,0) -- (330:.87) node[right] {$\eps\nu_2$};
 \end{tikzpicture}
 }
 }
 \subcaptionbox{\label{fig:D_b}}{
 \begin{tikzpicture}[scale=1.9]
  \node[above] at (90:.87)  {\phantom{$\eps\nu_1$}};
   \clip (-3,1) -- (1.8,1) -- (1.8,-2) -- (-3,-2);
  \draw[loosely dashed] ($(300:-2) + (210:.4)$)-- ($(300:1.7)+ (210:.4)$)node[right] {$l$};
  \fill[color=white]($(-1,0) + (60:-1)$) -- ($(1,0) + (60:-1)$) -- ($(1,0) + (60:1)$) -- ($(-1,0) + (60:1)$);
  
  \begin{scope}
   \clip ($(-1,0) + (60:-1)$) -- ($(1,0) + (60:-1)$) -- ($(1,0) + (60:1)$) -- ($(-1,0) + (60:1)$) -- cycle;
   \draw[densely dashed] ($(300:-1.3) + (210:.4)$)-- ($(300:1.7)+ (210:.4)$);   
  \end{scope}

  \draw ($(-1,0) + (60:-1)$) -- ($(1,0) + (60:-1)$) -- ($(1,0) + (60:1)$) -- ($(-1,0) + (60:1)$) -- cycle;

  \draw[dotted] ($(300:-2) + (210:1)$)-- ($(300:1.7)+ (210:1)$);
  \draw[dotted] ($(300:-1.3) + (210:-1)$)-- ($(300:1.7)+ (210:-1)$);
  \draw[{Latex[length=2mm]}-{Latex[length=2mm]}] ($(210:1)+(300:1.4)$) -- ($(210:-1)+(300:1.4)$);
  \node at (300:1.25) {$2c\eps$};

  \begin{scope}[shift={(-2.4,-.7)}]
   \draw[->] (0,0) -- (90:.4) node[above] {$\nu_1$};
   \draw[->] (0,0) -- (210:.4) node[left] {$\nu_3$};
   \draw[->] (0,0) -- (330:.4) node[right] {$\nu_2$};
  \end{scope}

 \end{tikzpicture}
 }
 \caption{a) Sketch of $D_\eps = \pi_1^{-1}(-\eps,\eps) \cap \pi_2^{-1}(-\eps, \eps)$.
 b) A significant part of the line $l:=\{x\cdot \nu_2 =s\}$ for $s\in(-c\eps, c\eps)$ intersects $D_\eps$.}
 \label{fig:D}
\end{figure}
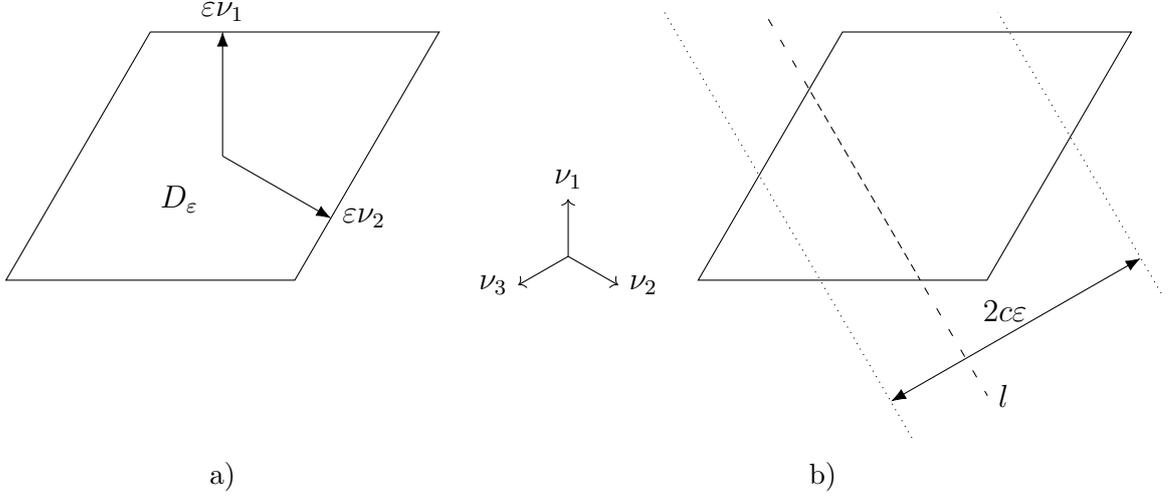

 By scaling arguments there exist $0<c<1$ and $\eta>0$ such that for $s \in (-c\eps,c\eps)$ we have
 \[\int_{\{x\cdot \nu_3 = s\}}\chi_{D_\eps}(x) \intd \Hd^1(x) \geq \eta \eps,\numberthis \label{significant_part_missing}\]
 see Figure \ref{fig:D}.
 Let $S_\eps:= \left\{s\in (-c\eps,c\eps): \int_{\{x\cdot\nu_3 = s\}} \chi_{K_1}(x\cdot \nu_1) \chi_{K_2}(x\cdot \nu_2) \intd \Hd^1(x) = 0 \right\}$, which implies that for $s \in S_\eps$ we also have
 \[\int_{\{x\cdot\nu_3 = s\}} \chi_{K_1\cap(-\eps,\eps)}(x\cdot \nu_1) \chi_{K_2\cap (-\eps,\eps)}(x\cdot \nu_2) \intd \Hd^1(x) = 0.\]
 As for such lines a locally significant part is missing from $\pi_1^{-1}(K_1) \cap \pi_2^{-1}(K_2)$ due to inequality \eqref{significant_part_missing} we get
 \[|\pi_1^{-1}(K_1) \cap \pi_2^{-1}(K_2) \cap D_\eps| \leq |D_\eps| - \eta \eps |S_\eps|.\]
 By algebraic manipulation of this inequality we see
 \begin{align*}
    \frac{|S_\eps|}{2c\eps}
    \leq & \, \frac{1}{2\eta c \eps^2}\left(|D_\eps|-|\pi_1^{-1}(K_1) \cap \pi_2^{-1}(K_2) \cap D_\eps|\right)\\
    \lesssim & \, 1- \frac{1}{|D_\eps|} |\pi_1^{-1}(K_1) \cap \pi_2^{-1}(K_2) \cap D_\eps|.
 \end{align*}
 Since the right-hand side of this inequality vanishes in the limit $\eps \to 0$, we see that $0$ is a point of density one for $M$ by definition of $S_\eps$.
 
 \textit{Step 2: We have $|M| \geq |K_1|+|K_2|$.}\\
 The geometric situation in the following argument can be found in Figure \ref{fig:geometry_surface}.
 Let $\tilde K_i \subset K_i$ for $i=1,2$ be the points of density one contained in the respective sets.
 By Lebesgue point theory we have $|K_i| = |\tilde K_i|$ for $i=1,2$.
 Let $\tilde t_1 := \inf \tilde K_1$ and $\tilde t_2 := \sup \tilde K_2$.
 Since both sets are non-empty and bounded, we have $\tilde t_i \in \R$ for $i=1,2$.
 Let $n\in \N$.
 Let $t_1^{(n)}\in \tilde K_1$ with $0 \leq t_1^{(n)}-\tilde t_1 < \frac{1}{n}$ and let $t_2^{(n)}\in \tilde K_2$ with $0 \leq \tilde t_2- t_2^{(n)} < \frac{1}{n}$.
 Let
 \[M_1^{(n)}:= M \cap (-\infty,-\tilde t_1 - \frac{1}{n} - t_2^{(n)})\]
 and
 \[M_2^{(n)}:= M \cap ( -t_1^{(n)}- \tilde t_2 + \frac{1}{n}, \infty).\]
 Adding the conditions of closeness for $t_i^{(n)}$ we see
 \[ t_1^{(n)}-\tilde t_1 + \tilde t_2- t_2^{(n)} < \frac{2}{n},\]
 which in turn implies
 \[ -\tilde t_1 - t_2^{(n)} -\frac{1}{n}< -t_1^{(n)} - \tilde t_2 + \frac{1}{n}.\]
 Thus $M_1^{(n)}$ and $M_2^{(n)}$ are disjoint and we have
 \begin{align}
  |M| \geq |M_1^{(n)}|+|M_2^{(n)}|.\label{M_lower_est}
 \end{align}
 
 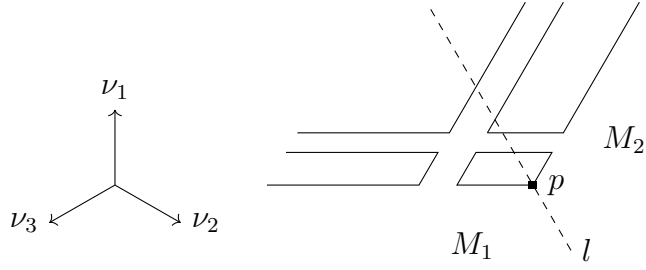
\begin{figure}
 \centering
 \begin{tikzpicture}
  \begin{scope}[shift={(-5,0)}]
   \draw[->] (0,0) -- (90:1) node[above] {$\nu_1$};
   \draw[->] (0,0) -- (210:1) node[left] {$\nu_3$};
   \draw[->] (0,0) -- (330:1) node[right] {$\nu_2$};
  \end{scope}
  
  \draw (-3,0) -- (-1,0) -- ++ (60:.5) -- ++ (-2,0);
  \draw ($(-3,0) + (60:.8)$) -- ++ (2,0) -- ++ (60:2);
  
  \draw (-.5,0) -- (.5,0) -- ++ (60:.5) -- ++ (-1,0) -- cycle;
  \draw ($(-.5,0) + (60:.8)$) -- ++ (1,0) -- ++ (60:2);
  \draw ($(-.5,0) + (60:.8)$) -- ++ (60:2);
  \begin{scope}[shift={(.5,0)}]
    \draw[dashed] (300:1)node[right]{$l$} -- (300:-2.7);
  \end{scope}
  \node at (-.3,-.8) {$M_1$};
  \node at (1.7,.6) {$M_2$};
  
  \node[fill=black,inner sep=1.5pt, label=0:{$p$}] at (.5,0) {};
 \end{tikzpicture}
 \caption{Sketch of $\pi_1^{-1}(K_1) \cap \pi_2^{-1}(K_2)$ with the corner $p := \pi_1^{-1}(\inf K_1) \cap \pi_2^{-1}(\sup K_2)$. Lines parallel to $l:=\{x\cdot \nu_2 = p\cdot \nu_2\}$ intersecting $\pi_1^{-1}(K_1) \cap \pi_2^{-1}(K_2)$ are sorted into $M_1$ if they lie on the left of $l$ or into $M_2$ if they lie on the right.}
 \label{fig:geometry_surface}
\end{figure}
 
 As $t_2^{(n)}$ is a point of density one of $K_1$ and $\tilde K_2$ are points of density one of $K_2$, we know by Step 1 that the set
 \[ -t_2^{(n)} - \tilde K_1 \cap (\tilde t_1 + \frac{1}{n},\infty)\]
 consists of points of density one for $M$.
 We thus know that $|M_1^{(n)}| \geq |\tilde K_1 \cap (\tilde t_1 + \frac{1}{n},\infty)|$.
 Similarly, we obtain $|M_2^{(n)}| \geq |\tilde K_2 \cap (-\infty, \tilde t_2 - \frac{1}{n})|$.
 Combining both inequalities with inequality \eqref{M_lower_est} we see
 \[|M|\geq \left|\tilde K_1 \cap \left(\tilde t_1 + \frac{1}{n},\infty\right)\right| + \left|\tilde K_2 \cap \left(-\infty, \tilde t_2 - \frac{1}{n}\right)\right|. \]
 In the limit $n \to \infty$ we obtain
 \[|M| \geq |\tilde K_1| + |\tilde K_2| = |K_1| + |K_2|.\qedhere\]
\end{proof}

\subsection*{Acknowledgement}
The author thanks his PhD advisor Felix Otto for suggesting the problem and the many helpful discussions.
 
 \emergencystretch=1em
 
 \printbibliography
\end{document}